\documentclass{amsart}

\address{Research and Education Center for Natural Sciences, Keio University, 4-1-1 Hiyoshi, Kohoku-ku, Yokohama, Kanagawa, 223-8521, Japan}
\email{isoshima@keio.jp}

\address{Department of Frontier Media Science, Meiji University, 4-21-1 Nakano, Nakano-ku, Tokyo 164-8525, Japan}
\email{suzukit519@meiji.ac.jp}

\usepackage{amsmath,amssymb}
\usepackage{overpic}
\usepackage{amsfonts}
\usepackage{graphicx}

\usepackage{hyperref}
\usepackage{mathtools}
\DeclarePairedDelimiter{\abs}{\lvert}{\rvert}

\usepackage[hang,small,bf]{caption}
\usepackage[subrefformat=parens]{subcaption}
\usepackage{xcolor}
\usepackage{amsthm}
\usepackage{longtable}
\usepackage{bm}

\theoremstyle{plain}
\newtheorem{thm}{Theorem}[section]
\newtheorem{prop}[thm]{Proposition}
\newtheorem{lem}[thm]{Lemma}
\newtheorem{cor}[thm]{Corollary}

\newtheorem*{thm*}{Theorem}
\newtheorem*{cor*}{Corollary}
\newtheorem*{prop*}{Proposition}

\theoremstyle{definition}
\newtheorem{dfn}[thm]{Definition}
\newtheorem{rem}[thm]{Remark}
\newtheorem{exm}[thm]{Example}
\newtheorem{que}[thm]{Question}

\newtheorem*{que*}{Question}
\newtheorem*{con*}{Conjecture}
\newtheorem*{nota*}{Notation}

\begin{document}

\title{The non-simply connected Price twist for the 4-sphere}

\author{Tsukasa Isoshima}
\author{Tatsumasa Suzuki}
\date{\today}

\begin{abstract}
A cutting and pasting operation on a $P^2$-knot $S$ in a $4$-manifold is called the Price twist. The Price twist for the $4$-sphere $S^4$ yields at most three $4$-manifolds up to diffeomorphism, namely, the $4$-sphere $S^4$, the other homotopy $4$-sphere $\Sigma_{S}(S^4)$ and a non-simply connected $4$-manifold $\tau_{S}(S^4)$. In this paper, we study some properties and diffeomorphism types of $\tau_{S}(S^4)$ for $P^2$-knots $S$ of Kinoshita type. 
\end{abstract}

\maketitle

\section{Introduction}\label{sec:intro}

A surface knot is a closed surface embedded in a 4-manifold. Given a 4-manifold and a surface knot in the 4-manifold, we may change the 4-manifold by a surgery on the surface knot, that is, an operation that cuts a neighborhood of the surface knot and reattaches it. The {\it Gluck twist} is arguably the most familiar operation of this type. 
For a 4-manifold $X$ and a 2-knot $K$ in $X$ with normal Euler number $e(K)=0$, the Gluck twisted 4-manifold $\Sigma_K(X)$ is defined as follows: $\Sigma_K(X)= (X-{\rm int}(N(K))) \cup_{\iota} S^2 \times D^2$, where $N(K)$ is a tubular neighborhood of $K$ and $\iota$ is a self-diffeomorphism of $S^2 \times S^1$ defined by $\iota(z,e^{i\theta})=(ze^{i\theta},e^{i\theta})$. Note that a 2-knot is a surface knot in the case where the surface is the 2-sphere $S^2$. It is known \cite{zbMATH03180037} that the Gluck twisted 4-manifold $\Sigma_K(S^4)$ is a homotopy 4-sphere, and hence it is homeomorphic to $S^4$ by Freedman's theory \cite{zbMATH03838948}. Moreover, there exist some studies showing that $\Sigma_K(S^4)$ is diffeomorphic to $S^4$ for some $K$ (see \cite{zbMATH03180037, key440561m, zbMATH06126098, zbMATH07570610} for example).

We have another surgery, the {\it Price twist}, which is an operation that cuts a neighborhood of a $P^2$-knot and reattaches it. Note that a $P^2$-knot is a surface knot in the case where the surface is the real projective plane $\mathbb{R}P^2$. Price \cite{MR436151} showed that the Price twist for a 4-manifold $X$ and a $P^2$-knot $S$ yields at most three $4$-manifolds up to diffeomorphism, namely, $X$, $\Sigma_{S}(X)$ and $\tau_{S}(X)$. Note that $\Sigma_{S}(X)$ may be diffeomorphic to $X$, but we see that $\tau_S(X)$ is not homotopy equivalent to $X$ since $H_1(\tau_{S}(X)) \not\cong H_1(X)$ by the Mayer-Vietoris exact sequence. For the second Price twist $\Sigma_{S}(X)$, \cite{MR1721575} says that if $S=K \# P_0^{\pm2}$ for a 2-knot $K$ with $e(K)=0$ and the unknotted $P^2$-knot $P_0^{\pm2}$ with $e(P_0^{\pm2})=\pm2$, then $\Sigma_{S}(X)$ is diffeomorphic to the Gluck twisted 4-manifold $\Sigma_{K}(X)$. However, to the best of the authors' knowledge, the third Price twist $\tau_{S}(X)$ has not been studied so far. In this paper, we study some properties and diffeomorphism types of $\tau_{S}(S^4)$ for $P^2$-knots $S$ of Kinoshita type. Note that a $P^2$-knot $S$ in $S^4$ is said to be of {\it Kinoshita type} if $S$ is the connected-sum of a 2-knot $K$ and the unknotted $P^2$-knot $P_0$. It is not yet known whether there exists a $P^2$-knot which is not of Kinoshita type.

In Section \ref{sec:property}, we study some properties of $\tau_{K\#P_0}(S^4)$. We first study a relationship between the Price twist and pochette surgery. 

Let $e_K: P_{1,1} \to X$ be the embedding that the cord is trivial and the $2$-knot $(S_{1,1})_{e_K}$ in $(P_{1,1})_{e_K}$ is equal to $K$ (for details, see Subsection \ref{subsec:pochettesurgery} or \cite[Section 1]{MR4619857}). 

\begin{prop*}[Proposition \ref{prop:pricepochette}]
The Price twist for $S^4$ on a $P^2$-knot of Kinoshita type is a special case of pochette surgery. 
Namely, the Price twists $S^4$, $\Sigma_{K\#P_0}(S^4)$ and $\tau_{K\#P_0}(S^4)$ are diffeomorphic to the pochette surgeries $S^4(e_K,1/0,0)$, $S^4(e_K,1/0,1)$ and $S^4(e_K,2,0)$, respectively. 
\end{prop*}

A pochette surgery is a cutting and pasting operation on the boundary connected sum $S^1 \times D^3 \natural D^2 \times S^2$ embedded in a 4-manifold. For details, see Subsection \ref{subsec:pochettesurgery}. Using this proposition, we have the following. Here, we write $\tau_{K \# P_0}(S^4)$ as $\tau_K$ for short, and $S(M)$ (resp. $\widetilde{S}(M)$) is the 4-manifold obtained by spinning (resp. twist-spinning) a 3-manifold $M$. The lens space of $(p,q)$-type is denoted by $L(p,q)$.

\begin{cor*}[Corollary \ref{cor:homology}]
The integral homology group $H_n(\tau_K)$ of $\tau_{K}$ is 
$$
H_n(\tau_K) \cong
\begin{cases}
   \mathbb{Z}&(n=0, 4), \\
   \mathbb{Z}_{2}&(n=1, 2), \\
   0&(n=3).  
\end{cases}
$$ 
In particular, the Price twist $\tau_K$ is not an integral homology $4$-sphere, but a rational homology $4$-sphere.
\end{cor*}

\begin{prop*}[Proposition \ref{prop:pao}]
For the unknotted $2$-knot $O$ in $S^4$, $\tau_O$ is diffeomorphic to $S(L(2, 1))$. 
\end{prop*}

We next calculate the fundamental group of some $\tau_{K}$. The $(p,q)$-torus knot is denoted by $T_{p,q}$. 
Let $k$ be a knot in $S^3$, $x$ a point of $k$, $B$ the subset $S^3-N(x)$ of $S^3$ and $k_0$ a tangle in $B$. 
We call the $2$-knot $S(k)$ defined by
$$(S^4,S(k))=\partial(B \times D^2, k_0 \times D^2)$$
the \textit{spun knot} of a $1$-knot $k$. 
Note that $S(T_{2,1})$ is the unknotted 2-knot $O$ in $S^4$.

We remark that we can check by handle calculus that $\tau_{S(T_{2,n})}$ is diffeomorphic to $\tau_{S(T_{2,-n})}$.

\begin{thm*}[Theorem \ref{thm:pi1oftaut22n+1}]
The fundamental group $\pi_1(\tau_{S(T_{2,2n+1})})$ is isomorphic to the dihedral group $D_{\abs{2n+1}}$.
\end{thm*}

To the best of the authors' knowledge, this is the first example of a rational homology 4-sphere whose fundamental group is a dihedral group. 

Based on Proposition \ref{prop:pao}, using Theorem \ref{thm:pi1oftaut22n+1}, we compare $\tau_{S(T_{2,2n+1})}$ with $S(M)$, $\widetilde{S}(M)$ and the Pao manifolds that are known as rational homology 4-spheres. 
Note that $S(L(2,1))$, $\widetilde{S}(L(2,1))$ and the Pao manifold $L_2$ (see Figure \ref{fig:Pao}) are diffeomorphic to one another.

\begin{figure}
    \centering
\begin{overpic}[scale=0.6]
{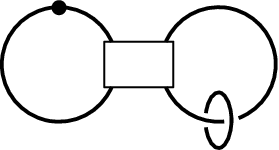}
\put(46,28){$n$}
\put(85,0){$0$}
\put(99,41){$\varepsilon$}
\put(155,32){$\cup$ $3$-handle}
\put(167,19){$4$-handle}
\end{overpic}
\caption{Handle diagrams of the Pao manifolds $L_n$ $(\varepsilon=0)$ and $L_n'$ $(\varepsilon=1)$.}
\label{fig:Pao}
\end{figure}

\begin{cor*}[Corollary \ref{cor:not spin case}]
The Price twists $\tau_{S(T_{2,2n+1})}$ and $\tau_{S(T_{2,2m+1})}$ are not homotopy equivalent to each other if $\abs{2n+1} \neq \abs{2m+1}$. In particular, when $n \neq -1,0$, $\tau_{S(T_{2,2n+1})}$ is homotopy equivalent to neither $S(M)$ nor $\widetilde{S}(M)$ for any closed $3$-manifold $M$.
\end{cor*}

\begin{cor*}[Corollary \ref{cor:tauandPao}]
The Price twist $\tau_{S(T_{2,2n+1})}$ is not homotopy equivalent to any Pao manifold for each $n \neq -1,0$. 
\end{cor*}

We also compare $\tau_K$ with $4$-manifolds $M(p,q,r;\alpha,\beta,\gamma)$ constructed by Iwase (see Subsection \ref{subsec:iwasemfd}) that are also known as rational homology 4-spheres if $\alpha \not= 0$. It is known \cite[Section 6]{MR1091159} that $H_n(\tau_K) \cong H_n(M(p,q,r;\pm2,\beta,\gamma))$.

\begin{cor*}[Corollary \ref{cor:tauandIwasemfd}]
The Price twist $\tau_{S(T_{2,2n+1})}$ is not homotopy equivalent to any Iwase manifold $M(p,q,r;\alpha,\beta,\gamma)$ for each $n \neq -1,0$.  
\end{cor*}

In Section \ref{sec:diffeo type}, we study diffeomorphism types of $\tau_K$ for ribbon 2-knots $K$. We first show the following theorem by handle calculus.

\begin{thm*}[Theorem \ref{thm:double}]
Let $K$ be a ribbon $2$-knot in the $4$-sphere $S^4$. 
Then, the Price twist $\tau_K$ is diffeomorphic to the double $DF(K\#P_0)$ of the $2$-handlebody $F(K\#P_0)$. 
\end{thm*}

Note that a handle diagram of $F(K\#P_0)$ is given in Figure \ref{fig:QuasiExterior}. Using this theorem, we introduce two kinds of handle calculus for $\tau_K$, which we call a {\it deformation} $\alpha$ and a {\it deformation} $\beta$ (Propositions \ref{prop:alpha} and \ref{prop:beta}, respectively). Then, we show the following main theorem by using deformations $\alpha$ and $\beta$.

\begin{figure}[htbp]
    \centering
\begin{overpic}[scale=0.6
]
{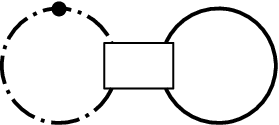}
\put(46,16){$2$}
\put(-10,32){$K$}
\put(97,32){$0$}
\end{overpic}
\caption{A simplified handle diagram of a $2$-handlebody $F(K\#P_0)$. For the definition of this diagram, see Section \ref{sec:diffeo type}. }
\label{fig:QuasiExterior}
\end{figure}

\begin{thm*}[Theorem \ref{thm:1-fusion}]
Let $K$ be a ribbon $2$-knot of $1$-fusion. Then, $\tau_K$ is diffeomorphic to $\tau_{S(T_{2,n})}$, where $n=\det(K)$. 
\end{thm*}

Note that by Theorem \ref{thm:pi1oftaut22n+1} (Corollary \ref{cor:not spin case}), Theorem \ref{thm:1-fusion} classifies the diffeomorphism types of $\tau_K$ completely for ribbon 2-knots $K$ of 1-fusion.

As a corollary of Theorem \ref{thm:1-fusion}, we have especially the following.

\begin{cor*}[Corollary \ref{cor:2-bridge}]
Let $k$ be a $2$-bridge knot. Then, $\tau_{S(k)}$ is diffeomorphic to $\tau_{S(T_{2, n})}$, where $n=\det(k)$. 
\end{cor*}

See Example \ref{exm:2-plat 2-knot} for an example of Theorem \ref{thm:1-fusion}, which is a 2-plat 2-knot.

Let $D(k)$ denote a knot diagram of a ribbon 1-knot $k$ and $R(D(k))$ denote a ribbon 2-knot obtained by taking the double of a ribbon disk properly embedded in $D^4$ that bounds $k$ described by $D(k)$. 

\begin{cor*}[Corollary \ref{cor:1-fusion}]
Let $k$ be a ribbon $1$-knot of $1$-fusion. 
Then, there exists a knot diagram $D(k)$ of $k$ such that $rf(R(D(k))) \le 1$ and $\tau_{R(D(k))}$ is diffeomorphic to $\tau_{S(T_{2,n})}$, where $n=\sqrt{\det(k)}$. 
\end{cor*}

See Subsection \ref{subsec:diffeo type of Price twist} for some concrete examples of Corollary \ref{cor:1-fusion} (Examples \ref{exm:12 cross}, \ref{exm:pretzel} and \ref{exm: 2-bridge ribbon}).

In Example \ref{exm:12 cross}, we deal with ribbon 1-knots up to 12 crossings. Let $k^*$ denote the mirror image of a 1-knot $k$. For a ribbon 1-knot $k$ up to 12 crossings, it is known that the fusion number $rf(k)$ of $k$ except for $12a_{631}$, $12a_{990}$, $12n_{553}$, $12n_{556}$, $3_1\#6_1\#3_1^*$ and $3_1\#3_1\#3_1^*\#3_1^*$ is 1. The fusion numbers $rf(12a_{631})$, $rf(12a_{990})$ and $rf(3_1\#6_1\#3_1^*)$ are less than or equal to 2, and $rf(12n_{553})$, $rf(12n_{556})$ and $rf(3_1\#3_1\#3_1^*\#3_1^*)$ are equal to 2 (see Remark \ref{rem:1-fusion} and Table \ref{tab:12 crossing ribbon 1-knot}).
We also deal with ribbon pretzel knots (Example \ref{exm:pretzel}) and all 2-bride ribbon knots (Example \ref{exm: 2-bridge ribbon}).

\begin{prop*}[Proposition \ref{prop:irregularcase}]
There exist knot diagrams $D(12n_{553})$, $D(12n_{556})$, $D(3_1\#6_1\#3_1^*)$ and $D(3_1 \# 3_1 \# 3_1^* \# 3_1^*)$ such that the Price twists $\tau_{R(D(12n_{553}))}$, $\tau_{R(D(12n_{556}))}$, $\tau_{R(D(3_1\#6_1\#3_1^*))}$ and $\tau_{R(D(3_1 \# 3_1 \# 3_1^* \# 3_1^*))}$ are diffeomorphic to one another. 
\end{prop*}

Note that the fundamental group $\pi_1(\tau_{R(D(k))})$ for any $2$-knot $R(D(k))$ in Proposition \ref{prop:irregularcase} is not isomorphic to $D_{|2n+1|}$ for each integer $n$. 
Thus, we have $rf(R(D(k)))=2$ from Proposition \ref{prop:pao} and Theorems \ref{thm:pi1oftaut22n+1} and \ref{thm:1-fusion}. 
This implies that Proposition \ref{prop:pao} and Theorems \ref{thm:pi1oftaut22n+1} and \ref{thm:1-fusion} provide one approach to proving that the fusion number of a ribbon 2-knot is 2 (see also Remarks \ref{rem:rf=2 example}, \ref{rem:12a990} and \ref{rem:fusion 2 case}). 

It is known \cite[Theorem 1]{zbMATH01117469} that $rf(S(T_{p,q}))=\min\{p,q\}-1$. 
We will show that the fundamental groups of $\tau_{R(D(12n_{553}))}$, $\tau_{R(D(12n_{556}))}$ and $\tau_{R(D(3_1\#6_1\#3_1^*))}$ for knot diagrams $D(10_{99})$, $D(12n_{553})$, $D(12n_{556})$ and $D(3_1\#6_1\#3_1^*)$ are isomorphic to the Coxeter group $W(3,3,\infty)$ (see Remarks \ref{rem:rf=2 example} and \ref{rem:fusion 2 case}(1)). 
We will also show that the fundamental groups of $\tau_{R(D(12a_{427}))}$ for a knot diagram $D(12a_{427})$ is isomorphic to the Coxeter group $W(3,5,\infty)$ (see Remark \ref{rem:fusion 2 case}(2)). 
Note that the dihedral group $D_{\abs{2n+1}}$ that is the fundamental group of $\tau_{S(T_{2,2n+1})}$ is also a Coxeter group.

\begin{que*}[Question \ref{que:coxeter}]
Is the fundamental group of $\tau_{S(T_{p,q})}$, a Coxeter group?
\end{que*}

\begin{que*}[Question \ref{que:n-fusion}]
Let $K$ be a ribbon 2-knot of $n$-fusion for $n \ge 2$. Is $\tau_K$ diffeomorphic to $\tau_{S(T_{n+1,m})}$ for some integer $m \ge n+1$?
\end{que*}

We finally study a double covering of $\tau_{S(T_{2,2n+1})}$. Recall that a Pao manifold is denoted by $L_n$ (see Subsection \ref{subsec:paomfd}). 

\begin{prop*}[Proposition \ref{prop:covering}]
There exists a double cover $\Sigma_2(\tau_{S(T_{2,2n+1})})$ of $\tau_{S(T_{2,2n+1})}$ such that $\Sigma_2(\tau_{S(T_{2,2n+1})})$ is diffeomorphic to $L_{2n+1} \# S^2 \times S^2$. 
\end{prop*}

\section*{Organization}
In Section \ref{sec:preliminaries}, we review precise definitions and properties of the Price twists (Subsection \ref{subsec:pricetwist}), pochette surgery (Subsection \ref{subsec:pochettesurgery}), the spun and twist-spun 4-manifolds (Subsection \ref{subsec:spinmfd}), the Pao manifolds (Subsection \ref{subsec:paomfd}) and the Iwase manifolds (Subsection \ref{subsec:iwasemfd}). 
In Sections \ref{sec:property} and \ref{sec:diffeo type}, we prove the propositions and theorems mentioned in Section \ref{sec:intro}.
In Section \ref{sec:main theorem for pochette}, we rephrase some theorems in Section \ref{sec:diffeo type} in terms of pochette surgery by using the relationship shown in Section \ref{sec:property}. 

\section{Preliminaries}\label{sec:preliminaries}
In this paper, unless otherwise stated, we suppose that every $3$ or $4$-manifold is compact, connected, oriented and smooth, that every surface knot is a closed, connected surface smoothly embbeded in a closed $4$-manifold and that every map is smooth. 

\subsection{Price twist}\label{subsec:pricetwist}
Let $X$ be a closed $4$-manifold and $S$ a $P^2$-knot in $X$ with normal Euler number $e(S)=\pm2$. The \textit{Price twist} is a cutting and pasting operation along $S$. 
The boundary $\partial N(S)$ of a tubular neighborhood $N(S)$ with $e(S)=\pm 2$ is diffeomorphic to the Seifert fibered space $M(S^2; 0, (2, \pm1), (2, \pm1), (2, \mp1))$ in the notation of \cite[Section 4]{csavk2024survey}. 
Hence, the closed 3-manifold $\partial N(S)$ is the quaternion space (i.e. $\partial N(S)$ is diffeomorphic to $S^3/Q$, where $Q$ is the quaternion group) with three exceptional fibers $S_0$, $S_1$ and $S_{-1}$ as in Figure \ref{fig:NSandfibers}. 
Their indices are $\pm2$, $\pm2$ and $\mp2$. 
Let $S_{-1}$ be the fiber with index $\mp2$. Price \cite{MR436151} showed that the Price twisted 4-manifold $(X-\mathrm{int}(N(S))) \cup_{f} N(S)$ yields at most three closed $4$-manifolds up to diffeomorphism, namely, 
\begin{itemize}
\item $X$ if $f(S_{-1})=S_{-1}$, 
\item $\Sigma_{S}(X)$ if $f(S_{-1})=S_1$ and 
\item $\tau_{S}(X)$ if $f(S_{-1})=S_0$, 
\end{itemize}
where $f:\partial N(S) \to \partial(X-\mathrm{int}(N(S)))$ is a diffeomorphism map. 

\begin{figure}[htbp]
    \centering
\begin{overpic}[scale=0.6
]
{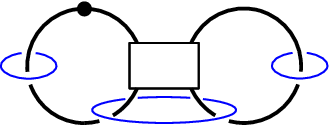}
\put(43,16){$\mp2$}
\put(88,32){$0$}
\put(-12,15){\color{blue}$S_0$}
\put(103,15){\color{blue}$S_{-1}$}
\put(45,-8){\color{blue}$S_1$}
\end{overpic}
\caption{A handle diagram of $N(S)$ and three exceptional fibers $S_0$, $S_1$ and $S_{-1}$ in $\partial{N(S)}$ with normal Euler number $e(S)=\pm2$.}
\label{fig:NSandfibers}
\end{figure}

It is obvious from the Mayer-Vietoris exact sequence that $H_1(\tau_{S}(X)) \not \cong H_1(X)$ (see also \cite{MR4071374, MR1721575}). 
In particular, if $X$ is the $4$-sphere $S^4$, $\tau_{S}(S^4)$ is not simply connected. 
We call the 4-manifold $\tau_{S}(S^4)$ a \textit{non-simply connected Price twist} for $S^4$ along $S$. 

A $P^2$-knot $S$ in $S^4$ is said to be \textit{of Kinoshita type} if $S$ is the connnected sum of a $2$-knot and the unknotted $P^2$-knot $P_0^{\pm2}$ with normal Euler number $\pm2$. 
It is conjectured that every $P^2$-knot in $S^4$ is of Kinoshita type. 
In this paper, we will deal with $P^2$-knots of Kinoshita type. 

A handle diagram of the Price twist is depicted as follows. 
Let a dotted circle with a label $K$ denote the exterior $E(K)$ of a $2$-knot $K$ in $S^4$ as in Figure \ref{fig:EKexterior} (for details, see \cite{MR1721575} for the notation). 
Then, we can depict a handle diagram of $E(K \# P_0^{\pm2})$ as in Figure \ref{fig:exterior}, where $k=n_2-n_1+1$ and $n_i$ is the number of $i$-handles of $E(K)$ ($i=1,2$). 
For example, if $K$ is the spun trefoil knot $S(T_{2,3})$, handle diagrams of $E(K)$ and $E(K\#P_0^{\pm2})$ are shown in Figures \ref{fig:EKexm} and \ref{fig:EKPexm}, respectively.
Handle diagrams of the three Price twisted $4$-manifolds $S^4$, $\Sigma_{K \# P_0^{\pm2}}(S^4)$ and $\tau_{K \# P_0^{\pm2}}(S^4)$ are obtained by adding a $0$-framed unknot to the handle diagram of $E(K\#P_0^{\pm2})$ as in Figure \ref{fig:Pricetwist1}, \ref{fig:Pricetwist2} and \ref{fig:Pricetwist3}, respectively by \cite[Subsection 5.5]{gompf20234}.

\begin{figure}[htbp]
    \centering
\begin{overpic}[scale=0.6
]
{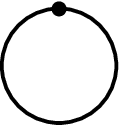}
\put(-20,75){$K$}
\put(115,50){$\cup$ $k$ $3$-handles}
\end{overpic}
\caption{A handle diagram of the exterior $E(K)$ of a 2-knot $K$ in $S^4$. 
}
\label{fig:EKexterior}
\end{figure}

\begin{figure}[htbp]
    \centering
\begin{overpic}[scale=0.6]
{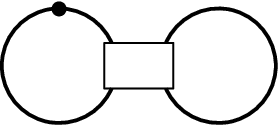}
\put(42,18){$\pm2$}
\put(-10,32){$K$}
\put(97,32){$0$}
\put(115,23){$\cup$ $k$ $3$-handles}
\end{overpic}
\caption{A handle diagram of the exterior $E(K \# P_0^{\pm2})$ of a 2-knot $K$ and the unknotted $P^2$-knots $P_0^{\pm2}$ in $S^4$. 
}
\label{fig:exterior}
\end{figure}

\begin{figure}[htbp]
    \centering
\begin{overpic}[scale=0.6]
{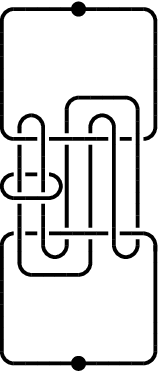}
\put(-8,47){$0$}
\put(41,47){$0$}
\put(85,53){$\cup$ $3$-handle}
\end{overpic}
\caption{A handle diagram of the exterior $E(S(T_{2,3}))$.}
\label{fig:EKexm}
\end{figure}

\begin{figure}[htbp]
    \centering
\begin{overpic}[scale=0.6]
{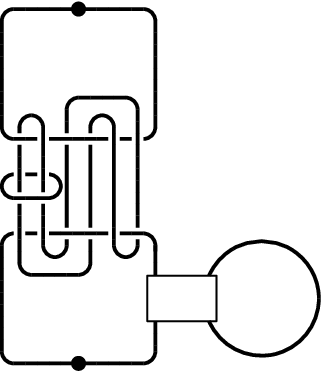}
\put(43,17){$\pm2$}
\put(-8,47){$0$}
\put(41,47){$0$}
\put(82,32){$0$}
\put(105,53){$\cup$ $3$-handle}
\end{overpic}
\caption{A handle diagram of the exterior $E(S(T_{2,3})\#P_0^{\pm2})$.}
\label{fig:EKPexm}
\end{figure}

\begin{figure}[htbp]
    \centering
\begin{overpic}[scale=0.6]
{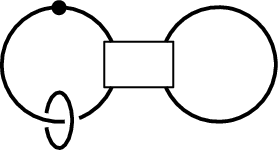}
\put(42,27){$\pm2$}
\put(-10,41){$K$}
\put(28,0){$0$}
\put(97,41){$0$}
\put(115,32){$\cup$ $k+1$ $3$-handles}
\put(159.5,19){$4$-handle}
\end{overpic}
\caption{A handle diagram of the trivial Price twisted $4$-manifold $S^4$. 
}
\label{fig:Pricetwist1}
\end{figure}

\begin{figure}[htbp]
    \centering
\begin{overpic}[scale=0.6]
{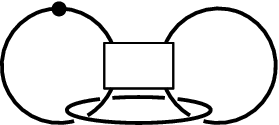}
\put(42,18){$\pm2$}
\put(-10,32){$K$}
\put(44.5,-10){$\pm1$} 
\put(97,32){$0$}
\put(115,23){$\cup$ $k+1$ $3$-handles}
\put(159.5,10){$4$-handle}
\end{overpic}
\caption{A handle diagram of the Price twisted $4$-manifold $\Sigma_{K \# P_0^{\pm2}}(S^4)$. 
}
\label{fig:Pricetwist2}
\end{figure}

\begin{figure}[htbp]
    \centering
\begin{overpic}[scale=0.6]
{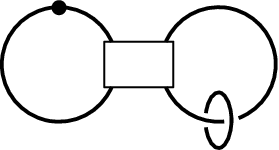}
\put(42,27){$\pm2$}
\put(-10,41){$K$}
\put(85,0){$0$}
\put(97,41){$0$}
\put(115,32){$\cup$ $k+1$ $3$-handles}
\put(159.5,19){$4$-handle}
\end{overpic}
\caption{A handle diagram of the Price twisted $4$-manifold $\tau_{K \# P_0^{\pm2}}(S^4)=\tau_K$. 
}
\label{fig:Pricetwist3}
\end{figure}

\begin{rem}
One can check by handle calculus that the diffeomorphism type of each Price twist for $S^4$ along each $P^2$-knot $S=K\#P_0^{\pm2}$ of Kinoshita type is determined regardless of the normal Euler number of the unknotted $P^2$-knot $P_0^{\pm2}$. 
Thus, in the following, we will consider only the unknotted $P^2$-knot $P_0^{+2}$ with normal Euler number $2$, and write it as $P_0$.
\end{rem}

\begin{nota*}
In this paper, we write $\tau_{K \# P_0}(S^4)$ as $\tau_K$, for short.
\end{nota*}

\begin{rem}
It is known \cite[Theorem 0.1]{MR1721575} that for a 2-knot $K$ in a 4-manifold $X$ with normal Euler number 0 and an unknotted $P^2$-knot $P_0$ with normal Euler number $\pm2$, $\Sigma_{K \# P_0}(X)$ is diffeomorphic to the Gluck twisted 4-manifold for $X$ on $K$.
\end{rem}

For more details, see \cite{MR4071374, MR436151} for example.

\subsection{Pochette surgery}\label{subsec:pochettesurgery}
Let $X$ be a closed $4$-manifold and $E(Y)$ the exterior $X-\mathrm{int}(N(Y))$ of a submanifold $Y$ of $X$, where $N(Y)$ is a tubular neighborhood of $Y$. 
The boundary connected sum $P_{1, 1}:=S^1\times D^3\natural D^2\times S^2$ is called a \textit{pochette}.  
A \textit{pochette surgery} is a cutting and pasting operation along the pochette $P_{1, 1}$. 
Let $e:P_{1, 1} \to X$ be an embedding, $Q_e$ the image $e(Q)$ of a subset $Q$ of $P_{1, 1}$ and $g:\partial{P_{1, 1}}\to\partial E((P_{1, 1})_e)$ a diffeomorphism. 
In the following, we fix an identification $\partial{P_{1, 1}}=\partial E((P_{1, 1})_e)=\#^2 S^1 \times S^2$.

The $4$-manifold $E((P_{1, 1})_e)\cup_{g}P_{1, 1}$ obtained by the pochette surgery on $X$ using $e$ and $g$ is denoted by $X(e, g)$. 
The 4-manifold $X(e, g)$ is also called the pochette surgery on $X$ for $e$ and $g$. 
We call the curves $l:=S^1\times \{*\}$ and $m:=\partial D^2\times \{*\}$ on $\partial P_{1, 1}$ a \textit{longitude} and a \textit{meridian} of $P_{1, 1}$, respectively. 

In the diffeomorphism type of $X(e,g)$, a framing around the knot $g(m)$ of $\partial P_{1,1}=\#^2 S^1 \times S^2$ only affects the parity of its framing coefficient $\varepsilon_0$.
The remainder $\varepsilon$ when the integer $\varepsilon_0$ is divided by $2$ is called a \textit{mod} $2$ \textit{framing}. 
For details on the definition of a mod $2$ framing around $g(m)$, see \cite{MR4619857} or \cite{zbMATH07751599}. 

Let $p$ and $q$ be coprime integers and $g_*:H_1(\partial P_{1,1}) \to H_1(\partial P_{1,1})$ the induced isomorphism of the diffeomorphism $g$. 
By \cite[Section 2]{iwase20044}, the homology class $g_*([m])=p[m]+q[l]$ in the first homology $H_1(\partial P_{1, 1})$ is determined by $p/q \in \mathbb{Q} \cup \{\infty\}$ up to the sign of $p$. 
The following theorem immediately follows from the observations above.
\begin{thm}[{\cite[Theorem 2]{iwase20044}}]
\label{thm:three conditions}
The diffeomorphism type of $X(e, g)$ is determined by the following data:

\def\labelenumi{(\theenumi)}
\begin{enumerate}
  \item An embedding $e:P_{1, 1}\to X$.
  \item A slope $p/q$ of the homology class $g_*([m])=p[m]+q[l]$ in $H_1(\partial P_{1, 1})$. 
  \item A mod $2$ framing $\varepsilon$ around the knot $g(m)$ in $\#^2 S^1 \times S^2$.
\end{enumerate}
\end{thm}

Let $g_{p/q,\varepsilon}:\partial{P_{1, 1}}\to \partial{P_{1, 1}}$ be a diffeomorphism which satisfies $g_{p/q,\varepsilon *}([m])=p[m]+q[l]$ and the mod $2$ framing of $g_{p/q,\varepsilon}(m)$ is $\varepsilon$ in $\{0, 1\}$. 
By Theorem~\ref{thm:three conditions}, we can write $X(e, p/q, \varepsilon)$ as $X(e, g_{p/q, \varepsilon})$. 
In this paper, we define the mod $2$ framing $\varepsilon$ so that the trivial surgery is $X(e, 1/0, 0)$. 
From the construction, any pochette surgery for ($e, 1/0, 1$) is nothing but the Gluck twist along $(S_{1,1})_e$, where $S_{1,1}$ is the subset $\{*\}\times S^2$ of $P_{1, 1}$. 

Let $ST$ be the solid torus $S^1 \times D^2$ and $e_0: S^1 \times ST \to X$ an embedding of $S^1 \times ST$ into a 4-manifold $X$. 
A \textit{torus surgery} (\textit{logarithmic transformation}) on $X$ is an operation that removes the interior of $(S^1 \times ST)_{e_0}$ in $X$ with trivial normal bundle and glues $S^1 \times ST$ by a diffeomorphism $g_0: \partial(S^1 \times ST) \to \partial E((S^1 \times ST)_{e_0})$. 

Fix an identification between $\partial(S^1 \times ST)$ and $\partial E((S^1 \times ST)_{e_0})$. 
The pochette $P_{1,1}$ is diffeomorphic to $S^1 \times ST \cup H$, where $H$ is a $0$-framed $2$-handle attached to $S^1 \times ST$ along $S^1 \times \{*\} \times \{*\}$. 
Fix an identification between $S^1 \times ST \cup H$ and $P_{1, 1}$. 
The curves $\{*\} \times S^1 \times \{*\}$ and $\{*\} \times \{*\} \times \partial D^2$ are nothing but $m$ and $l$ of $P_{1, 1}$, respectively. 
Then, the set $\{[m], [l], [s]\}$ is a basis of $H_1(S^1 \times \partial ST)$, where $s:=S^1 \times \{*\} \times \{*\}$. 

The diffeomorphism type of the torus surgery $E((S^1 \times ST)_{e_0}) \cup_ {g_0} (S^1 \times ST)$ on $X$ is determined by $e_0$ and $(g_0)_{*}([m])=\alpha[m]+\beta[l]+\gamma[s]$ in $H_1(S^1 \times \partial ST)$. 
If $e_0=e|_{S^1 \times ST}$, then we see that a pochette surgery with $e$ and $g$ is a torus surgery with $e_0$ and $g_0$. 
Therefore, any pochette surgery on $X$ is nothing but a torus surgery on $X$. 

For the definition of the linking number for an embedding $e: P_{1,1} \to S^4$, see \cite[Subsection 2D]{MR4619857}. 
In \cite{MR4619857}, the homology groups of the pochette surgery $S^4(e, p/q, \varepsilon)$ are detected. 
\begin{prop}[{\cite[Proposition 2.5]{MR4619857}}]
\label{prop:homology pochette surgery of $4$-sphere}
Let $e: P_{1,1} \to S^4$ be an embedding with linking number $\ell$. 
Then, we have
\begin{enumerate}
\item[(i)] If $p+q\ell\neq 0$, then
$$
H_n(S^4(e,p/q,\varepsilon))\cong
\begin{cases}
    \mathbb{Z} & (n=0, 4), \\
    \mathbb{Z}_{p+q\ell} & (n=1, 2), \\
    0 & (n=3).
\end{cases}
$$
\item[(ii)] If $p+q\ell=0$, then 
$$
H_n(S^4(e,p/q,\varepsilon))\cong
\begin{cases}
    \mathbb{Z} & (n=0, 1, 3, 4), \\
    \mathbb{Z}^2 & (n=2). 
\end{cases}
$$
\end{enumerate}
\end{prop}
Since $p$ and $q$ are coprime, we have that $p+q\ell=0$ if and only if $(p,q)=(\ell,-1), (-\ell,1)$. 

\begin{rem}
In Proposition \ref{prop:homology pochette surgery of $4$-sphere}, the case where $\ell=0$ is first proven in \cite[Theorem 1.1]{okawa}. 
\end{rem}

Consider $P_{1,1}$ as $D^2\times S^2\cup h^1$, where $h^1$ is a $1$-handle. 
The $1$-handle gives a properly embedded, simple arc in $E((S_{1,1})_e)$ by taking the core of $h^1$.
We call this arc a \textit{cord} of the embedding $e:P_{1, 1}\to X$. 
If a cord is boundary parallel, then the cord is said to be \textit{trivial}. 

\begin{rem}\label{rem:link0}
If a cord of an embedding $e:P_{1, 1}\to X$ is trivial, then we can make $\ell=0$.  
For details, see \cite{MR4619857}. 
\end{rem}

\subsection{4-manifolds obtained by spinning 3-manifolds}\label{subsec:spinmfd}
In this subsection, we review closed 4-manifolds $S(M)$ and $\widetilde{S}(M)$.

Let $M$ be a closed $3$-manifold and $B^3$ an open $3$-ball embedded in $M$. 
Then, 4-manifolds $S(M)$ and $\widetilde{S}(M)$ are defined by Plotnick \cite{zbMATH03796838} as follows: 
\begin{eqnarray*}
S(M)&:=&(M-B^3)\times S^1\cup_{\mathrm{id}_{S^2 \times S^1}} S^2 \times D^2, \\
\widetilde{S}(M)&:=&(M-B^3)\times S^1\cup_{\iota} S^2 \times D^2,
\end{eqnarray*}
where $\iota$ is the self-diffeomorphism of $S^2 \times S^1$ defined by $\iota(z, e^{i\theta})=(ze^{i\theta},e^{i\theta})$, which is  not isotopic to the identity $\mathrm{id}_{S^2 \times S^1}$. 
The 4-manifolds $S(M)$ and $\widetilde{S}(M)$ are called the \textit{spin} and \textit{twist-spin} of $M$, respectively. 
The 4-manifolds $S(M)$ and $\widetilde{S}(M)$ are also called the \textit{spun} and \textit{twist-spun $4$-manifold} of $M$, respectively. 

It is known that $\pi_1(S(M)) \cong \pi_1(\widetilde{S}(M)) \cong \pi_1(M)$ and $H_2(S(M)) \cong H_2(\widetilde{S}(M)) \cong H_1(M) \oplus H_2(M)$ \cite{MR922225}. 
Thus, if $M$ is an integral (resp. a rational) homology $3$-sphere, then $S(M)$ and $\widetilde{S}(M)$ are integral (resp. rational) homology $4$-spheres.
It is known \cite{plotnick1986equivariant} that $S(L(p,q))$ is diffeomorphic to $\widetilde{S}(L(p,q))$, where $L(p,q)$ is the lens space of $(p,q)$-type.

\subsection{4-manifolds constructed by Pao}\label{subsec:paomfd} 
Let $N_0$ and $N_1$ be $4$-manifolds diffeomorphic to $D^2 \times T^2$. 
We identify $\partial N_0$ and $\partial N_1$ with $\partial D^2 \times T^2=T^3$ and identify $T^3 = S_1^1 \times S_2^1 \times S_3^1$ with $\mathbb{R}^3/\mathbb{Z}^3$. 
Let $\alpha:\mathrm{GL}(3;\mathbb{Z}) \times \mathbb{R}^3 \to \mathbb{R}^3$ be the action defined by $\alpha(A,\bm{x})=A\bm{x}$, where $A$ is an element of $\mathrm{GL}(3;\mathbb{Z})$. 
We define a self-diffeomorphism $f_A$ of $\partial D^2 \times T^2$ as 
$$f_A([x_1,x_2,x_3])=[(x_1,x_2,x_3){}^t\! A],$$
where $(x_1,x_2,x_3)={}^t\! \bm{x}$. 
Let $m$, $n$, $p$ and $q$ be integers such that $\mathrm{gcd}(m, n)=\mathrm{gcd}(n,p,q)=1$. 
We define an element $A(n;p,q;m)$ of $\mathrm{GL}(3;\mathbb{Z})$ as 
$$
A(n;p,q;m)
=
\begin{pmatrix}
   ma   & mb     & \alpha\\
   na   & nb     & \beta\\
   na+q & nb-p & 0
\end{pmatrix}
,$$
where $a$, $b$, $\alpha$ and $\beta$ are integers such that $ap+bq=1$ and $\alpha n-\beta m=1$. 
Let $c: D^3 \to D^2 \times S_3^1$ be an embedding, $i: S^2 \to \partial c(D^3)$ a diffeomorphism and $h:=i \times \mathrm{id}_{S^1}$. 
Then, we define a closed 4-manifold $L(n;p,q;m)$ as 
$$L(n;p,q;m)=D^2 \times S^2 \cup_h(N_0 - (\mathrm{int}(c(D^3)) \times S^1))\cup_{f_{A(n;p,q;m)}}N_1.$$
We call the closed $4$-manifold $L(n; p, q; m)$ a {\it Pao manifold of type $(n; p, q; m)$}. 
The following classification result exists for the Pao manifolds. 
\begin{thm}[{\cite[Theorem V.1]{MR431231}}]
The Pao manifold $L(n; p, q; m)$ is diffeomorphic to either $L(n; 0, 1; 1)$ or $L(n; 1, 1; 1)$. 
\end{thm}
We write $L(n; 0, 1; 1)$ and $L(n; 1, 1; 1)$ as $L_n$ and $L_n'$, for short, respectively. 
The 4-manifolds $L_n$ and $L_n'$ are diffeomorphic if and only if $n$ is odd and are not homotopy equivalent if $n$ is even \cite[Theorem V.2]{MR431231}. 
It is known \cite{MR431231} that $L_n$ is diffeomorphic to $S(L(n,k))$. 

Handle diagrams of $L_n$ and $L_n'$ are depicted in Figure \ref{fig:Pao} from \cite[Figure 21]{hayano2011genus}. 
We note that $\pi_1(L_n) \cong \pi_1(L_n') \cong \mathbb{Z}_{|n|}$. 

\subsection{4-manifolds constructed by Iwase}\label{subsec:iwasemfd}
In Section \ref{sec:property}, we calculate the homology group of $\tau_{K}$ for any 2-knot $K$. In this subsection, we recall 4-manifolds constructed by Iwase that have the same homology group as that of $\tau_{K}$. 
Iwase \cite{MR941924}, \cite{MR1091159} investigated the diffeomorphism types of 4-manifolds 
obtained by torus surgeries of $S^4$. 

Let $T$ be a submanifold in $S^4$ that is diffeomorphic to a torus $T^2$. 
We call $T$ a $T^2$-\textit{knot}. 
Let $k$ be a 1-knot in $S^3$ and $B^3$ an open $3$-ball embedded in the exterior $E(k)=S^3-\mathrm{int}(N(k))$, where $N(k)$ is a tubular neighborhood of $k$. 
We define $T^2$-knots $T(k)$ and $\widetilde{T}(k)$ as follows: 

\begin{eqnarray*}
(S^4, T(k))&=&((S^3, k)-B^3)\times S^1\cup_{\mathrm{id}_{S^2 \times S^1}}S^2\times D^2, \\
(S^4, \widetilde{T}(k))&=&((S^3, k)-B^3)\times S^1\cup_{\iota}S^2\times D^2. 
\end{eqnarray*}

Let $T_{p, q}$ denote the torus knot of $(p, q)$-type in $S^3$. 
Note that $T_{1,0}$ is the trivial knot $O$. 

\begin{dfn}[{\cite[Definition 2.2]{MR941924}}]
A $T^2$-knot $T$ in $S^4$ is said to be \textit{unknotted} if $T$ bounds a solid torus $S^1\times D^2$ in $S^4$. 
\end{dfn}

Let $T_0$ be the unknotted $T^2$-knot in $S^4$.  

\begin{dfn}[{\cite[Definition 2.4]{MR941924}}]
A $T^2$-knot $T$ in $S^4$ is called a \textit{torus $T^2$-knot} if $T$ is incompressibly embedded in $\partial N(T_0)$. 
\end{dfn}

Note that the torus $T^2$-knots $T(T_{1,0})$ and $\widetilde{T}(T_{1,0})$ are unknotted.
Recall that $m=\{*\} \times S^1 \times \{*\}$, $l=\{*\} \times \{*\} \times \partial D^2$ in $P_{1,1}=S^1 \times ST \cup H$ and $s=S^1 \times \{*\} \times \{*\}$ in $S^1 \times \partial ST=\partial N(T)$. 

\begin{dfn}[{\cite[Definition 3.2]{MR941924}}]
Let $T$ be a torus $T^2$-knot in $S^4$, $i: \partial N(T)\to \partial E(N(T))$ the natural identification and $h: \partial N(T)\to \partial N(T)$ a diffeomorphism such that
$$i\circ h([m]) =\alpha[m]+\beta[l]+\gamma[s]. $$
Then, a 4-manifold $M(T; \alpha, \beta, \gamma)=(S^4-\mathrm{int}(N(T)))\cup_{i\circ h}N(T)$ is called an {\it Iwase manifold of $(\alpha, \beta, \gamma)$-type along $T$}. 
\end{dfn}

For a torus $T^2$-knot $T$, the Iwase manifold $M(T; \alpha, \beta, \gamma)$ is a torus surgery along $T$ for $S^4$. 

Let $M(p, q, 0; \alpha, \beta, \gamma)$ and $M(p, q, q; \alpha, \beta, \gamma)$ be 4-manifolds obtained by the Dehn surgeries of $(\alpha, \beta, \gamma)$-type along $T(T_{p, q})$ and $\widetilde{T}(T_{p, q})$, respectively.  
It is known \cite[Proposition 2.9]{MR941924} that any torus $T^2$-knot is isotopic to either $T(T_{p,q})$, $\widetilde{T}(T_{p,q})$ or the unknotted $T^2$-knot $T(T_{1,0})$. 
Thus, $M(T; \alpha, \beta, \gamma)$ is diffeomorphic to either $M(p,q,0; \alpha, \beta, \gamma)$, $M(p,q,q; \alpha, \beta, \gamma)$ ($1<p<q$, $\mathrm{gcd}(p,q)=1$) or $M(1,0,0; \alpha, \beta, \gamma)$. 

It is known \cite[Section 6]{MR1091159} that for $\alpha \not= 0$ and $r=0, q$, 
$$
H_n(M(p, q, r; \alpha, \beta, \gamma)) \cong 
\begin{cases}
   \mathbb{Z}&(n=0, 4), \\
   \mathbb{Z}_{|\alpha|}&(n=1, 2), \\
   0&(n=3).  
\end{cases}
$$

\cite[Proposition 7.1]{MR1091159} says that $\pi_1(M(p, q, r; \alpha, \beta, \gamma))$ is isomorphic to $\pi_1(M(p, q, r; \alpha, \beta, 0))$ for any $\gamma$. 
Let $S^3_{\alpha/\beta}(T_{p,q})$ be the 3-manifold obtained by the Dehn surgery for $S^3$ of $(\alpha, \beta)$-type along $T_{p, q}$. 
It is known \cite[Theorem 1.3 (iv)]{MR941924} that $M(p, q, r; \alpha, \beta, 0)$ is diffeomorphic to $S(S^3_{\alpha/\beta}(T_{p,q}))$ if $r=0$ and $\widetilde{S}(S^3_{\alpha/\beta}(T_{p,q}))$ if $r=q$. 

Recall that $e_0 : S^1 \times ST \to X$ is an embedding and $g_0 : \partial(S^1 \times ST ) \to \partial E((S^1 \times ST)_{e_0})$ is a diffeomorphism. 
If $e_0=e|_{S^1 \times ST}$, then we see that the pochette surgery $X(e,\alpha/\beta, 0)$ is the torus surgery with $e_0$ and $(g_0)_*([m])=\alpha[m]+\beta[l]$ (see \cite[Section 3]{iwase20044}). 
Thus, $X(e, \alpha/\beta, 0)$ is diffeomorphic to $S(S^3_{\alpha/\beta}(T_{p,q}))$ if $e_0(S^1 \times S^1 \times \{*\})=T(T_{p,q})$ and $\widetilde{S}(S^3_{\alpha/\beta}(T_{p,q}))$ if $e_0(S^1 \times S^1 \times \{*\})=\widetilde{T}(T_{p,q})$. 

\section{Properties of the non-simply connected Price twist for the 4-sphere}\label{sec:property}
In this section, we study some properties of the Price twist $\tau_K$. First, we show a relationship between Price twists and pochette surgeries.

Let $K$ be a $2$-knot in $X$ and $e_K: P_{1,1} \to X$ the embedding that the cord is trivial and the $2$-knot $(S_{1,1})_{e_K}$ in $(P_{1,1})_{e_K}$ is equal to $K$. 

\begin{prop}\label{prop:pricepochette}
The Price twist for $S^4$ on a $P^2$-knot of Kinoshita type is a special case of pochette surgery. 
Namely, the Price twists $S^4$, $\Sigma_{K\#P_0}(S^4)$ and $\tau_{K\#P_0}(S^4)$ are diffeomorphic to the pochette surgeries $S^4(e_K,1/0,0)$, $S^4(e_K,1/0,1)$ and $S^4(e_K,2,0)$, respectively.
\end{prop}

\begin{proof}
Since the pochette surgery for $(e_K, 1/0, 0)$ is nothing but the trivial surgery along the $2$-knot $(S_{1,1})_{e_K}$ in $(P_{1,1})_{e_K}$, the trivial Price twist $S^4$ is diffeomorphic to $S^4(e_K,1/0,0)$. 

Since the Price twist $\Sigma_{K\#P_0}(S^4)$ is diffeomorphic to the Gluck twist along the $2$-knot $K$ with normal Euler number 0 \cite[Theorem 4.1]{MR1721575} and the pochette surgery for $(e_K, 1/0, 1)$ is nothing but the Gluck twist along the $2$-knot $(S_{1,1})_{e_K}$ in $(P_{1,1})_{e_K}$, the Price twist $\Sigma_{K\#P_0}(S^4)$ is diffeomorphic to $S^4(e_K,1/0,1)$. 

If the cord of $e_K$ is trivial, then a handle diagram of $S^4$ can be taken as in Figure \ref{fig: trivial case S4} and the manifold $(P_{1,1})_{e_K}$ consists of the $2$-handle presented by the leftmost $0$-framed unknot, the $3$-handle and the $4$-handle in Figure \ref{fig: trivial case S4}. 

By \cite[Proposition 1]{zbMATH07751599} and the argument of \cite[Subsection 2F]{MR4619857}, we see that a handle diagram of the pochette surgery for $(e_K, 2, 0)$ is shown in Figure \ref{fig:Pricetwist3} from Figure \ref{fig: trivial case S4}. 
Therefore, the Price twist $\tau_{K}$ is diffeomorphic to $S^4(e_K,2,0)$. 
\end{proof}

\begin{figure}
    \centering
\begin{overpic}[scale=0.6]
{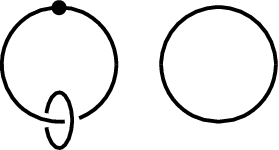}
\put(-8,41){$K$}
\put(27,0){$0$}

\put(97,41){$0$}

\put(115,32){$\cup$ $k+1$ $3$-handles}
\put(159.5,19){$4$-handle}
\end{overpic}
\caption{A handle diagram of the $4$-manifold $S^4$.}
\label{fig: trivial case S4}
\end{figure}

\begin{rem}
In fact, Proposition \ref{prop:pricepochette} can be generalized to any $4$-manifold.
Namely, if $K$ is a $2$-knot with $e(K)=0$ and $P_0^{\pm2}$ is the unknotted $P^2$-knot with $e(P_0^{\pm2})=\pm2$ in a 4-manifold $X$, 
then the Price twists $X$, $\Sigma_{K\#P_0^{\pm2}}(X)$ and $\tau_{K\#P_0^{\pm2}}(X)$ are diffeomorphic to the pochette surgeries $X(e_K,1/0,0)$, $X(e_K,1/0,1)$ and $X(e_K,2,0)$, respectively. 
This follows from the fact that the handle diagrams for $S^4$ in Figures \ref{fig:EKexterior}, \ref{fig:exterior}, \ref{fig:Pricetwist1}, \ref{fig:Pricetwist2} and \ref{fig:Pricetwist3} can also be interpreted as part of handle diagrams for a $4$-manifold $X$, a 2-knot $K$ in $X$ with $e(K)=0$ and the unknotted $P^2$-knot $P_0^{\pm2}$ in $X$ with $e(P_0^{\pm2})=\pm2$. 
\end{rem}

\begin{cor}\label{cor:homology}
The integral homology group $H_n(\tau_K)$ of $\tau_{K}$ is 
$$
H_n(\tau_K) \cong
\begin{cases}
   \mathbb{Z}&(n=0, 4), \\
   \mathbb{Z}_{2}&(n=1, 2), \\
   0&(n=3).  
\end{cases}
$$ 
In particular, the Price twist $\tau_K$ is not an integral homology $4$-sphere, but a rational homology $4$-sphere.
\end{cor}

\begin{proof}
From Proposition \ref{prop:pricepochette}, the Price twist $\tau_K$ is diffeomorphic to the pochette surgery $S^4(e_K,2,0)$. 
From Proposition \ref{prop:homology pochette surgery of $4$-sphere}, we obtain the desired result since the cord of the embedding $e_K$ is trivial and its linking number can be zero by Remark \ref{rem:link0}. 
\end{proof}

We next study a diffeomorphism type of the Price twist $\tau_K$ for the trivial case. The lens space of $(p,q)$-type is denoted by $L(p,q)$.

\begin{prop}\label{prop:pao}
For the unknotted $2$-knot $O$ in $S^4$, $\tau_O$ is diffeomorphic to $S(L(2, 1))$. 
\end{prop}

\begin{proof}
By Proposition \ref{prop:pricepochette}, the Price twist $\tau_O$ is diffeomorphic to $S^4(e_O,2,0)$. 
Using the argument in \cite[Subsection 2F]{MR4619857}, a handle diagram of $S^4(e_O,2,0)$ is shown in Figure \ref{fig:trivial tau}. 

A handle diagram of the closed $4$-manifold $L(2,1) \times S^1$ is depicted in Figure \ref{fig:prespin21} by \cite[Subsections 4.6 and 5.4]{gompf20234}. 
By the definition of $S(M)$, the spin $S(L(2,1))$ is obtained by removing the $3$-handle $D^3 \times D^1$ and the $4$-handle from $L(2,1) \times S^1$, then attaching a $0$-framed meridian to the $1$-handle that appears when we construct $L(2,1) \times S^1$ from $L(2,1)$, and finally gluing the $4$-handle. 
Therefore, a handle diagram of $S(L(2,1))$ is depicted in Figure \ref{fig:spin21}. 
By canceling a $1$-handle/$2$-handle pair, we obtain the handle diagram in Figure \ref{fig:anotherspin21}. 
From handle calculus of \cite[Figure 5.9]{gompf20234}, this diagram is exactly the same as in Figure \ref{fig:trivial tau}. 
This completes the proof. 
\end{proof}

\begin{figure}
    \centering
\begin{overpic}[scale=0.6]
{pao.eps}
\put(46,28){$2$}
\put(85,0){$0$}
\put(99,41){$0$}
\put(135,32){$\cup$ $3$-handle}
\put(147,19){$4$-handle}
\end{overpic}
\caption{A handle diagram of the Price twist $\tau_O$.}
\label{fig:trivial tau}
    \centering
\begin{overpic}[scale=0.6]
{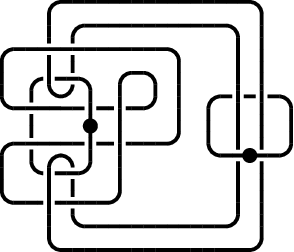}
\put(48,26){$2$}
\put(92,4){$0$}
\put(135,32){$\cup$ $2$ $3$-handles}
\put(157,19){$4$-handle}
\end{overpic}
\caption{A handle diagram of the closed manifold $L(2,1) \times S^1$.}
\label{fig:prespin21}
    \centering
\begin{overpic}[scale=0.6]
{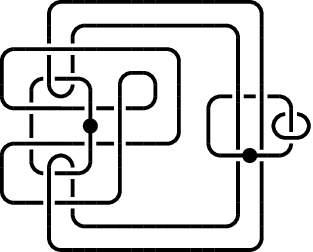}
\put(48,24){$2$}
\put(90,4){$0$}
\put(100,40){$0$}
\put(135,32){$\cup$ $3$-handle}
\put(147,19){$4$-handle}
\end{overpic}
\caption{A handle diagram of the spun 4-manifold $S(L(2,1))$.}
\label{fig:spin21}
    \centering
\begin{overpic}[scale=0.6]
{pao.eps}
\put(46,27.5){$2$}
\put(85,0){$0$}
\put(99,41){$2$}
\put(135,32){$\cup$ $3$-handle}
\put(147,19){$4$-handle}
\end{overpic}
\caption{Another handle diagram of the spun 4-manifold $S(L(2,1))$.}
\label{fig:anotherspin21}
\end{figure}

Note that in the proof of Proposition \ref{prop:pao}, the handle diagram of $\tau_O$ shown in Figure \ref{fig:trivial tau} is constructed via pochette surgery. However, we can also construct the handle diagram of $\tau_O$ directly since $E(O)$ is described by a dotted circle.

In the following, we consider the Price twist $\tau_{S(T_{2,2n+1})}$ along the $P^2$-knot $S(T_{2,2n+1}) \# P_0$ in $S^4$. 
A handle diagram of $\tau_{S(T_{2,2n+1})}$ is depicted in Figure \ref{fig:S(T_{2,2n+1})}. 
In particular, the $\tau$-handle diagrams for $n=1$ and $2$ are drawn in Figures \ref{fig:S(T_{2,3})} and \ref{fig:S(T_{2,5})}, respectively.

\begin{figure}[htbp]
    \centering
\begin{overpic}[scale=0.6
]
{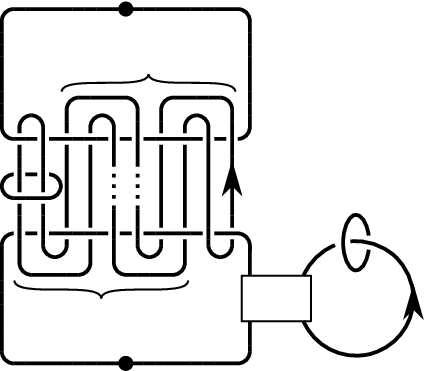}
    \put(-5,42){$0$}
    \put(60,45){$0$}
    \put(84,39){$0$}
    \put(95,26){$0$}
    \put(63,15){$2$}
    \put(32,73){$m$}
    \put(20,11){$m$}

    \put(3,78){$y$}
    \put(52,6){$x$}

    \put(113,49){$\cup$ $2$ $3$-handles}
    \put(128,40){$4$-handle}
    \end{overpic}
\caption{A handle diagram of the Price twist $\tau_{S(T_{2,2n+1})}$, where $m=n$ if $n \ge 0$ and $m=-n-1$ if $n \le -1$. }
\label{fig:S(T_{2,2n+1})}
\end{figure}

\begin{figure}[htbp]
    \centering
\begin{overpic}[scale=0.6
]
{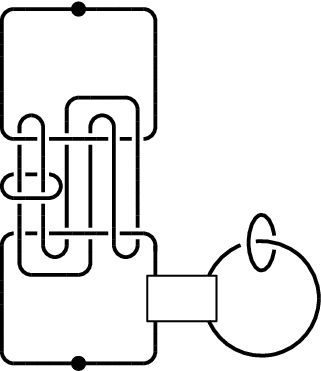}
    \put(-7,47){$0$}
    \put(42,50){$0$}
    \put(68,46){$0$}
    \put(82,31){$0$}
    \put(46,17){$2$}

    \put(113,49){$\cup$ $2$ $3$-handles}
    \put(130,38.5){$4$-handle}
    \end{overpic}
\caption{A handle diagram of the Price twist $\tau_{S(T_{2,3})}$.}
\label{fig:S(T_{2,3})}
\end{figure}

\begin{figure}[htbp]
    \centering
\begin{overpic}[scale=0.6
]
{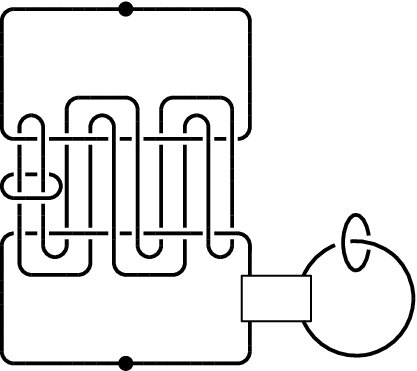}
    \put(-5,42){$0$}
    \put(60,45){$0$}
    \put(84,39){$0$}
    \put(98,26){$0$}
    \put(64,15){$2$}

    \put(113,49){$\cup$ $2$ $3$-handles}
    \put(128,40){$4$-handle}
    \end{overpic}
\caption{A handle diagram of the Price twist $\tau_{S(T_{2,5})}$.}
\label{fig:S(T_{2,5})}
\end{figure}

Let $D_m$ be the dihedral group of order $2m$, where $m$ is any positive integer. 
Recall that $D_m$ has the presentation
$$\langle a,b \mid a^2=1, (ab)^2=1, b^{m}=1\rangle$$
for each positive integer $m$. 
Note that $D_1$ is the finite cyclic group $\mathbb{Z}_2$ of order $2$.
We compute the fundamental group of $\tau_{S(T_{2,2n+1})}$ here. 
From the construction of spun 2-knots, we observe that $S(T_{2,m})$ is isotopic to $S(T_{2,-m})$ for any odd integer $m$. 
Therefore, $\tau_{S(T_{2,m})}$ is diffeomorphic to $\tau_{S(T_{2,-m})}$ for any odd integer $m$. 

We can obtain a presentation of the fundamental group $\pi_1(X)$ from a handle diagram of a 4-manifold $X$.  

In the relations of the presentation of $\pi_1(X)$, we adopt the convention that if a framed knot in this handle diagram passes a dotted circle from top to bottom, it contributes a generator corresponding to the dotted circle, and if the framed knot passes from bottom to top, it contributes the inverse of that generator. 

\begin{thm}\label{thm:pi1oftaut22n+1}
The fundamental group $\pi_1(\tau_{S(T_{2,2n+1})})$ is isomorphic to the dihedral group $D_{\abs{2n+1}}$.
\end{thm}

\begin{proof}
It suffices to show the statement in the case where $n \ge 0$.
By Figure \ref{fig:S(T_{2,2n+1})} and Tietze transformations, we obtain 
\begin{eqnarray*}
\pi_1(\tau_{S(T_{2,2n+1})})&\cong&\langle x,y \mid x^2=1, (yx^{-1})^ny(xy^{-1})^nx=1\rangle\\
&\cong&\langle x,y \mid x^2=1, (yx^{-1})^ny=x^{-1}(yx^{-1})^n\rangle\\
&\cong&\langle x,y \mid x^2=1, x^{-1}(yx^{-1})^n=(yx^{-1})^ny\rangle\\
&\cong&\langle x,y \mid x^2=1, (x^{-1}y)^nx^{-1}=y(x^{-1}y)^n\rangle \\
&\cong&\langle x,y \mid x^2=1, (xy)^nx=x^{-1}(xy)^{n+1}\rangle.
\end{eqnarray*} 
Note that in the third isomorphism, we swap the left- and right-hand sides of the second relation in the second finite presentation from the top.

Let $a=(xy)^nx$ and $b=y^{-1}x^{-1}$. Then, we get $x=b^na$ and $y=a^{-1}b^{-n-1}$. 
Thus, we have
$$\langle x,y \mid x^2=1, (xy)^nx=x^{-1}(xy)^{n+1} \rangle=\langle a,b \mid (b^na)^2=1, a^2b^{2n+1}=1\rangle.$$
Therefore, by Tietze transformations,
\begin{eqnarray*}
&&\langle a,b \mid (b^na)^2=1, a^2b^{2n+1}=1\rangle\\
&=&\langle a,b \mid (b^na)^2=1, (b^na)^2=1, a^2b^{2n+1}=1 \rangle\\
&=&\langle a,b \mid (b^na)^2=1, (ab^n)^2=1, b^{n+1}a=ab^{n} \rangle\\
&=&\langle a,b \mid b^nab^n=a^{-1}, (ab^{n+1})^2=1, b^{n+1}a=ab^{n} \rangle\\
&=&\langle a,b \mid b^nab^n=a^{-1}, ba^{-1}ba=1, b^{n+1}a=ab^{n} \rangle\\
&=&\langle a,b \mid b^nab^n=a^{-1}, bab=a, b^{n+1}a=ab^{n} \rangle\\
&=&\langle a,b \mid a=a^{-1}, bab=a, b^{n+1}a=ab^{n} \rangle\\
&=&\langle a,b \mid a^2=1, bab=a, b^{n+1}a=ab^{n} \rangle\\
&=&\langle a,b \mid a^2=1, ab=b^{-1}a, b^{2n+1}=1\rangle\\
&=&\langle a,b \mid a^2=1, (ab)^2=1, b^{2n+1}=1\rangle\\
&\cong& D_{2n+1}.
\end{eqnarray*} 
This completes the proof. 
\end{proof}

Note that in fact, we can omit the latter Tietze transformations (see \cite[p.11]{MR562913}).

We recalled in Subsection \ref{subsec:spinmfd} that $S(M)$ is a rational homology 4-sphere if $M$ is a rational homology 3-sphere. Moreover, Proposition \ref{prop:pao} says that $\tau_O$ is diffeomorphic to $S(L(2,1))$, and also to $\widetilde{S}(L(2,1))$ (see Subsection \ref{subsec:spinmfd}). 
Thus, it is natural to compare $\tau_K$ with $S(M)$ and $\widetilde{S}(M)$.

\begin{cor}\label{cor:not spin case}
The Price twists $\tau_{S(T_{2,2n+1})}$ and $\tau_{S(T_{2,2m+1})}$ are not homotopy equivalent to each other if $\abs{2n+1} \neq \abs{2m+1}$. In particular, when $n \neq -1,0$, $\tau_{S(T_{2,2n+1})}$ is homotopy equivalent to neither $S(M)$ nor $\widetilde{S}(M)$ for any closed $3$-manifold $M$.
\end{cor}

\begin{proof}
The first claim follows from Theorem \ref{thm:pi1oftaut22n+1} and $|D_{\abs{2n+1}}|=\abs{4n+2}\neq\abs{4m+2}=|D_{\abs{2m+1}}|$ if $\abs{2n+1} \neq \abs{2m+1}$. 

In general, if the fundamental group $\pi_1(M)$ of a closed 3-manifold $M$ is a finite, then the manifold $M$ is diffeomorphic to $S^3/\pi_1(M)$ from the positive solution of Thurston's geometrization conjecture (especially the elliptization conjecture) \cite{arXiv:math/0307245}. 
On the other hand, since the dihedral group $D_{\abs{2n+1}}$ does not act freely on $S^3$ from \cite[Subsection 6.2]{orlik2006seifert}, we have 
$$\pi_1(S(M))\cong\pi_1(\widetilde{S}(M))\cong\pi_1(M) \not\cong D_{\abs{2n+1}} \cong \pi_1(\tau_{S(T_{2,2n+1})})$$
for any closed 3-manifold $M$ and $n \neq -1,0$. 
\end{proof}

We may expect from Proposition \ref{prop:pao} that $\tau_{S(T_{2,2n+1})}$ for $n \neq -1,0$ is also diffeomorphic to the Pao manifold $L_m$ for some $m$ since $S(L(2,1))$ is diffeomorphic to the Pao manifold $L_2$.

\begin{cor}\label{cor:tauandPao}
The Price twist $\tau_{S(T_{2,2n+1})}$ is not homotopy equivalent to any Pao manifold for each $n \neq -1,0$. 
\end{cor}

\begin{proof}
From Theorem \ref{thm:pi1oftaut22n+1} and Figure \ref{fig:Pao}, we obtain 
$$\pi_1(L_k)\cong\pi_1(L_k')\cong\mathbb{Z}_{|k|}\not\cong D_{\abs{2n+1}} \cong \pi_1(\tau_{S(T_{2,2n+1})}).$$
\end{proof}

Note that it also follows from Corollary \ref{cor:not spin case} that $\tau_{S(T_{2,2n+1})}$ is not homotopy equivalent to the Pao manifold $L_p$ for $n \neq -1,0$ since $L_p$ is diffeomorphic to $S(L(p,q))$.

We reviewed in Subsection \ref{subsec:iwasemfd} and Corollary \ref{cor:homology} that $M(p, q, r; \alpha, \beta, \gamma)$ is a rational homology 4-sphere if $\alpha \not= 0$, and the homology group of $M(p, q, r; \pm2, \beta, \gamma)$ is the same as that of $\tau_{K}$.
Thus, it is natural to ask whether they are homotopy equivalent or not. 

\begin{cor}\label{cor:tauandIwasemfd}
The Price twist $\tau_{S(T_{2,2n+1})}$ is not homotopy equivalent to any Iwase manifold $M(p,q,r;\alpha,\beta,\gamma)$ for each $n \neq -1,0$. 
\end{cor}

\begin{proof}
If $\alpha \not= \pm2$, $\tau_{S(T_{2,2n+1})}$ and $M(p,q,r;\alpha,\beta,\gamma)$ are not homotopy equivalent since their homology groups are different by Subsection \ref{subsec:iwasemfd} and Corollary \ref{cor:homology}.

It is known \cite{MR1091159} that $\pi_1(M(p,q,r;\alpha,\beta,\gamma)) \cong \pi_1(M(p,q,r;\alpha,\beta,0))$ for each $\gamma$. Thus, if $\alpha=\pm2$, it suffices to show that $\tau_{S(T_{2,2n+1})}$ is not homotopy equivalent to $M(p,q,r;\pm2,\beta,0)$ for each $n \neq -1,0$. We see from \cite[Theorem 1.3 (iv)]{MR941924} that $M(p,q,r;\pm2,\beta,0)$ is diffeomorphic to $S(M)$ or $\widetilde{S}(M)$ for some closed 3-manifold $M$. However, we show in Corollary \ref{cor:not spin case} that $\tau_{S(T_{2,2n+1})}$ is not homotopy equivalent to $S(M)$ or $\widetilde{S}(M)$ for each $n \neq -1,0$. Hence, $\tau_{S(T_{2,2n+1})}$ is not homotopy equivalent to $M(p,q,r;\pm2,\beta,0)$ for each $n \neq -1,0$.
This completes the proof. 
\end{proof}

Therefore, the Price twist $\tau_{S(T_{2,2n+1})}$ is homotopy equivalent to neither $S(M)$, $\widetilde{S}(M)$, $L_m$, $L_m'$ nor $M(p,q,r;\alpha,\beta,\gamma)$ from Corollaries \ref{cor:not spin case}, \ref{cor:tauandPao} and \ref{cor:tauandIwasemfd} for $n \neq -1,0$.

By considering the above corollaries, the following question naturally arises. 

\begin{que}
Does there exist a $2$-knot $K$ except for the unknotted $2$-knot such that $\tau_K$ is diffeomorphic to $S(M)$, $\widetilde{S}(M)$ for some closed $3$-manifold $M$, a Pao manifold or an Iwase manifold?
\end{que}

\section{Diffeomorphism types of non-simply connected Price twists for the 4-sphere}\label{sec:diffeo type}

In this section, we study the diffeomorphism type of the Price twist $\tau_K$ for some ribbon 2-knots $K$. 

Let $X$ be a $4$-manifold. 
Recall that $N(Y)$ is a tubular neighborhood of a submanifold $Y$ of $X$ and $E(Y)$ is the exterior $X-N(Y)$ of $Y$. 
We use schematic pictures defined as in Figure \ref{fig:schematic pictures} for some handle diagrams. 
\begin{figure}[htbp]
    \centering
\begin{overpic}[scale=0.6
]
{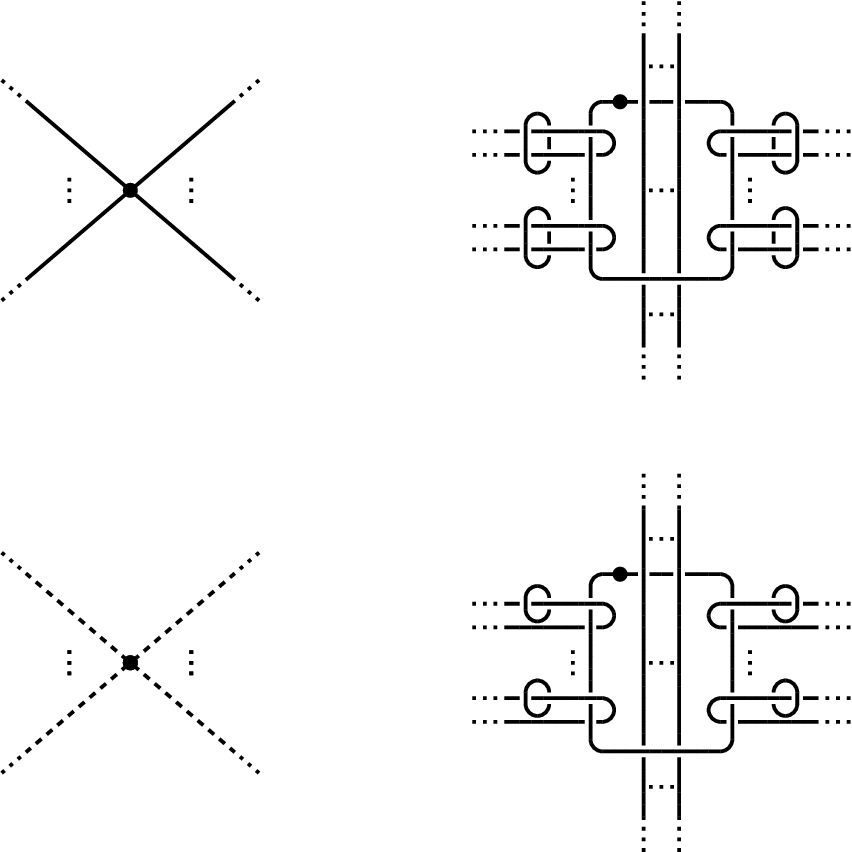}
    \put(40,77){$:=$}
    \put(60,87){$0$}
    \put(57,77){$0$}
    \put(60,76){$0$}
    \put(57,66){$0$}

    \put(94,87){$0$}
    \put(97,77){$0$}
    \put(94,76){$0$}
    \put(97,66){$0$}

    \put(40,21){$:=$}
    \put(60,32){$0$}
    \put(57,22){$0$}
    \put(60,21){$0$}
    \put(57,11){$0$}

    \put(94,32){$0$}
    \put(97,22){$0$}
    \put(94,21){$0$}
    \put(97,11){$0$}

    \end{overpic}
\caption{The definitions of the schematic pictures. A vertex (black dot) corresponds to a dotted circle, and edges are correspond to 2-handles that intertwine with the dotted circle.}
\label{fig:schematic pictures}
\end{figure}

Let $K$ be a ribbon $2$-knot in the $4$-sphere $S^4$. 
It can be seen that a handle diagram of the exterior $E(K)$ can be shown in Figure \ref{fig:egraph} (see \cite[Figure 12.38 (b)]{gompf20234}). 
Note that the number of $3$-handles in Figure \ref{fig:egraph} is the same as that of the edges in Figure \ref{fig:egraph} and the number of $4$-handles in Figure \ref{fig:egraph} is $0$. 
\begin{figure}[htbp]
    \centering
\begin{overpic}[scale=0.6
]
{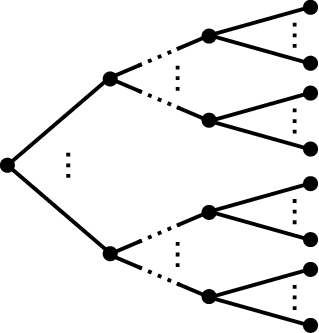}
    \end{overpic}
\caption{A schematic diagram of $E(K)$.}
\label{fig:egraph}
\end{figure}

Let $D(K)$ be a 4-manifold described in the schematic handle diagram depicted in Figure \ref{fig:egraph2}, where the shape of the graph and the numbers of $3,4$-handles are the same as those in Figure \ref{fig:egraph}. 
\begin{figure}[htbp]
    \centering
\begin{overpic}[scale=0.6
]
{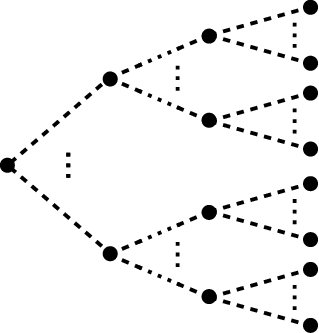}
    \end{overpic}
\caption{A schematic diagram of $D(K)$.}
\label{fig:egraph2}
\end{figure}

\begin{lem}\label{lem:predouble}
Let $K$ be a ribbon $2$-knot in the $4$-sphere $S^4$. 
Then, the exterior $E(K)$ is diffeomorphic to the $4$-manifold $D(K)$. 
\end{lem}

\begin{proof}
By the handle calculus in Figures \ref{fig:total}, \ref{fig:localA} and \ref{fig:localB}, we obtain the handle diagram depicted in Figure \ref{fig:egraph2} from Figure \ref{fig:egraph}. 
\end{proof}

\begin{figure}[htbp]
    \centering
\begin{overpic}[scale=0.6
]
{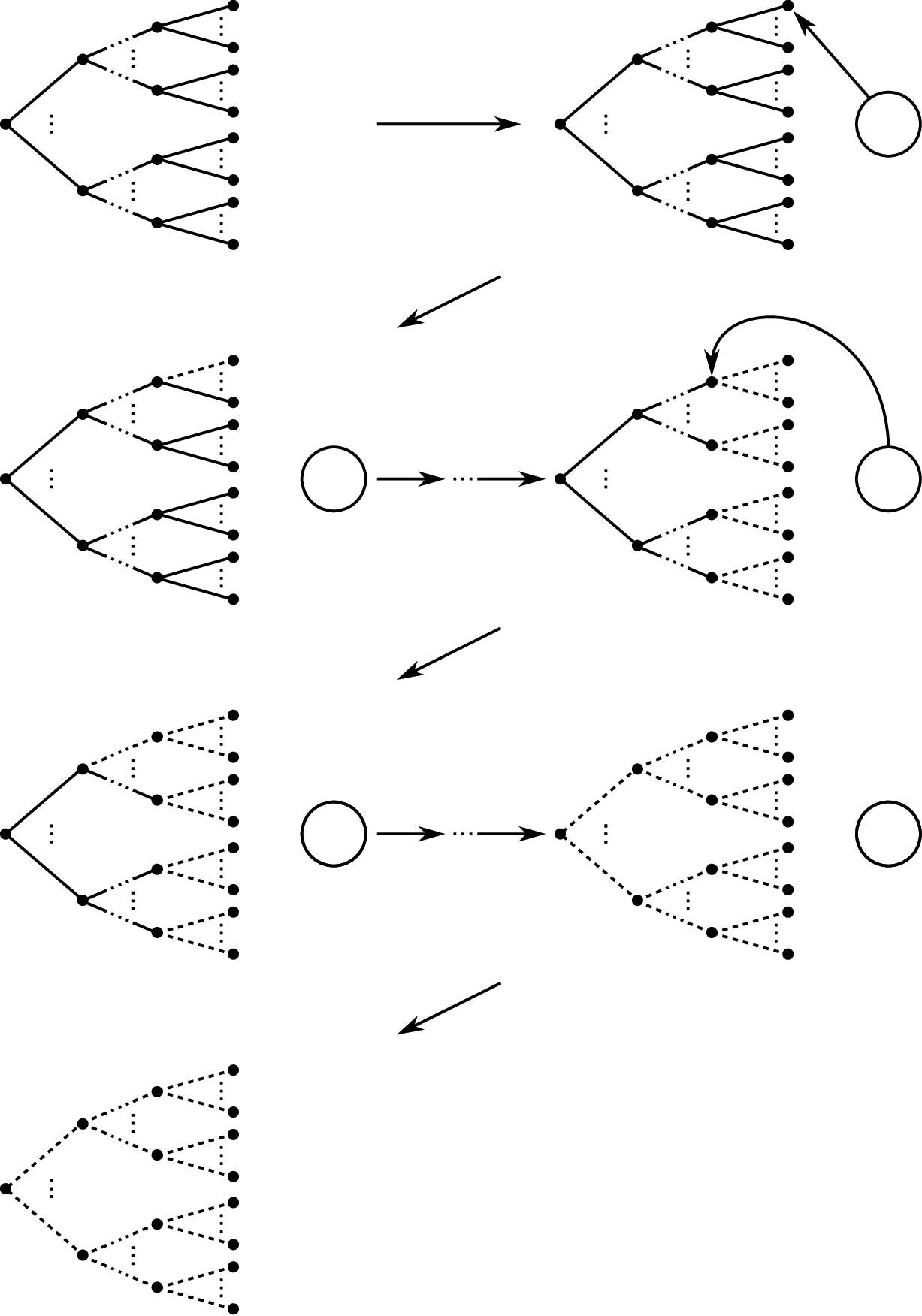}
    \put(31,92){create}
    \put(69,93){$0$}  

    \put(32,79){(A)} 

    \put(27,66){$0$} 
    \put(29,65){(A)} 
    \put(37,65){(A)} 
    \put(69,66){$0$}

    \put(32,52){(B)} 

    \put(27,39){$0$} 
    \put(29,38){(B)} 
    \put(37,38){(B)} 
    \put(69,39){$0$}

    \put(30,25){cancel} 
    \end{overpic}
\caption{A schematic diagram of the proof of Lemma \ref{lem:predouble}. 3-handles are omitted. The first calculus (i.e. the creation) is the creation of a cancelling 2-3 pair.}
\label{fig:total}
\end{figure}

\begin{figure}[htbp]
    \centering
\begin{overpic}[scale=0.6
]
{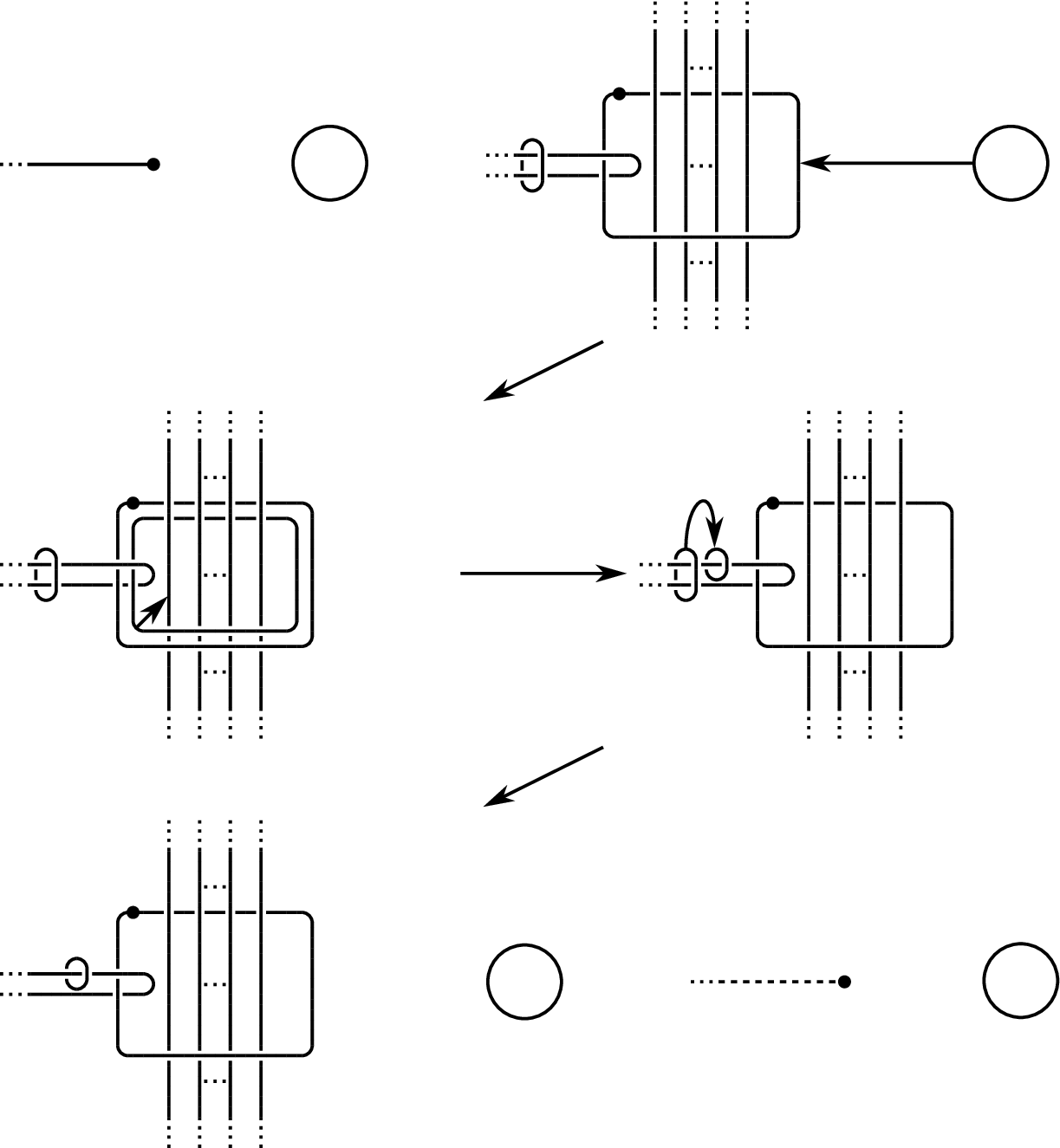}
    \put(32,88){$0$}
    \put(37,85){$=$}
    \put(47,88){$0$}
    \put(49,82){$0$}
    \put(91,88){$0$}

    \put(44,69){slide}

    \put(5,53){$0$}
    \put(7,47){$0$}
    \put(24,52){$0$}
    \put(91,88){$0$}

    \put(44,52){slides}

    \put(59,45){$0$}
    \put(63,47){$0$}
    \put(63,52){$0$}

    \put(43,34){slides}

    \put(7,17){$0$}
    \put(8,11){$0$}
    \put(49,17){$0$}

    \put(54,14){$=$}

    \put(91.5,17){$0$}
    \end{overpic}
\caption{Transformation details (A). In the second calculus (i.e. the second slide), we slide the 0-framed unknot over some 0-framed unknots that intertwine with two lines that describe 2-handles and over the 0-framed meridians.}
\label{fig:localA}
\end{figure}

\begin{figure}[htbp]
    \centering
\begin{overpic}[scale=0.6
]
{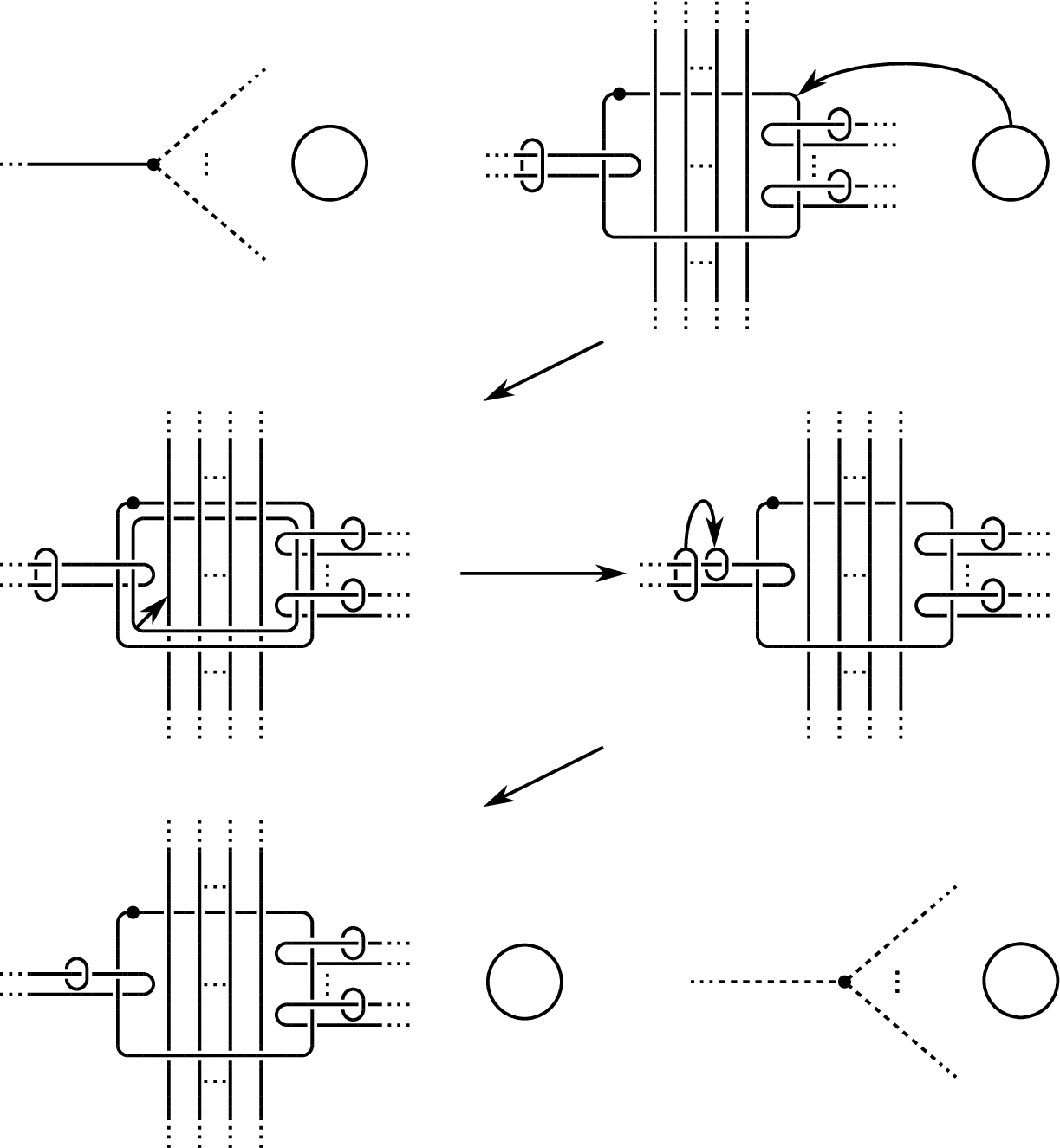}
    \put(32,88){$0$}
    \put(37,85){$=$}
    \put(47,88){$0$}
    \put(49,82){$0$}
    \put(91,88){$0$}

    \put(74,90){$0$}
    \put(76,85){$0$}
    \put(74,84.5){$0$}
    \put(76,80){$0$}

    \put(44,69){slide}

    \put(5,53){$0$}
    \put(7,47){$0$}
    \put(12,52){$0$}
    \put(91,88){$0$}

    \put(32,54){$0$}
    \put(34,49){$0$}
    \put(32,48.5){$0$}
    \put(34,44){$0$}

    \put(44,52){slides}

    \put(59,45){$0$}
    \put(63,47){$0$}
    \put(63,52){$0$}

    \put(88,54){$0$}
    \put(90,49){$0$}
    \put(88,48.5){$0$}
    \put(90,44){$0$}

    \put(43,34){slides}

    \put(7,17){$0$}
    \put(8,11){$0$}
    \put(49,17){$0$}

    \put(32,19){$0$}
    \put(34,13.5){$0$}
    \put(32,13.5){$0$}
    \put(34,8.5){$0$}

    \put(54,14){$=$}

    \put(91.5,17){$0$}

    \end{overpic}
\caption{Transformation details (B). In the second calculus (i.e. the second slide), we slide the 0-framed unknot over some 0-framed unknots that intertwine with two lines that describe 2-handles and over the 0-framed meridians.}
\label{fig:localB}
\end{figure}

Recall that $P_0$ is the unknotted $P^2$-knot in $S^4$. 
Let $DX=X\cup_{\mathrm{id}_{\partial X}}(-X)$ denote the double of $X$. 

Let $F(K)$ denote the $2$-handlebody obtained by removing all the 2-handles that do not intertwine with dotted circles (i.e. all the 0-framed meridians of $D(K)$) and all the 3-handles from $E(K)$. 
We describe the handle diagram of $F(K)$ as in Figure \ref{fig:QuasiExteriorpre}. 
Let $F(K\#P_0)$ denote the 2-handlebody described by the handle diagram in Figure \ref{fig:QuasiExterior}. 
For example, if $K$ is the spun trefoil knot $S(T_{2,3})$, handle diagrams of $F(K)$ and $F(K\#P_0)$ are shown in Figures \ref{fig:FKexm} and \ref{fig:FKPexm}, respectively.

\begin{figure}[htbp]
    \centering
\begin{overpic}[scale=0.6]
{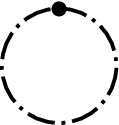}
\put(-20,72){$K$}
\end{overpic}
\caption{A handle diagram of a 2-handlebody $F(K)$.}
\label{fig:QuasiExteriorpre}
\end{figure}

\begin{figure}[htbp]
    \centering
\begin{overpic}[scale=0.6
]
{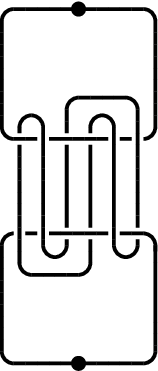}
\put(41,47){$0$}
\end{overpic}
\caption{A handle diagram of a 2-handlebody $F(S(T_{2,3}))$.}
\label{fig:FKexm}
\end{figure}

\begin{figure}[htbp]
    \centering
\begin{overpic}[scale=0.6
]
{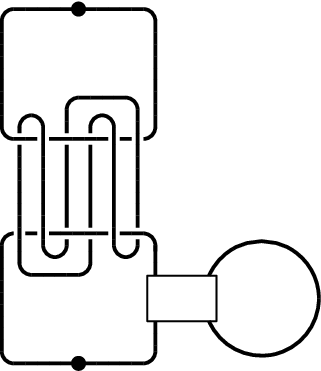}
\put(46,17){$2$}
\put(41,47){$0$}
\put(82,32){$0$}
\end{overpic}
\caption{A handle diagram of a 2-handlebody $F(S(T_{2,3})\#P_0)$.}
\label{fig:FKPexm}
\end{figure}

\begin{thm}\label{thm:double}
Let $K$ be a ribbon $2$-knot in the $4$-sphere $S^4$. 
Then, the Price twist $\tau_K$ is diffeomorphic to the double $DF(K\#P_0)$ of the $2$-handlebody $F(K\#P_0)$. 
\end{thm}

\begin{proof}
Figure \ref{fig:taugraph} (left) is a handle diagram of $\tau_K$. 
Here, the white circle in Figure \ref{fig:taugraph} is assumed to be as shown in the handle diagram in Figure \ref{fig:whitecircle}. 
By Lemma \ref{lem:predouble}, we obtain the handle diagram depicted in Figure \ref{fig:taugraph} (right) from Figure \ref{fig:taugraph} (left) by using handle calculus. 
Figure \ref{fig:taugraph} (right) is a handle diagram of the double of a 2-handlebody with a 0-handle, $n$ 1-handles and $n$ 2-handles, which is just $F(K\#P_0)$, where $n$ is some non-negative integer. 
This completes the proof. 
\end{proof}

\begin{figure}[htbp]
    \centering
\begin{overpic}[scale=0.6]
{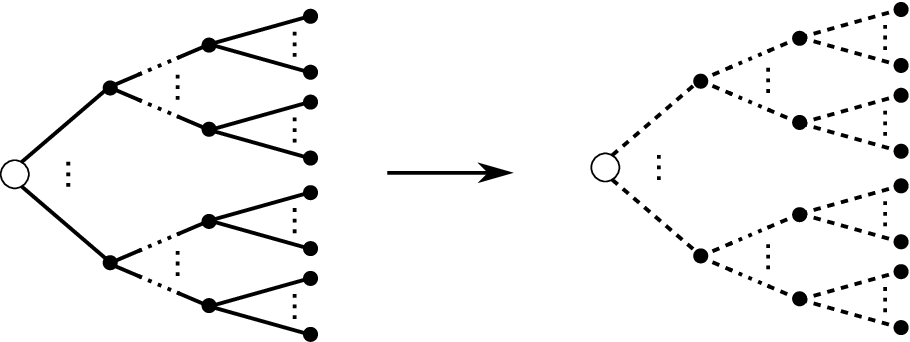}
    \end{overpic}
\caption{A handle diagram for the Price twist $\tau_K$.}
\label{fig:taugraph}
\end{figure}

\begin{figure}[htbp]
    \centering
\begin{overpic}[scale=0.6
]
{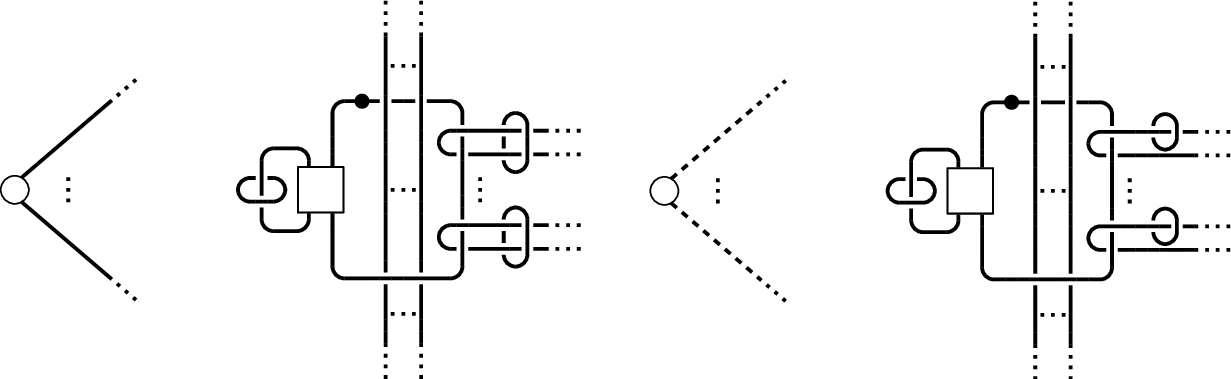}
    \put(12,14.5){$:=$}
    \put(18,17){$0$}
    \put(22,9.5){$0$}
    \put(25.5,14.5){$2$}
    \put(42,7){$0$}
    \put(45,21){$0$}
    \put(42,14.5){$0$}
    \put(45,13){$0$}
    
    \put(50,10){,}

    \put(65,14.5){$:=$}
    \put(71,17){$0$}
    \put(75,9.5){$0$}
    \put(78,14.5){$2$}
    \put(94,22){$0$}
    \put(97,8){$0$}
    \put(94,14.5){$0$}
    \put(97,15.5){$0$}

    \end{overpic}
\caption{Definitions of the white circles.}
\label{fig:whitecircle}
\end{figure}

\begin{rem}
By changing the definition of the white circle in the proof of Theorem \ref{thm:double}, Theorem \ref{thm:double} holds for pochette surgeries for $(e_K, p/q, \varepsilon)$
(see Proposition \ref{prop:pricepochette}).
For details, see Section \ref{sec:main theorem for pochette}. 
\end{rem}

\subsection{A construction method of special handle diagrams}\label{subsec:diagram}

Using Theorem \ref{thm:double}, we introduce two kinds of deformations used in the proofs of the main theorems.

\begin{prop}
\label{prop:alpha}
The handle diagram depicted on the left side of Figure \ref{fig:alpha} is isotopic to the handle diagram on the right side of Figure \ref{fig:alpha}. 
\begin{figure}[htbp]
    \centering
\begin{overpic}[scale=0.6]
{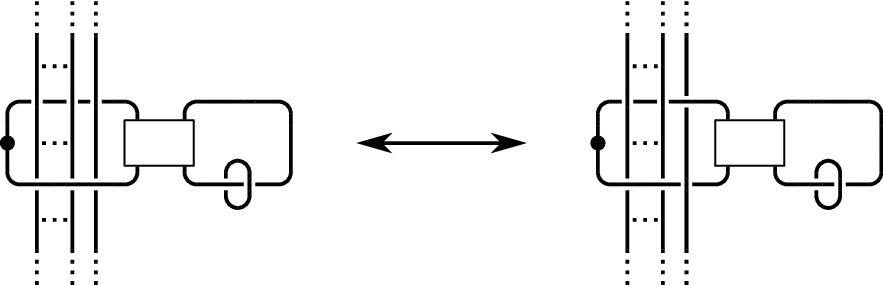}
    \put(33,21){$0$}
    \put(26,5){$0$}
    \put(17,15){$2$}

    \put(100,21){$0$}
    \put(93,5){$0$}
    \put(84,15){$2$}
    \end{overpic}
\caption{A deformation $\alpha$. }
\label{fig:alpha}
\end{figure}
\end{prop}

\begin{proof}
Due to Theorem \ref{thm:double}, we can suppose that each 2-handle in the handle diagram has a 0-framed meridian. Thus, we can perform handle calculus described in Figures \ref{fig:alpha1} and \ref{fig:alpha2}. This completes the proof.
\begin{figure}[htbp]
    \centering
\begin{overpic}[scale=0.6
]
{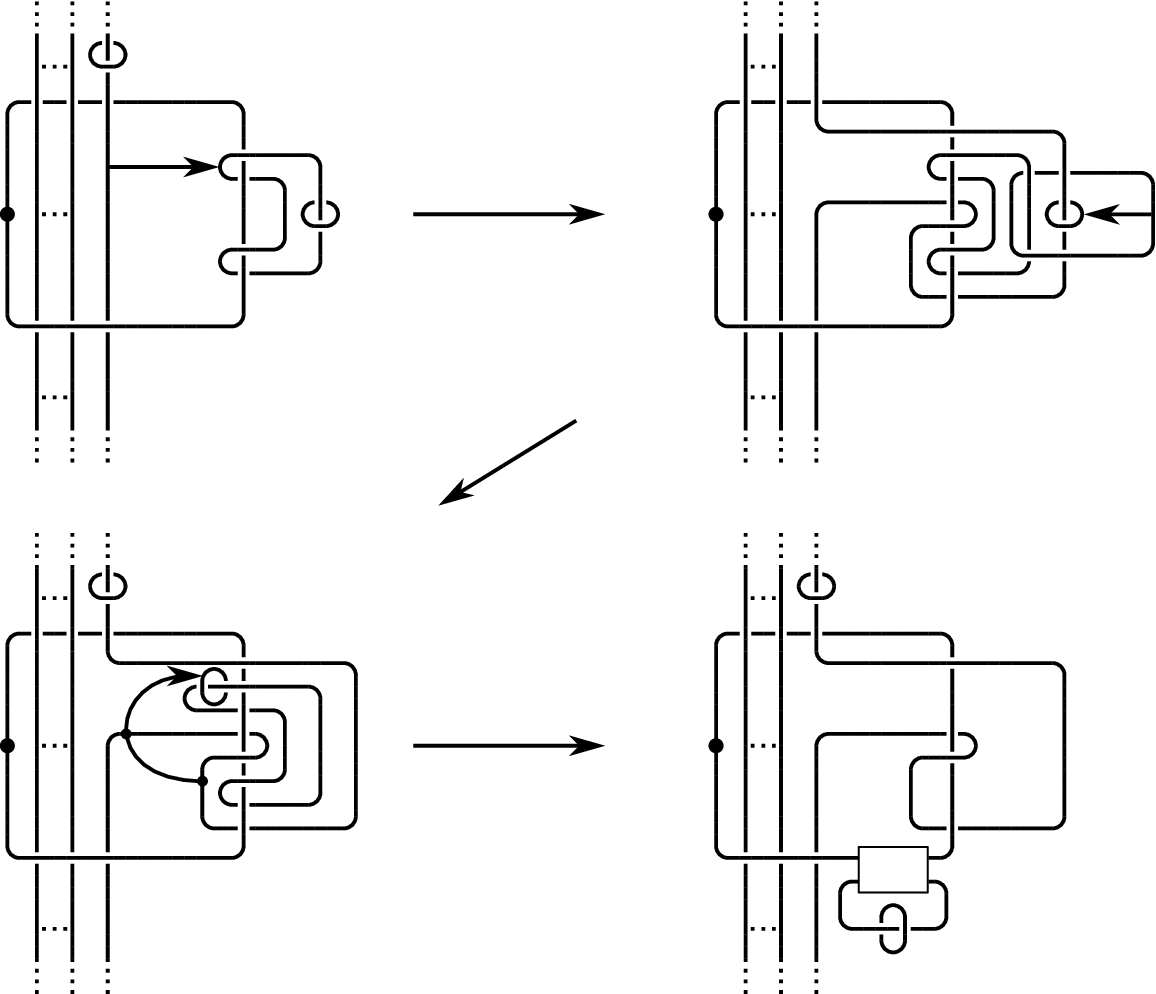}
    \put(12,80){$0$}
    \put(28,72){$0$}
    \put(30,66){$0$}

    \put(40,69){\text{slide}}

    \put(78,71){$0$}
    \put(93,64.5){$0$}
    \put(95,61){$0$}

    \put(40,48){\text{slide}}

    \put(12,34){$0$}
    \put(14,24){$0$}
    \put(28,20){$0$}

    \put(40,23){\text{slides}}

    \put(73,34){$0$}
    \put(76.5,10){$2$}
    \put(83,8){$0$}
    \put(77,1){$0$}
    \end{overpic}
\caption{Handle calculus in the proof for a deformation $\alpha$ (1/2).}
\label{fig:alpha1}
\end{figure}

\begin{figure}[htbp]
    \centering
\begin{overpic}[scale=0.6
]
{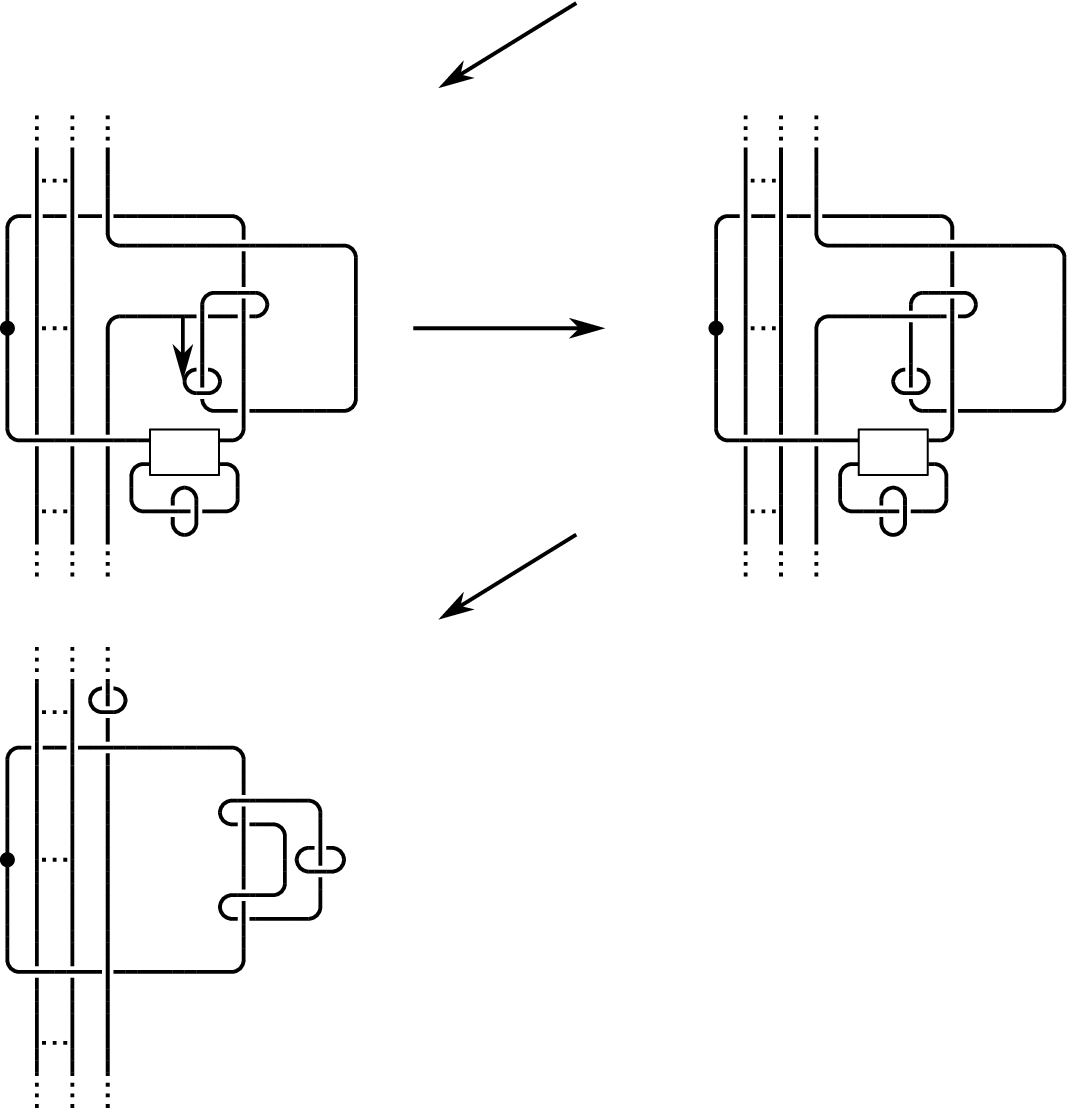}
    \put(34,96){\text{isotopy}}

    \put(14.5,64){$0$}
    \put(16,58){$2$}    
    \put(23,56){$0$}
    \put(18,50){$0$}

    \put(40,72){\text{slide}}

    \put(78.5,64){$0$}
    \put(80,58){$2$}    
    \put(87,56){$0$}
    \put(82,50){$0$}

    \put(34,49){\text{isotopy}}

    \put(12,35){$0$}
    \put(29,28){$0$}
    \put(32,20){$0$}
    \end{overpic}
\caption{Handle calculus in the proof for a deformation $\alpha$ (2/2).}
\label{fig:alpha2}
\end{figure}
\end{proof}

Let  
$$\langle a, {\bf b} \mid a^2=1, uav=1, {\bf w}={\bf 1} \rangle$$
be the presentation of $\pi_1(\tau_K)$ for the handle decomposition of $\tau_K$ corresponding to the handle diagram in the left side of Figure \ref{fig:alpha}. 
Here ${\bf b}$ is a generating subset $\{b_1,\ldots,b_n\}$, $u$ and $v$ are words in the generating set $\{a, {\bf b}\}$, and ${\bf w}={\bf 1}$ is a set of relations $\{w_1=1,\ldots,w_{n-1}=1\}$. 

\begin{cor}
\label{cor:alpha}
The following Tietze transformation on finitely presented groups of the fundamental group $\pi_1(\tau_K)$ does not change the diffeomorphism type of $\tau_K$: 
\[
\langle a, {\bf b} \mid a^2=1, uav=1, {\bf w}={\bf 1} \rangle=\langle a, {\bf b} \mid a^2=1, ua^{-1}v=1, {\bf w}={\bf 1} \rangle
\]
\end{cor}
\begin{proof}
This claim is obtained from Proposition \ref{prop:alpha}. 
\end{proof}

We call the deformation in Proposition \ref{prop:alpha} or Corollary \ref{cor:alpha} a \textit{deformation} $\alpha$. 

\begin{prop}
\label{prop:beta}
The handle diagram on the left side of Figure \ref{fig:beta} is isotopic to the handle diagram on the right side of Figure \ref{fig:beta}. 
\end{prop}
\begin{figure}[htbp]
    \centering
\begin{overpic}[scale=0.6
]
{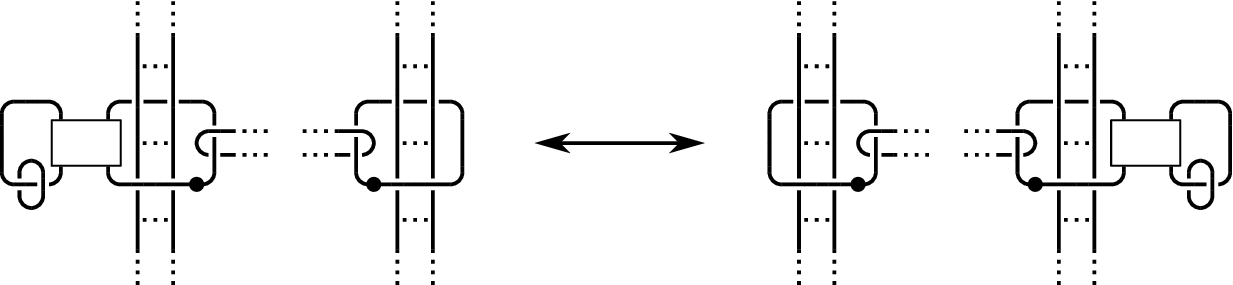}
    \put(-2,12){$0$}
    \put(0,5){$0$}
    \put(6,10.5){$2$}
    \put(18,7.5){$0$}

    \put(100.5,12){$0$}
    \put(94.75,5){$0$}
    \put(92,10.5){$2$}
    \put(72,7.5){$0$}
    \end{overpic}
\caption{A deformation $\beta$. }
\label{fig:beta}
\end{figure}

\begin{proof}
Due to the deformation $\alpha$, the 0-framed knot between the two dotted circles in Figure \ref{fig:beta} can always be deformed so that it has no twist (i.e. it is parallel to the plane of the paper). Then, handle calculus described in Figure \ref{fig:betaproof} complete the proof.
\end{proof}

\begin{figure}[htbp]
    \centering
\begin{overpic}[scale=0.6
]
{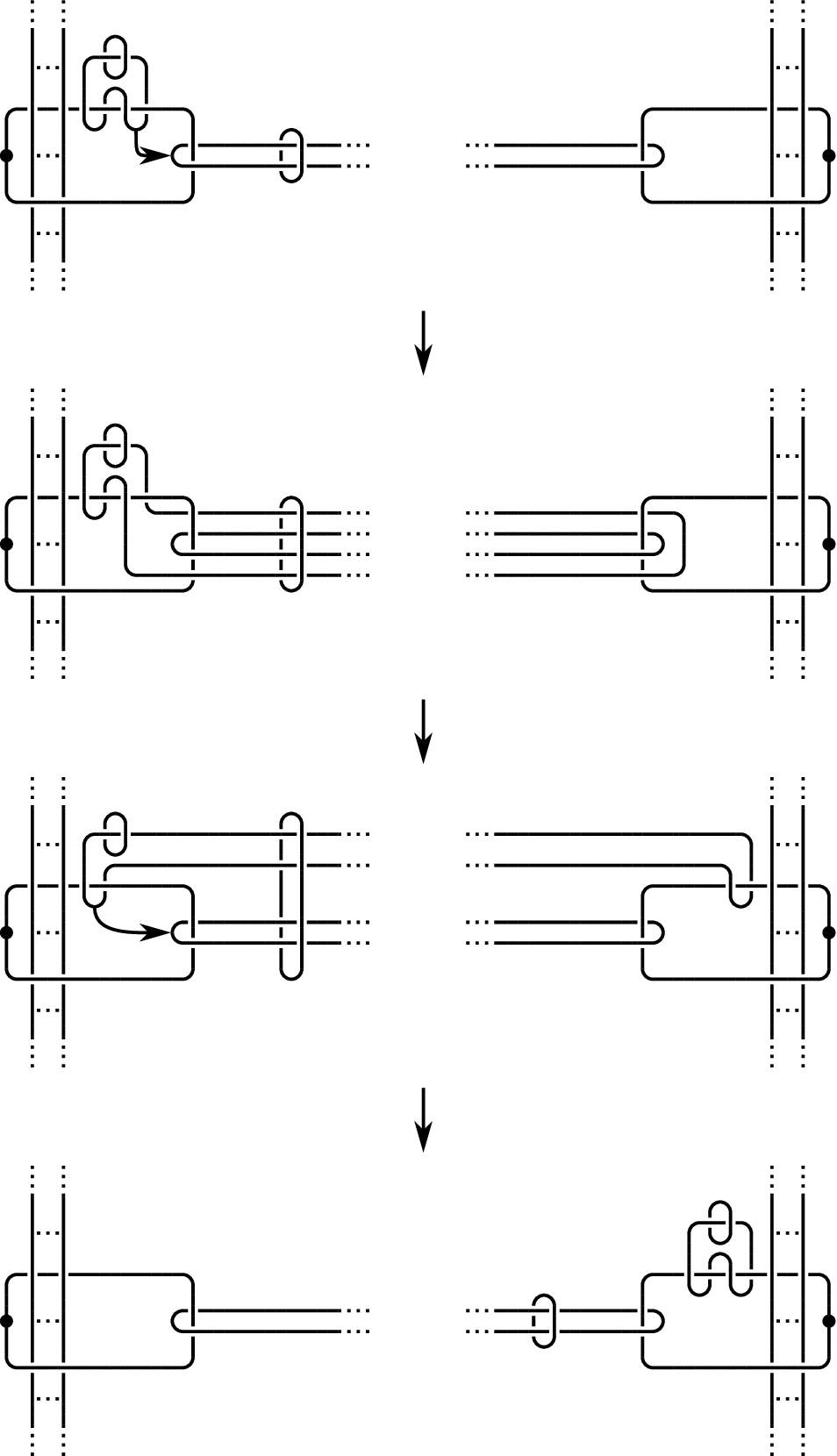}
    \put(9,97){$0$}
    \put(10.5,93.5){$0$}
    \put(20,91.5){$0$}
    \put(14,86.5){$0$}

    \put(31,76){slide}

    \put(9,70.5){$0$}
    \put(10.5,67){$0$}
    \put(20,66.5){$0$}
    \put(10.5,61.5){$0$}

    \put(31,49){isotopy}

    \put(9,43.5){$0$}
    \put(12.5,43){$0$}
    \put(20,45){$0$}
    \put(11,33.5){$0$}

    \put(31,23){slide}

    \put(47.5,17){$0$}
    \put(45.5,14){$0$}
    \put(42,6.5){$0$}
    \put(37.5,11.5){$0$}
\end{overpic}
\caption{Handle calculus in the proof for a deformation $\beta$.}
\label{fig:betaproof}
\end{figure}

Note that Proposition \ref{prop:beta} actually can be shown without using Theorem \ref{thm:double} since the 2-handle that links twice with the dotted circle has the 0-framed meridian.

Let  
$$\langle a, b,{\bf c} \mid a^2=1, a=ub^{\pm1}u^{-1}, {\bf w}={\bf 1}\rangle$$
be the presentation of $\pi_1(\tau_K)$ for the handle decomposition of $\tau_K$ corresponding to the handle diagram in the left side of Figure \ref{fig:beta}. 
Here ${\bf c}$ is a generating subset $\{c_1,\ldots,c_n\}$, $u$ is a word in the generating set $\{a, b,{\bf c}\}$, and ${\bf w}={\bf 1}$ is a set of relations $\{w_1=1,\ldots,w_n=1\}$. 
\begin{cor}
\label{cor:beta}
The following Tietze transformation on finitely presented groups of the fundamental group $\pi_1(\tau_K)$ does not change the diffeomorphism type of $\tau_K$: 
\begin{eqnarray*}
&&\langle a, b,{\bf c} \mid a^2=1, a=ub^{\pm1}u^{-1}, {\bf w}={\bf 1}\rangle\\
&=&\langle a, b,{\bf c} \mid b^2=1, a=ub^{\pm1}u^{-1}, {\bf w}={\bf 1}\rangle. 
\end{eqnarray*}
\end{cor}

\begin{proof}
This claim is obtained from Propositions \ref{prop:alpha} and \ref{prop:beta}. 
\end{proof}

We call the deformation in Proposition \ref{prop:beta} or Corollary \ref{cor:beta} a \textit{deformation} $\beta$. 

\begin{rem}\label{rem:deformation}
We cannot apply the deformation $\alpha$ to pochette surgery with slopes other than $\pm2$. 
On the other hand, we can apply the deformation $\beta$ to all pochette surgeries including the Price twist $\tau_K$ (see Proposition \ref{prop:pricepochette}).
\end{rem}

By Theorem \ref{thm:double}, we can construct a handle diagram of $\tau_K$ so that the framing coefficient of each framed knot entangled with some dotted circles is $0$, and each such knot has exactly one $0$-framed meridian.
Furthermore, the handle diagram of $\tau_K$ shown in Figure \ref{fig:taugraph} is a tree.
Therefore, we can construct a handle diagram of $\tau_K$ so that any two elements of the generating set in a finite presentation of $\pi_1(\tau_K)$ obtained from the handle diagram are conjugate to each other via some relations in the presentation.
Hence, by the deformations $\alpha$ and $\beta$, we may assume that the $0$-framed knot entangled exactly twice with some dotted circle is entangled exactly twice with any one of the dotted circles.
Since the deformation $\alpha$ can be applied to any dotted circle due to this assumption, we may suppose that each $0$-framed knot entangled between any two dotted circles has no twist.
From the above, we see that the diffeomorphism type of $\tau_K$ can be determined even if the handle diagram of $\tau_K$ shown in Figure \ref{fig:tauHD} is abbreviated as in Figure \ref{fig:tauSHD}.
We call such a simplified handle diagram a \textit{$\tau$-handle diagram}.
Note that, just as in ordinary handle diagrams, 3- and 4-handles can be omitted.
By Theorem \ref{thm:double}, over/under crossings in each framed knot entangled with a dotted circle in a handle diagram of $\tau_K$ can be modified arbitrarily by its $0$-framed meridian.
We remark that changing over/under crossings in a $\tau$-handle diagram does not affect the diffeomorphism type of $\tau_K$ represented by the $\tau$-handle diagram.

\begin{figure}[htbp]
    \centering
\begin{overpic}[scale=0.6]
{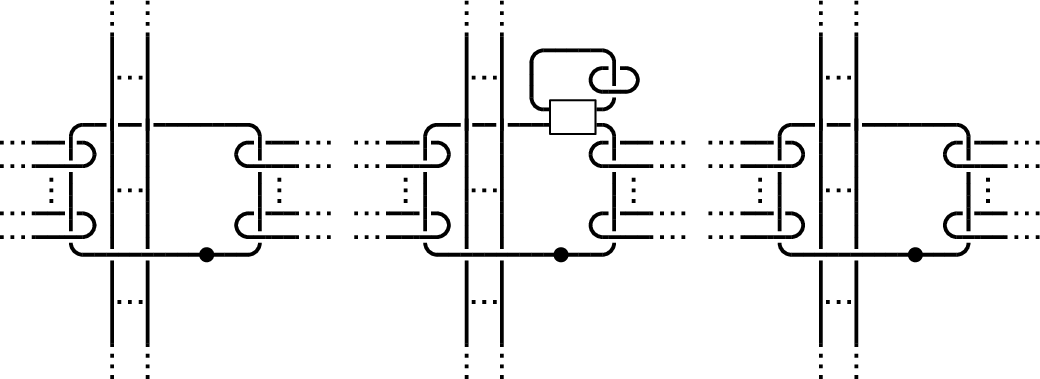}
    \put(54,24.2){$2$}
    \put(55,32.5){$0$}
    \put(62,28){$0$}
    \end{overpic}
\caption{A handle diagram of $\tau_{K}$. The framing coefficient of each 2-handle is 0, and each 2-handle has the 0-framed meridian.}
\label{fig:tauHD}
\end{figure}

\begin{figure}
    \centering
\begin{overpic}[scale=0.6]
{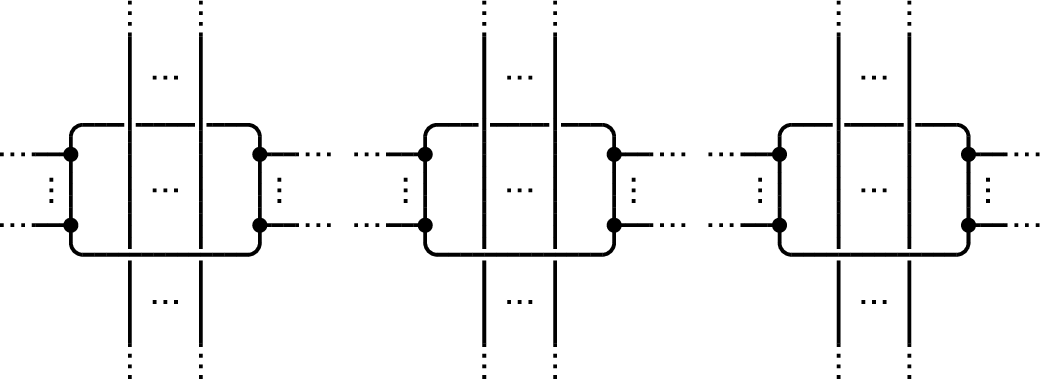}
    \end{overpic}
\caption{A simplified handle diagram of $\tau_{K}$. We call such a diagram a $\tau$-handle diagram.}
\label{fig:tauSHD}
\end{figure}

\begin{exm}
From the handle diagram of $\tau_{S(T_{2,2n+1})}$ depicted in Figure \ref{fig:S(T_{2,2n+1})} or the left side of Figure \ref{fig:S(T_{2,2n+1})second}, we obtain a $\tau$-handle diagram of $\tau_{S(T_{2,2n+1})}$ depicted in the right side of Figure \ref{fig:S(T_{2,2n+1})second}.  
\end{exm}

\begin{figure}[htbp]
    \centering
\begin{overpic}[scale=0.6]
{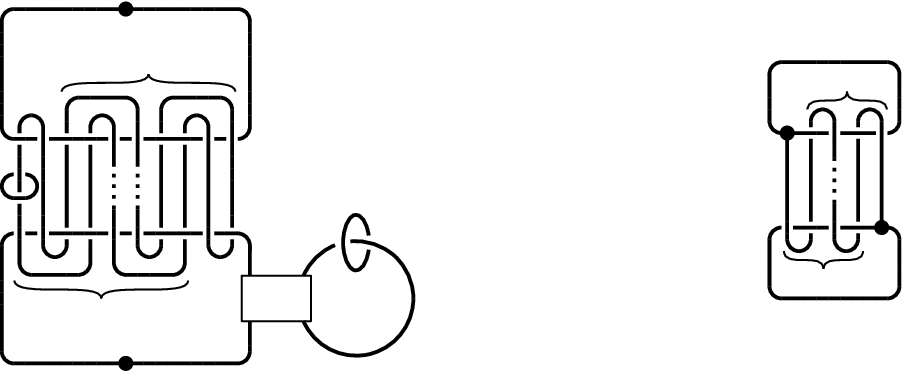}
    \put(-2.5,19){$0$}
    \put(27,20){$0$}
    \put(41,17){$0$}
    \put(46,3){$0$}
    \put(30,7){$2$}
    \put(15,34){$n$}
    \put(10,5){$n$}

    \put(50,19){$\cup$ $2$ $3$-handles}
    \put(57,15){$4$-handle}

    \put(93,31.5){$n$}
    \put(90,9){$n$}
    \end{overpic}
\caption{(Left) A handle diagram of the Price twist $\tau_{S(T_{2,2n+1})}$. (Right) A $\tau$-handle diagram of the Price twist $\tau_{S(T_{2,2n+1})}$. See Figures \ref{fig:S(T_{2,3})} and \ref{fig:S(T_{2,5})} for $n=1$ and $2$, respectively.}
\label{fig:S(T_{2,2n+1})second}
\end{figure}
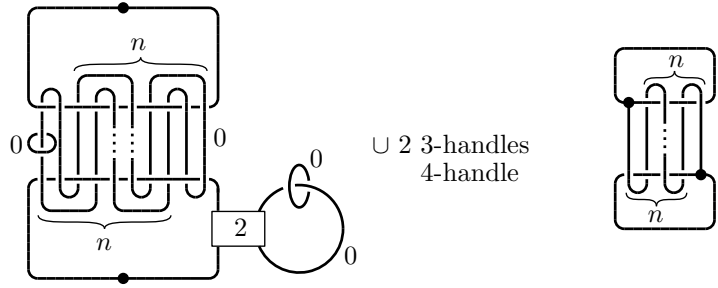

\subsection{Special handle calculus}
Suppose that $K$ is a ribbon $2$-knot and the number of the 1-handles in a handle decomposition for which a $\tau$-handle diagram of $\tau_K$ can be drawn is $n$. 
From Subsection \ref{subsec:diagram}, we can obtain the presentation $\langle {\bf x}  \mid {\bf r}={\bf 1}\rangle$ of $\pi_1(\tau_{K})$ from a handle diagram of $\tau_{K}$.  
Here, ${\bf x}$ is a generating set $\{x_1,\ldots,x_n\}$ that corresponds to the dotted circles in a handle diagram and ${\bf r}={\bf 1}$ is a set of relations $\{r_1=1,\ldots,r_n=1\}$ that corresponds to framed knots that entangle some dotted circles in the handle diagram. 
We call this presentation a {\it $\tau$-presentation of $\tau_{K}$}. 
Without loss of generality, we may assume 
$r_1=x_1^2$ and $r_k=x_{i_k}w_kx_{j_k}w_k^{-1}$ ($k\ge2$), where $(i_k,j_k)\in\{(m_1,m_2)\in\mathbb{Z}^2 \mid 1 \le m_1,m_2\le n, m_1 \neq m_2\}$ and $w_k$ is a word in the generating set ${\bf x}$. 
The operation of obtaining a finite presentation of $\pi_1(\tau_K)$ from a $\tau$-handle diagram is sometimes denoted by $\tau$-d.
\begin{exm}
The presentation $\langle x_1, x_2 \mid x_1^2=1, x_1(x_2x_1)^nx_2((x_2x_1)^n)^{-1}=1\rangle$ is a $\tau$-presentation of $\tau_{S(T_{2,2n+1})}$ (see Figure \ref{fig:S(T_{2,2n+1})second}). 
\end{exm}

In this subsection, we introduce some calculus for $\tau$-presentations that correspond to handle calculus for handle diagrams.
All of the following deformations on finitely presented groups of $\pi_1(\tau_K)$ can be realized as handle calculus that preserve the diffeomorphism type of $\tau_K$:
\begin{itemize}
\item[(a)] \textbf{Isotopy} \  For any word $u, v, w$ in the generating set ${\bf x}$, we obtain 
\centerline{$uww^{-1}v=1 \longleftrightarrow uv=1 \longleftrightarrow uw^{-1}wv=1$.}
These transformations are sometimes denoted by i.

\item[(b)] \textbf{Handle slide} \ For any relations $r_i=1$ and $r_j=1$, we obtain 
\centerline{$r_i=1$, $r_j=1$ $\longleftrightarrow$ $r_i=1$, $r_ir_j=1$.}
These transformations are sometimes denoted by s.

\item[(c)] \textbf{Handle canceling/creating} \ For any element $x_k$ in ${\bf x}$ and any set of relations ${\bf r'}= {\bf1}$ that each relation does not contain $x_k$, we obtain 
 $$\langle x_k, {\bf x'} \mid x_kw^{-1}=1 , {\bf r'}= {\bf1}\rangle=\langle {\bf x'} \mid {\bf r'}= {\bf1}\rangle$$
Transforming from the left side to the right side corresponds to handle cancellation, and transforming in the opposite direction corresponds to handle creation. 
Note that handle canceling/creating corresponds to only a canceling 1-2 pair. 
These transformations are sometimes denoted by c. 

\item[(d)] \textbf{Deformations $\boldsymbol{\alpha}$ and $\boldsymbol{\beta}$} \ 
By combining Corollaries \ref{cor:alpha} and \ref{cor:beta}, we obtain 
\begin{eqnarray*}
&&x_{i_k}w_kx_{j_k}w_k^{-1}=1 \longleftrightarrow x_{i_k}w_kx_{j_k}^{-1}w_k^{-1}=1\\
&\longleftrightarrow&x_{i_k}^{-1}w_kx_{j_k}w_k^{-1}=1 \longleftrightarrow  x_{i_k}^{-1}w_kx_{j_k}^{-1}w_k^{-1}=1.
\end{eqnarray*}
Also, a word $x_i$ (resp. $x_i^{-1}$) in $w_k$ can be changed to $x_i^{-1}$ (resp. $x_i$). 
These deformations $\alpha$ (resp. $\beta$) are sometimes denoted by $\alpha$ (resp. $\beta$).
\end{itemize}

Note that in a handle diagram of $\tau_K$, changing a self-intersection of a framed knot entangled with a dotted circle preserves the diffeomorphism type of $\tau_K$. 
We also note that base transformations (inversion and permutation of generators and relators) in the $\tau$-presentation of $\tau_K$ do not change the diffeomorphism type of $\tau_K$. 

\begin{lem}
Let $K_1$ and $K_2$ be ribbon $2$-knots in $S^4$. 
The Price twists $\tau_{K_1}$ and $\tau_{K_2}$ are diffeomorphic if and only if their $\tau$-presentations are related by a finite sequence of the above calculus $(a)$, $(b)$, $(c)$ and $(d)$, changing a self-intersection of a framed knot entangled with a dotted circle and base transformations, and handle canceling or handle creating a canceling $2$-$3$ or $3$-$4$ pair in handle diagrams. 
\end{lem}
\begin{proof}
It is known \cite{PMIHES_1970__39__5_0} that two closed 4-manifolds are diffeomorphic if and only if two corresponding handle diagrams are related by a finite sequence of isotopy, handle slide, handle cancellation and handle creation. 
\end{proof}

A finite sequence of transformations consisting of $(a)$, $(b)$, $(c)$ and $(d)$ in handle diagrams is called {\it $\tau$-handle calculus}. 
This process transforms all the relations in a $\tau$-presentation of $\pi_1(\tau_K)$ while preserving conjugacy between any two generators in each generating set. 

\begin{rem}
Performing $\tau$-handle calculus on a $\tau$-handle diagram corresponds to a sequence of transformations that preserve the conjugacy between any two generators in a generating set of a $\tau$-presentation.
When we perform $\tau$-handle calculus using an ordinary handle diagram, it suffices to preserve the conjugacy between any two generators in a $\tau$-presentation only immediately before applying the deformations $\alpha$ and $\beta$. 
\end{rem}

Finally, we explain how to construct a $\tau$-handle diagram of the Price twist $\tau_{K}$ from a ribbon 1-knot $k$ that is the equatorial knot of a ribbon 2-knot $K$ (the definition of the equatorial knot is given before Corollary \ref{cor:1-fusion}). 
For ribbon bands placed on a knot diagram $D(k)$ of $k$ as in the left side of Figure \ref{fig:knottotau} (red), a $\tau$-handle diagram of $\tau_{K}$ can be obtained by replacing all the ribbon bands as in the right side of Figure \ref{fig:knottotau}. 
\begin{figure}[htbp]
    \centering
\begin{overpic}[scale=0.6]
{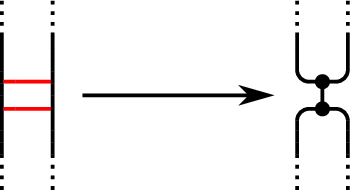}
\end{overpic}
\caption{A method to change a knot diagram $D(k)$ of a ribbon 1-knot $k$ to a $\tau$-handle diagram of $\tau_{K}$.}
\label{fig:knottotau}
\end{figure}

\subsection{Diffeomorphism types of some non-simply connected Price twists}\label{subsec:diffeo type of Price twist}
Recall that for a ribbon 2-knot $K$ obtained from a trivial $(n + 1)$-component 2-link by adding $n$ 1-handles, the {\it ribbon fusion number} (or simply {\it fusion number}) $rf(K)$ of $K$ is the minimal number of $n$ possible for $K$. Moreover, recall that the $(p,q)$-torus knot and the spun knot of a 1-knot $k$ are denoted by $T_{p,q}$ and $S(k)$, respectively.

\begin{thm}\label{thm:1-fusion}
Let $K$ be a ribbon $2$-knot of $1$-fusion. Then, $\tau_K$ is diffeomorphic to $\tau_{S(T_{2,n})}$, where $n=\det(K)$. 
\end{thm}

\begin{proof}
We first show that $\tau_K$ is diffeomorphic to $\tau_{S(T_{2,n})}$ for some odd integer $n \ge 1$.
Suppose that $K$ has a ribbon presentation $R(m_1, n_1, \ldots, m_s, n_s)$ described as in Figure \ref{fig:pre1-fusion} (see also \cite[Figures 1 and 2]{MR4160334} for example). 
Then, a $\tau$-handle diagram of $\tau_K$ can be drawn as in the left side of Figure \ref{fig:1-fusion}.
By repeatedly applying the deformation $\alpha$ (Proposition \ref{prop:alpha}) to the left side of Figure \ref{fig:1-fusionpre}, we obtain the transformations depicted in Figure \ref{fig:1-fusionpre}.
By applying the transformations in Figure \ref{fig:1-fusionpre} and the deformation $\beta$ (Proposition \ref{prop:beta}) to the left side of Figure \ref{fig:1-fusion} some times, we obtain the right side of Figure \ref{fig:1-fusion}, which is just the $\tau$-handle diagram of $\tau_{S(T_{2,n})}$ depicted in Figure \ref{fig:S(T_{2,2n+1})second} for some odd integer $n \ge 1$. 

We next show that $n=\det(K)$. Let $\Delta_K(t)$ be an Alexander polynomial of $K$. We see from the ribbon presentation $R(m_1, n_1, \ldots, m_s, n_s)$ of $K$ that 

\begin{eqnarray*}
&&\Delta_K(t)=t^{m_1 + m_2 + \cdots + m_s}(1-t^{-n_s}+t^{-m_s-n_s}-t^{-m_{s-1}-n_s-m_s}\\
&&+\cdots-t^{-n_1-\cdots-m_{s-1}-n_s-m_s}+t^{-m_1-n_1-\cdots-m_{s-1}-n_s-m_s})
\end{eqnarray*}
(see for example \cite{MR1701683, MR133126, MR4160334, MR467763}). 
Thus, we have 

\begin{eqnarray*}
&&\det(K)=|\Delta_K(-1)|\\
&=&|1-(-1)^{n_s}+(-1)^{m_s+n_s}-(-1)^{m_{s-1}+n_s+m_s}\\
&&+\cdots-(-1)^{n_1+\cdots+m_{s-1}+n_s+m_s}+(-1)^{m_1+n_1+\cdots+m_{s-1}+n_s+m_s}|.
\end{eqnarray*}

We can see that $\det(K)=\det(R(1,1,\ldots,1,1))$, where the number of 1 is $p-1$ for $p:=\det(K)$. Moreover, we see that $S(T_{2,p})$ has the ribbon presentation $R(1,1,\ldots,1,1)$. Thus, we have $n=p=\det(K)$.
This completes the proof.
\end{proof}

\begin{rem}
The methods in the proof of Theorem \ref{thm:1-fusion} are similar to those given in \cite[Appendix]{viro1973local}. 
\end{rem}

Note that by Theorem \ref{thm:pi1oftaut22n+1} (Corollary \ref{cor:not spin case}), Theorem \ref{thm:1-fusion} classifies the diffeomorphism types of $\tau_K$ completely for ribbon 2-knots $K$ of 1-fusion. 

\begin{figure}[htbp]
    \centering
\begin{overpic}[scale=0.6
]
{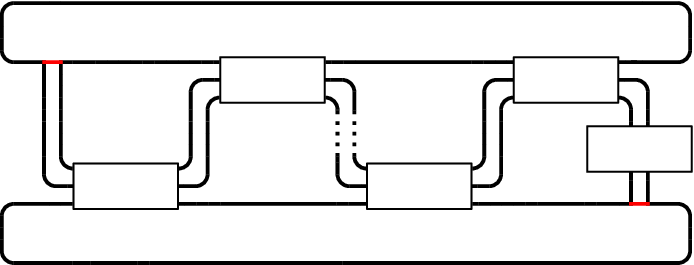}
\put(13,10){$-m_1$}
\put(37,26){$n_1$}
\put(55,10){$-m_s$}
\put(79,26){$n_s$}
\put(91,15.5){$x$}
\end{overpic}
\caption{A knot diagram of the equatorial knot of $K$ and its ribbon band (red). The definition of the equatorial knot is given before Example \ref{exm:12 cross}. A box labeled $n$ represents $n$ full twists. Note that $x=\sum_{i=1}^{s}(m_i-n_i)$. }
\label{fig:pre1-fusion}
\end{figure}

\begin{figure}[htbp]
    \centering
\begin{overpic}[scale=0.6
]
{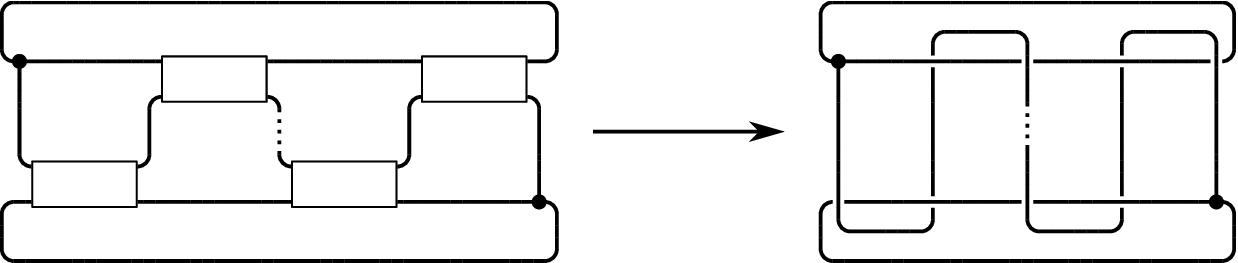}
\put(4,5.5){$-m_1$}
\put(16,14.5){$n_1$}
\put(25,5.5){$-m_s$}
\put(37,14.5){$n_s$}
\end{overpic}
\caption{A $\tau$-handle diagram of $\tau_K$ for a ribbon 2-knot $K$ of 1-fusion.}
\label{fig:1-fusion}
\end{figure}

\begin{figure}[htbp]
    \centering
\begin{overpic}[scale=0.6]
{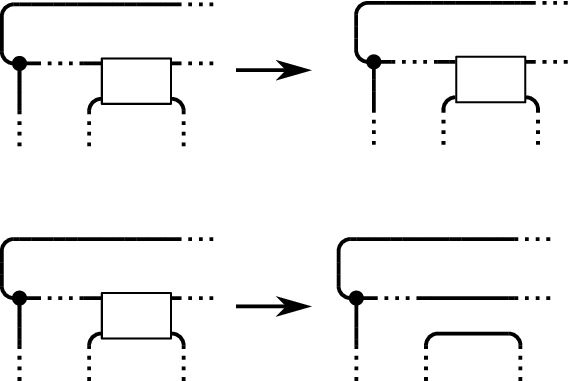}
\put(22,51){$n$}
\put(45,57){$\alpha$}
\put(84,51){$1$}
\put(110,51){($n$ is odd)}

\put(22,9.5){$n$}
\put(45,15.5){$\alpha$}
\put(110,9.5){($n$ is even)}
\end{overpic}
\caption{Handle calculus using the deformation $\alpha$ (top: $n$ is odd , bottom: $n$ is even).}
\label{fig:1-fusionpre}
\end{figure}
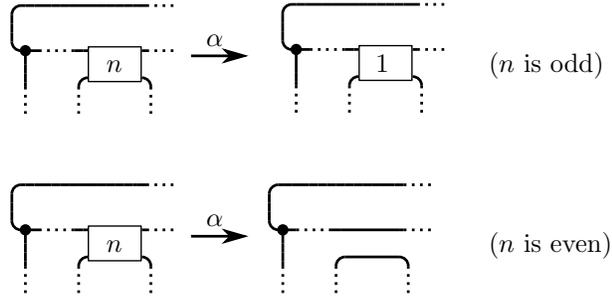

\begin{cor}\label{cor:2-bridge}
Let $k$ be a $2$-bridge knot. Then, $\tau_{S(k)}$ is diffeomorphic to $\tau_{S(T_{2, n})}$, where $n=\det(k)$.
\end{cor}

\begin{proof}
We see from \cite[Proposition 4]{zbMATH01117469} that $rf(S(k))=1$. 
Thus, by Theorem \ref{thm:1-fusion}, $\tau_{S(k)}$ is diffeomorphic to $\tau_{S(T_{2,n})}$, where $n=\det(S(k))$. Moreover, we know that for each 1-knot $j$, $\det(S(j))=\det(j)$ since their Alexander polynomials are the same up to $\pm t^m$. This completes the proof.
\end{proof}

For the determinant of a 2-bridge knot, the following are known (for example, see \cite[p.20]{zbMATH05522298}).

\begin{lem}\label{lem:2-bridge}
Let $a_0 \in \mathbb{R}$, $a_1, \ldots, a_n \in \mathbb{R}-\{0\}$, $p_0:=a_0$, $p_1:=a_0a_1+1$, $p_k:=a_kp_{k-1}+p_{k-2}$, $q_0:=1$, $q_1:=a_1$ and $q_k:=a_kq_{k-1}+q_{k-2}$ $(k \ge 2)$. If $q_1, \ldots, q_k \not= 0$, then
\[
[a_0, a_1, \ldots, a_k]=\cfrac{p_k}{q_k},
\]
where
\[
[a_0, a_1, \ldots, a_k]:=a_0+\cfrac{1}{a_1+\cfrac{1}{\cdots+{\cfrac{1}{a_k}}}}.
\]
\end{lem}

We use $C[a_1, \ldots, a_n]$ as the Conway notation of 2-bridge knots.
 
\begin{cor}\label{cor:2-bridgedet}
Let $d_n$ be the determinant of a $2$-bridge knot $C[a_1, \ldots, a_n]$. Then, $d_n=a_n d_{n-1} + d_{n-2}$, where $d_0:=1$ and $d_1=a_1$.
\end{cor}

\begin{proof}
The numerator of $[a_1, \ldots, a_n]$ is the determinant of the $2$-bridge knot $C[a_1, \ldots, a_n]$. Thus, we can apply $d_n$ to $p_n$ in Lemma \ref{lem:2-bridge}.
\end{proof}

\begin{exm}\label{exm:2-plat 2-knot}
Let $\beta$ be a 2-dimensional $2n$-braid in $D^4$. 
For the definition of a 2-dimensional braid, see \cite{kamada2017surface}. 
Then $\partial \beta$ is a $2n$-component link in $S^3$. 
A surface link obtained from $\beta$ by trivially gluing $n$ annuli to $\partial \beta$ is called the {\it plat closure} of the $2n$-braid $\beta$ (for details, see \cite{arXiv:2105.08634}). 
A 2-knot $K$ is said to be {\it $n$-plat} if $K$ is ambiently isotopic to the plat closure of some $2n$-braid. 
An $n$-plat 2-knot is first defined in \cite{arXiv:2506.15401}. 
Any 1-plat 2-knot is either a trivial 2-knot or a trivial non-orientable surface knot \cite[Theorem 1.1]{arXiv:2105.08634}. 
Yasuda \cite{arXiv:2506.15401} introduced normal forms of 2-plat 2-knots using rational numbers. 
Let $p$ and $a$ be integers which satisfy that $p$ is positive and $\gcd(p, a) = 1$. 

Let $F(p/a)$ be the 2-knot whose equatorial knot is represented by the knot diagram depicted in Figure \ref{fig:2plat2knot}, where $p/a = [c_1, \ldots, c_m]$. 
Note that the roles of the numerator $p$ and the denominator $a$ of the fraction $p/a$ for $F(p/a)$ are reversed in \cite{arXiv:2506.15401}. 
Any 2-plat 2-knot $K$ is isotopic to $F(p/a)$ for some positive odd integer $p$ and integer $a$ with $\gcd(p,a)=1$ by \cite[Theorem 1.1]{arXiv:2506.15401}. 
The 2-plat 2-knot $F(p/a)$ is a ribbon 2-knot of 1-fusion (see \cite[Proposition 2.6]{arXiv:2506.15401}). 
Thus, by Theorem \ref{thm:1-fusion}, the Price twist $\tau_{F(p/a)}$ is diffeomorphic to $\tau_{S(T_{2,n})}$, where $n=\det(F(p/a))=p$ (see \cite[Corollary 1.4]{arXiv:2506.15401}). 
Note that while the Alexander polynomial of the spun knot $S(T_{2,p})$ of the torus knot $T_{2,p}$ is reciprocal (i.e. $\Delta_{S(T_{2,p})}(t) \, \dot{=} \, \Delta_{S(T_{2,p})}(t^{-1})$, where, $g(t) \, \dot{=} \, h(t)$ means that $g(t)$ equals $h(t)$ up to multiplication by $\pm t^m$ for some integer $m$), the 2-plat 2-knot $F(p/a)$ with $p\le2000$ is not (i.e. $\Delta_{F(p/a)}(t) \, \dot{\neq} \, \Delta_{F(p/a)}(t^{-1})$) from \cite[Theorem 1.7]{arXiv:2506.15401}. 

It follows from this example that there exist two 2-knots $K_1$ and $K_2$ such that $K_1$ is not isotopic to $K_2$, $rf(K_1) = rf(K_2) = 1$, and $\tau_{K_1}$ is diffeomorphic to $\tau_{K_2}$.

\begin{figure}[htbp]
    \centering
\begin{overpic}[scale=0.6
]
{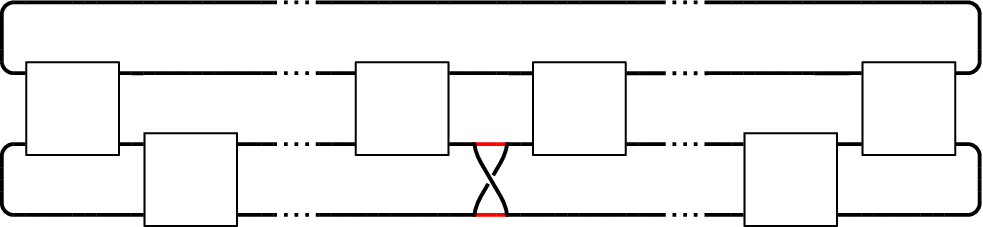}
\put(6,11){$c_1$}
\put(17,4){$-c_2$}
\put(39,11){$c_m$}

\put(56,11){$-c_m$}
\put(79,4){$c_2$}
\put(90,11){$-c_1$}
\end{overpic}
\caption{A knot diagram for the equatorial knot of $F(p/a)$ and a ribbon band (red).}
\label{fig:2plat2knot}
\end{figure}
\end{exm}

For a diagram $D(k)$ of a ribbon 1-knot $k$, let $R(D(k))$ denote a ribbon 2-knot obtained by taking the double of a ribbon disk properly embedded in $D^4$ that bounds the ribbon 1-knot $k$ described by $D(k)$. The ribbon 1-knot $k$ is called the {\it equatorial knot} of $R(D(k))$. Note that every ribbon 2-knot is described by $R(D(k))$ for some knot diagram $D(k)$ of some ribbon 1-knot $k$. 

\begin{cor}\label{cor:1-fusion}
Let $k$ be a ribbon $1$-knot of $1$-fusion. 
Then, there exists a knot diagram $D(k)$ of $k$ such that $rf(R(D(k))) \le 1$ and $\tau_{R(D(k))}$ is diffeomorphic to $\tau_{S(T_{2,n})}$, where $n=\sqrt{\det(k)}$. 
\end{cor}

\begin{proof}
First, we prove this claim for the 1-knot $k=k(m_1,n_1,\ldots,m_s,n_s)$ whose diagram $D(k)$ is shown in Figure \ref{fig:pre1-fusion}. 
In this case, the $2$-knot $R(D(k))$ is isotopic to $R(m_1,n_1,\ldots,m_s,n_s)$ and $rf(R(D(k)))\le1$. 
From \cite{MR467763} or \cite[Remark 1.8]{mizuma2005ribbon}, we obtain 
$\Delta_k(t)=f(t)f(t^{-1})$, where 
$$
f(t)=\sum_{i=1}^s(t^{\phi(i)}-t^{\psi(i)})+1, \phi(i)=\sum_{j=i}^s(m_j+n_j) \text{ and } \psi(i)=-m_i+\sum_{j=i}^s(m_j+n_j). 
$$
Therefore, we have
\begin{eqnarray*}
f(t)&=&1+\sum_{i=1}^s(-t^{\psi(s-i+1)}+t^{\phi(s-i+1)})\\    
&=&1-t^{n_s}+t^{m_s+n_s}-t^{m_{s-1}+n_s+m_s}\\
&&+\cdots-t^{n_1+\cdots+m_{s-1}+n_s+m_s}+t^{m_1+n_1+\cdots+m_{s-1}+n_s+m_s}\\ 
&\dot{=}&(t^{-1})^{m_1 + m_2 + \cdots + m_s}(1-(t^{-1})^{-n_s}+(t^{-1})^{-m_s-n_s}-(t^{-1})^{-m_{s-1}-n_s-m_s}\\
&&+\cdots-(t^{-1})^{-n_1-\cdots-m_{s-1}-n_s-m_s}+(t^{-1})^{-m_1-n_1-\cdots-m_{s-1}-n_s-m_s})\\ 
&=&\Delta_{R(D(k))}(t^{-1}). 
\end{eqnarray*}
Similarly, we obtain $f(t^{-1})\,\dot{=}\,\Delta_{R(D(k))}(t)$. 
Therefore, we have 
$$\Delta_{k}(t)\,\dot{=}\,\Delta_{R(D(k))}(t)\Delta_{R(D(k))}(t^{-1}). $$
This means that $\det(k)=(\det(R(D(k))))^2$, that is, $\det(R(D(k)))=\sqrt{\det(k)}$. 
Thus, by Theorem \ref{thm:1-fusion}, $\tau_{R(D(k))}$ is diffeomorphic to $\tau_{S(T_{2,n})}$, where $n=\sqrt{\det(k)}$. 

We next prove this claim for any ribbon 1-knot $k$ of 1-fusion. 
There exist a positive integer $s$ and integers $m_1, n_1, \ldots, m_s$ and $n_s$ such that $k$ and $k(m_1,n_1,\ldots,m_s,n_s)$ differ only an integer number of full twists and self-intersections of a ribbon band and isotopy. 
By \cite{MR467763} or \cite[Remark 1.8]{mizuma2005ribbon}, these differences do not affect the Alexander polynomial. 
Therefore, we have
$\Delta_k(t)\,\dot{=}\,\Delta_{k(m_1,n_1,\ldots,m_s,n_s)}(t)$. 
In this case, the 1-knot $k$ admits a knot diagram $D(k)$ consisting of two disks and a single band. 
Moreover, by performing finitely many band self-crossing changes and full twists of the band in $D(k)$, one can arrange $D(k)$ so that the resulting knot diagram and the attachment of the ribbon disks coincide with those of $k(m_1,n_1,\ldots,m_s,n_s)$ depicted in the top side of Figure \ref{fig:1-2correspond}.
Then, there exists a knot diagram $D(k)$ such that the $2$-knot $R(D(k))$ is isotopic to $R(m_1,n_1,\ldots,m_s,n_s)$ and $rf(R(D(k)))\le1$ (see Figure \ref{fig:1-2correspond}). 
Therefore, we obtain 
$$\Delta_{R(D(k))}(t)=\Delta_{R(m_1,n_1,\ldots,m_s,n_s)}(t).$$ 
Thus, we have 
$$\Delta_{k}(t)\,\dot{=}\,\Delta_{R(D(k))}(t)\Delta_{R(D(k))}(t^{-1})$$
and by Theorem \ref{thm:1-fusion}, $\tau_{R(D(k))}$ is diffeomorphic to $\tau_{S(T_{2,n})}$, where $n=\sqrt{\det(k)}$. 
This completes the proof. 
\end{proof}

\begin{figure}[htbp]
    \centering
\begin{overpic}[scale=0.6
]
{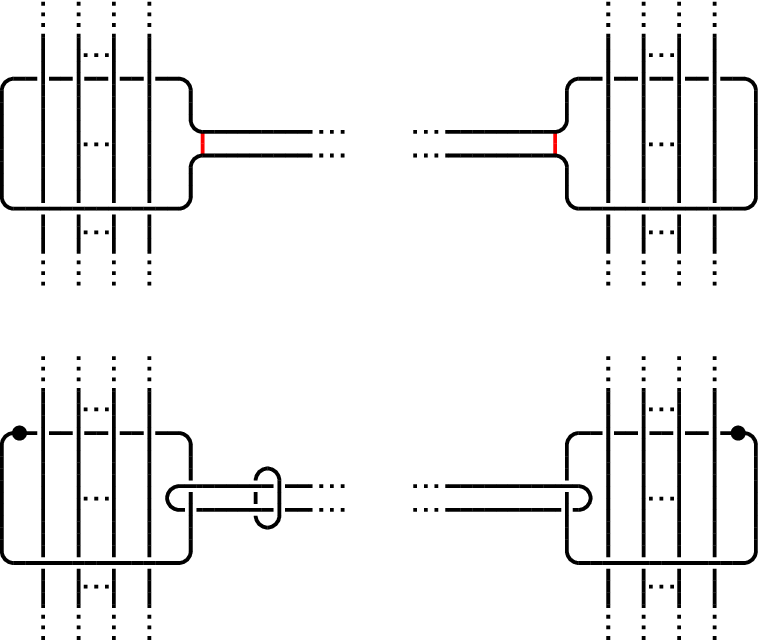}
\put(35,23.5){$0$}
\put(70,21){$0$}
\end{overpic}
\caption{Top: The knot diagram $D(k)$ of the ribbon 1-knot $k$ and a ribbon band (red). 
Note that the two disks at both ends share the same side (front or back) of the ribbon disk. 
Bottom: A handle diagram of the exterior $E(R(D(k)))$. }
\label{fig:1-2correspond}
\end{figure}

Recall that the mirror image of a knot $k$ is denoted by $k^*$. We see some examples of Corollary \ref{cor:1-fusion}. 

\begin{exm}\label{exm:12 cross}
Let $k$ be a ribbon $1$-knot up to $12$ crossings except for $12a_{631}$, $12a_{990}$, $12n_{553}$, $12n_{556}$, $3_1\#6_1\#3_1^*$ and $3_1\#3_1\#3_1^*\#3_1^*$. We see that the fusion number of $k$ is $1$ from the knot diagram in Figure \ref{fig:composite cases (ribbon)} or Table \ref{tab:12 crossing ribbon 1-knot}. 
For the knot diagram $D(k)$ mentioned in Table \ref{tab:12 crossing ribbon 1-knot}, we obtain that $rf(R(D(k)))\le1$. 
Therefore, $\tau_{R(D(k))}$ is diffeomorphic to $\tau_{S(T_{2,n})}$ by Proposition \ref{prop:pao} and Corollary \ref{cor:1-fusion}, where $n=\det(R(D(k)))=\sqrt{\det(k)}$. 
If $\det(k)\neq1$, then $\det(R(D(k)))\neq1$. 
Note that $\det(O)=1$. 
Thus, $R(D(k))$ is not isomorphic to $O$ and $rf(R(D(k)))=1$. 
If $\det(k)=1$, then $k$ is $0_1, 10_{153}, 11n_{42}, 11n_{49}, 11n_{116}, 12n_{19}, 12n_{214}, 12n_{309}, 12n_{313}, 12n_{318}$ or $12n_{430}$. 
From handle diagrams of the exteriors $E(R(D(k))$ obtained by the diagrams $D(k)$ and presentations of the fundamental groups $\pi_1(R(D(k)))$ from these handle diagrams, we can check that the 2-knots $R(D(0_1))$, $R(D(10_{153}))$, $R(D(11n_{42}))$, $R(D(11n_{49}))$, $R(D(11n_{116}))$, $R(D(12n_{19}))$, $R(D(12n_{214}))$, $R(D(12n_{309}))$, $R(D(12n_{313}))$, $R(D(12n_{318}))$ and $R(D(12n_{430}))$ are isotopic to $O$, $R(1,2)$, $O$, $R(-1,2)$, $R(-1,2)$, $R(-1,-2)$, $R(1,2)$, $R(1,2)$, $O$, $R(1,2)$ and $O$, respectively. 
Note that $rf(K)=0$ if and only if $K$ is isotopic to $O$.  
Thus, $rf(R(D(k)))=0$ if and only if $k$ is $0_1$, $11n_{42}$, $12n_{313}$ or $12n_{430}$ in this case.

It follows from this example that there exists a $2$-knot $K$ such that $rf(K)\neq0$ and $\tau_K$ is diffeomorphic to $\tau_O$. 
\end{exm}

\begin{figure}[htbp]
    \centering
\begin{overpic}[scale=0.6]
{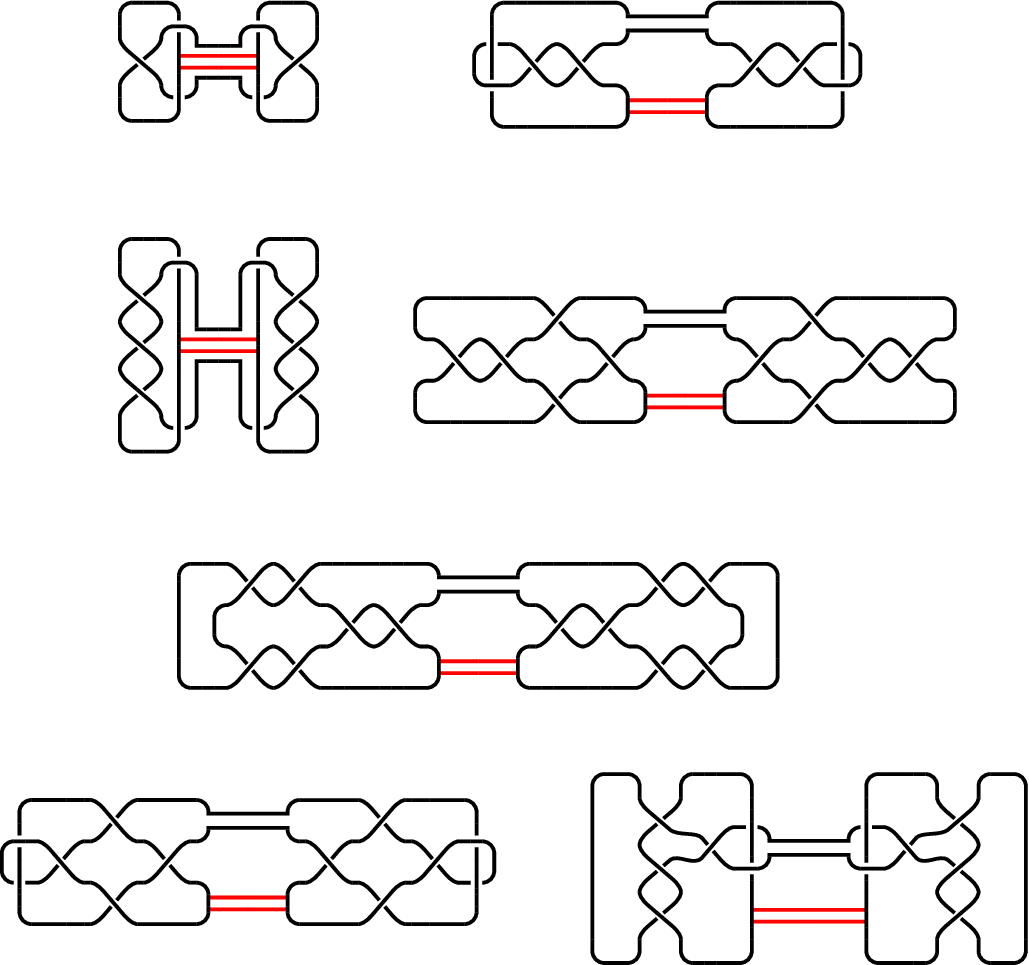}
\put(17,77){$3_1\#3_1^*$}
\put(61,77){$4_1\#4_1^*$}
\put(17,46){$5_1\#5_1^*$}
\put(62,48){$5_2\#5_2^*$}
\put(43,23){$6_1\#6_1^*$}
\put(20,-1){$6_2\#6_2^*$}
\put(74,-3){$6_3\#6_3^*$}
\end{overpic}
\caption{Knot diagrams of the composite ribbon knots $k\#k^*$ up to 12 crossings and ribbon bands (red).}
\label{fig:composite cases (ribbon)}
\end{figure}

\begin{exm}\label{exm:pretzel}
Let $k$ be the pretzel knot $P(-p,p,q)$ for any odd integer $p$ and any integer $q$ and $P(1,a,-a-4)$ for any odd integer $a$. 
From the knot diagram $D(P(-p,p,q))$ of $P(-p,p,q)$ depicted in Figure \ref{fig:pretzel}, we can see that 
$$
rf(P(-p,p,q))
=
\begin{cases}
0&(\abs{p}=1),\\
1&(\abs{p}\neq1).
\end{cases}
$$
If $\abs{p}=1$, then $R(D(P(-p,p,q)))$ is isotopic to $O$. 
If $\abs{p}\neq1$, then we have 
\begin{eqnarray*}
&&\pi_1(E(R(D(P(-p,p,q)))))\\
&\cong&
\begin{cases}
\langle x, y \mid x = ((x^{-1}y^{-1})^{\frac{\abs{p}-1}{2}})^{-1}y(x^{-1}y^{-1})^{\frac{\abs{p}-1}{2}}\rangle&(q\text{ is even}),\\
\langle x, y \mid x = ((xy^{-1})^{\frac{\abs{p}-1}{2}})^{-1}y(xy^{-1})^{\frac{\abs{p}-1}{2}}\rangle&(q\text{ is odd})
\end{cases}
\end{eqnarray*}
from the right side of Figure \ref{fig:pretzel}. 
Therefore, if $q$ is even, the $2$-knot $R(D(P(-p,p,q)))$ is isotopic to $R(1,1,\ldots,1,1)$, where the number of $1$ is $\abs{p}-1$. 
Furthermore, we can see that the $2$-knot $R(1,1,\ldots,1,1)$ is isotopic to $S(T_{2,p})$. Thus, $R(D(P(-p,p,q)))$ is isotopic to $S(T_{2,p})$. 
If $q$ is odd, the $2$-knot $R(D(P(-p,p,q)))$ is isotopic to $R(1,-1,\ldots,1,-1)$, where the numbers of $1$ and $-1$ is $(\abs{p}-1)/2$. 
Furthermore, we can see that the $2$-knot $R(1,-1,\ldots,1,-1)$ is isotopic to $F(p)$. Thus, $R(D(P(-p,p,q)))$ is isotopic to $F(p)$. 
Hence, we have 
$$
rf(D(P(-p,p,q)))
=
\begin{cases}
0&(\abs{p}=1),\\
1&(\abs{p}\neq1).
\end{cases}
$$
Thus, by Proposition \ref{prop:pao}, Corollary \ref{cor:1-fusion} and Example \ref{exm:2-plat 2-knot}, $\tau_{R(D(P(-p,p,q))}$ is diffeomorphic to $\tau_{S(T_{2,p})}$. 
The statement for $P(1,a,-a-4)$ also holds from Example \ref{exm: 2-bridge ribbon} since $P(1,a,-a-4)$ is the 2-bridge knot $C[a+1,a+3]$ which belongs to Family 0. 
Note that $\det(P(-p,p,q))=p^2$ for any odd integer $p$ and any integer $q$. 
Indeed, $\det(P(p,q,r))=\abs{pq+qr+rp}$ for odd integers $p$, $q$ and $r$. Thus, if $q$ is odd, then $\det(P(-p,p,q))=p^2$.
Since $P(-p,p,0)=T_{2,p}\#T_{2,p}^*$, we have
$$\Delta_{P(-p,p,0)}(t)=\Delta_{T_{2,p}\#T_{2,p}^*}(t)=\Delta_{T_{2,p}}(t)^2=(t^{p-1}-t^{p-2}+ \dots +t^2-t+1)^2.$$
For any even integer $q$, we obtain 
$$\Delta_{P(-p,p,q)}(t)-\Delta_{P(-p,p,q+2)}(t)=-(t^{1/2}-t^{-1/2})\Delta_{o \sqcup o}(t)=0.$$
Thus, we have
$$\det(P(-p,p,q))=|\Delta_{P(-p,p,q)}(-1)|=|\Delta_{P(-p,p,0)}(-1)|=p^2$$
for any even integer $q$.
We can also calculate the determinant directly from \cite[Theorem 1]{arXiv:2502.10370}. 
Therefore, the statement $\det(R(D(P(-p,p,q))))=\sqrt{\det(P(-p,p,q))}$ also holds.

\begin{figure}[htbp]
    \centering
\begin{overpic}[scale=0.6
]
{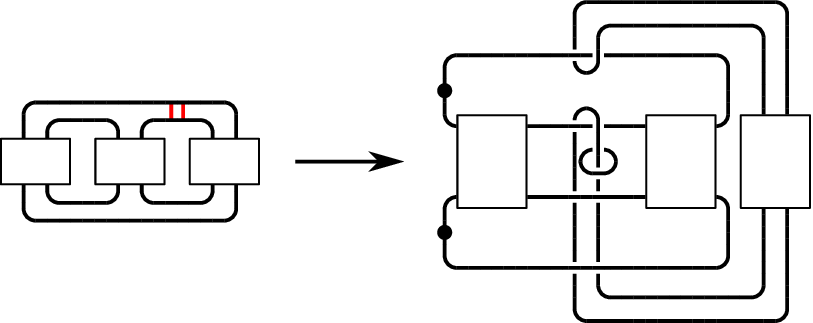}
\put(1,19){$-p$}
\put(15,19){$p$}
\put(27,19){$q$}

\put(76.5,19){$0$}
\put(98,35){$0$}
\put(57,19){$-\displaystyle\frac{p}{2}$}
\put(82,19){$\displaystyle\frac{p}{2}$}
\put(94,19){$\displaystyle\frac{q}{2}$}
\end{overpic}
\caption{Left: The knot diagram $D(P(-p,p,q))$ of the pretzel knot $P(-p,p,q)$ and a ribbon band (red). 
Note that when $\abs{p}=1$, the ribbon band in this ribbon presentation is unnecessary.
Right: A handle diagram of the exterior $E(R(D(P(-p,p,q))))$. 
A box labeled $n/2$ represents $n$ half twists. }
\label{fig:pretzel}
\end{figure}
\end{exm}

\begin{exm}\label{exm: 2-bridge ribbon}
It is known \cite{MR900252, MR2302495} that a $2$-bridge $1$-knot $k$ is ribbon if and only if $k$ is one of the following appearing in \cite{MR4394065} (see also \cite{zbMATH07899097}):

\begin{itemize}
\item (Family 0) $C[a_1,a_2, \ldots, a_{n-1}, a_n, a_n+2, a_{n-1}, \ldots, a_2, a_1]$ with $a_i >0$ for $i=1,2,\ldots,n$,
\item (Family 1) $C[2a,2,2b,-2,-2a,2b]$ with $a,b\not=0$,
\item (Family 2) $C[2a,2,2b,2a,2,2b]$ with $a,b\not=0$.
\end{itemize}
From Corollary \ref{cor:2-bridgedet}, we have $\det(C[a_1,a_2, \ldots, a_{n-1}, a_n, a_n+2, a_{n-1}, \ldots, a_2, a_1])>a_1\ge1$, $\det(C[2a,2,2b,-2,-2a,2b])=(8ab+2b-1)^2>1$ and $\det(C[2a,2,2b,2a,2,2b])=(8ab+2a+2b+1)^2>1$. 
We see from the knot diagram $D(k)$ in \cite{MR4394065} (see also Figures \ref{fig:Family0} and \ref{fig:Family1and2}) and $\det(k)\neq1$ that the fusion numbers of Family 0, 1 and 2 are $1$. 
Since the $2$-knot $R(D(k))$ for the knot diagram $D(k)$ is $1$-fusion, $\tau_{R(D(k))}$ is diffeomorphic to $\tau_{S(T_{2,n})}$ by Corollary \ref{cor:1-fusion}, where $n=\det(R(D(k)))=\sqrt{\det(k)}$. 
\end{exm}

\begin{figure}[htbp]
    \centering
\begin{overpic}[scale=0.6
]
{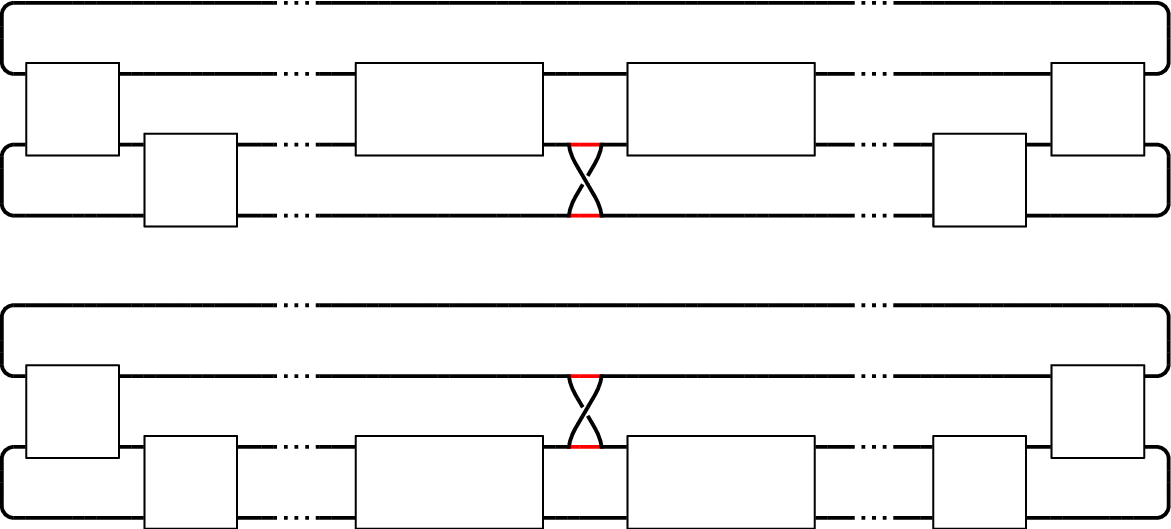}
\put(5,35){$a_1$}
\put(14,29){$-a_2$}
\put(34,35){$a_n+1$}

\put(56,35){$-a_n-1$}
\put(82,29){$a_2$}
\put(91,35){$-a_1$}

\put(5,9){$a_1$}
\put(14,3){$-a_2$}
\put(33,3){$-a_n-1$}

\put(58,3){$a_n+1$}
\put(82,3){$a_2$}
\put(91,9){$-a_1$}
\end{overpic}
\caption{The knot diagram $D(k)$ of the ribbon knot $k=C[a_1,a_2, \ldots, a_{n-1}, a_n, a_n+2, a_{n-1}, \ldots, a_2, a_1]$ of Family 0 and a ribbon band (red). Top: $n$ is odd. Bottom: $n$ is even.}
\label{fig:Family0}
\end{figure}

\begin{figure}[htbp]
    \centering
\begin{overpic}[scale=0.6
]
{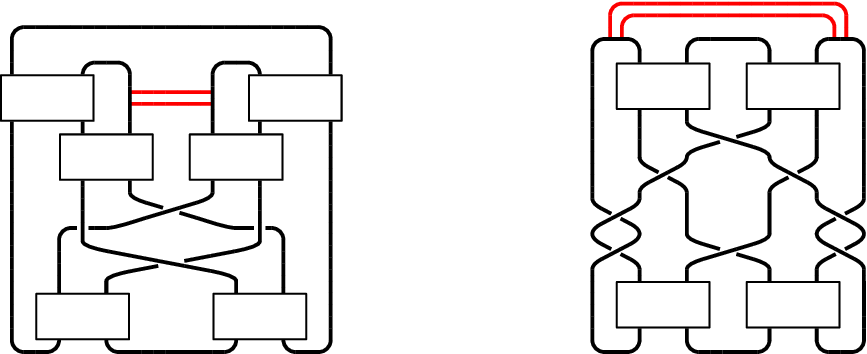}
\put(2,28){$-2b$}
\put(32,28){$2b$}

\put(8,21){$-2a$}
\put(25,21){$2a$}

\put(6,3){$-2a$}
\put(28,3){$2a$}

\put(74.5,29.5){$2a$}
\put(88,29.5){$-2a$}

\put(73,4){$-2b$}
\put(90,4){$2b$}
\end{overpic}
\caption{Left: The knot diagram $D(k_1)$ of the ribbon knot $k_1=C[2a,2,2b,-2,-2a,2b]$ of Family 1 and a ribbon band (red). Right: The knot diagram $D(k_2)$ of the ribbon knot $k_2=C[2a,2,2b,2a,2,2b]$ of Family 2 and a ribbon band (red).}
\label{fig:Family1and2}
\end{figure}

\begin{rem}\label{rem:1-fusion}
In Example \ref{exm:12 cross}, we except for $12a_{631}$, $12a_{990}$, $12n_{553}$, $12n_{556}$, $3_1\#6_1\#3_1^*$ and $3_1\#3_1\#3_1^*\#3_1^*$. We immediately see from \cite{lamm2021search}, Figures \ref{fig:316131} and \ref{fig:31313131andtau} that the fusion numbers of these exceptional knots are all $2$ or less. 

\begin{figure}[htbp]
    \centering
\begin{overpic}[scale=0.6
]
{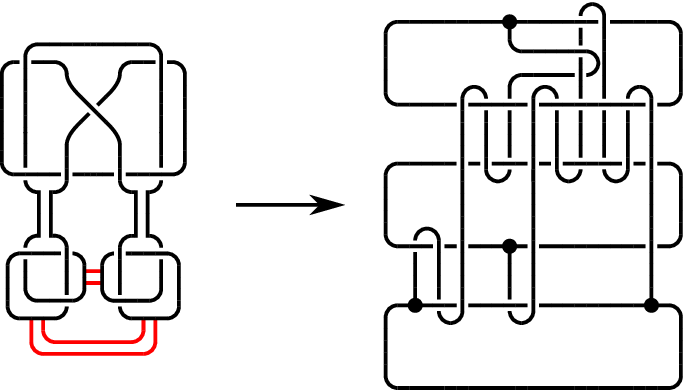}
\put(59,46){$x_1$}
\put(59,26){$x_2$}
\put(59,3){$x_3$}
\end{overpic}
\caption{Left: The knot diagram $D(3_1\#6_1\#3_1^*)$ of $3_1\#6_1\#3_1^*$ and ribbon bands (red). Right: A $\tau$-handle diagram of $\tau_{R(D(3_1\#6_1\#3_1^*))}$. 
}
\label{fig:316131}
\end{figure}

\begin{figure}[htbp]
    \centering
\begin{overpic}[scale=0.6
]
{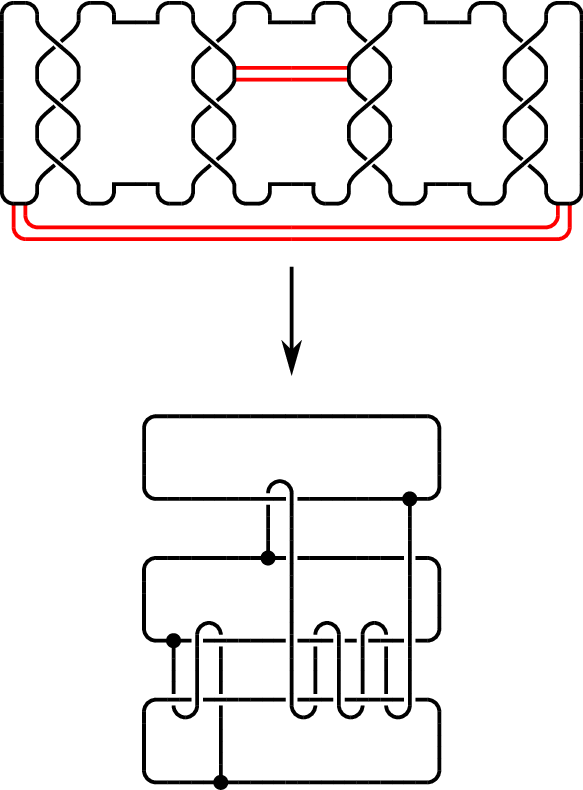}
\put(20,4){$x_3$}
\put(20,24){$x_2$}
\put(20,43){$x_1$}
\end{overpic}
\caption{Top: The knot diagram $D(3_1\#3_1\#3_1^*\#3_1^*)$ of $3_1\#3_1\#3_1^*\#3_1^*$ and ribbon bands (red). Bottom: A $\tau$-handle diagram of $\tau_{R(D(3_1\#3_1\#3_1^*\#3_1^*))}$. 
}
\label{fig:31313131andtau}
\end{figure}

It is known \cite{MR704925} that $rf(\ell) \ge m(\ell)/2$, where $m(\ell)$ is the Nakanishi index of a 1-knot $\ell$. We see from KnotInfo and \cite{MR634000} that $m(12n_{553}) = m(12n_{556}) = 3$ and $m(3_1\#3_1\#3_1^*\#3_1^*) =4$. Thus, we have that $rf(12n_{553})$, $rf(12n_{556})$ and $rf(3_1\#3_1\#3_1^*\#3_1^*) \ge 2$. 
Hence, we have that $rf(12n_{553}) = rf(12n_{556}) = rf(3_1\#3_1\#3_1^*\#3_1^*) = 2$. 
Note that it is not known whether $rf(12a_{990})$ is $1$ or $2$ (see \cite[Question 6.3]{zbMATH07379313} and \cite[Question 2]{abetange} for example). 

These contents including Example \ref{exm:12 cross} are summarized in Table \ref{tab:12 crossing ribbon 1-knot}. 
\end{rem}

\begin{prop}\label{prop:irregularcase}
There exist knot diagrams $D(12n_{553})$, $D(12n_{556})$, $D(3_1\#6_1\#3_1^*)$ and $D(3_1 \# 3_1 \# 3_1^* \# 3_1^*)$ such that the Price twists $\tau_{R(D(12n_{553}))}$, $\tau_{R(D(12n_{556}))}$, $\tau_{R(D(3_1\#6_1\#3_1^*))}$ and $\tau_{R(D(3_1 \# 3_1 \# 3_1^* \# 3_1^*))}$ are diffeomorphic to one another. 
\end{prop}

\begin{proof}
A knot diagram $D(12n_{553})$ of $12n_{553}$ and a $\tau$-handle diagram of $\tau_{R(D(12n_{553}))}$ is depicted in Figure \ref{fig:12n553}, which is obtained from Lamm's ribbon representation in \cite{lamm2021search}. 
\begin{figure}[htbp]
    \centering
\begin{overpic}[scale=0.6
]
{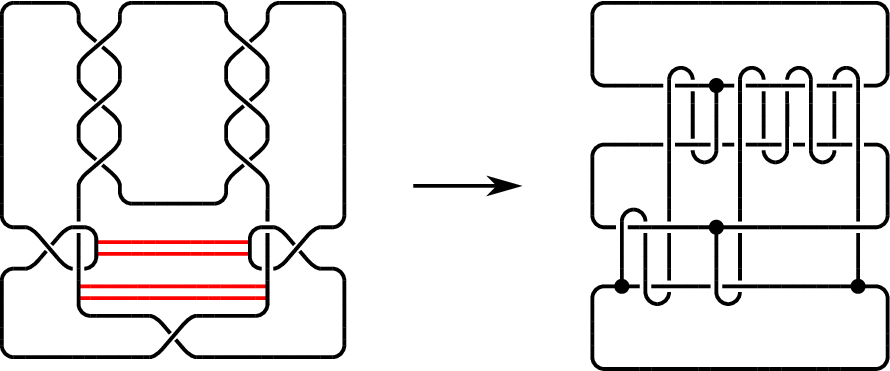}
\put(69,36){$x_1$}
\put(69,21){$x_2$}
\put(69,4){$x_3$}
\end{overpic}
\caption{Left: The knot diagram $D(12n_{553})$ of $12n_{553}$ and ribbon bands (red). Right: A $\tau$-handle diagram of $\tau_{R(D(12n_{553}))}$.}
\label{fig:12n553}
\end{figure}
From the $\tau$-handle diagram in Figure \ref{fig:12n553} and some $\tau$-handle calculus, we have 

\begin{eqnarray*}
&&\pi_1(\tau_{R(D(12n_{553}))}) \\
&\underset{\tau\text{-d.}}{=}& 
\left\langle
  \begin{lgathered} 
    x_1, x_2, \\
    x_3
  \end{lgathered} 
\, \middle| \,
  \begin{lgathered} 
    x_1^2=1, x_1(x_2x_1^{-1}x_2x_3^{-1}x_2^{-1})x_3(x_2x_1^{-1}x_2x_3^{-1}x_2^{-1})^{-1}=1, \\
    x_2(x_3x_2^{-1}x_1x_2^{-1}x_1^{-1}x_2x_1^{-1}x_2)x_3(x_3x_2^{-1}x_1x_2^{-1}x_1^{-1}x_2x_1^{-1}x_2)^{-1}=1
  \end{lgathered} 
  \right\rangle\\
&\underset{\text{s.}, \alpha, \beta, \text{i.}}{=}& 
\left\langle
    x_1, x_2, x_3
\, \middle| \,
  \begin{lgathered} 
    x_1^2=1, x_1(x_2x_1^{-1}x_2x_3^{-1}x_2^{-1})x_3(x_2x_1^{-1}x_2x_3^{-1}x_2^{-1})^{-1}=1, \\
    x_2(x_3x_2)x_3(x_3x_2)^{-1}=1
  \end{lgathered} 
  \right\rangle\\
&\underset{\text{s.}, \alpha, \beta, \text{i.}}{=}& 
\left\langle x_0,x_1,x_2 \, \middle| \,
  \begin{lgathered} 
    x_1^2=1, x_1(x_2x_1)x_2(x_2x_1)^{-1}=1, \\
    x_2(x_3x_2)x_3(x_3x_2)^{-1}=1
  \end{lgathered} 
  \right\rangle. 
\end{eqnarray*}

A knot diagram $D(12n_{556})$ of $12n_{556}$ and a $\tau$-handle diagram of $\tau_{R(D(12n_{556}))}$ is depicted in Figure \ref{fig:12n556}, which is obtained from Lamm's ribbon representation in \cite{lamm2021search}. 
\begin{figure}[htbp]
    \centering
\begin{overpic}[scale=0.6
]
{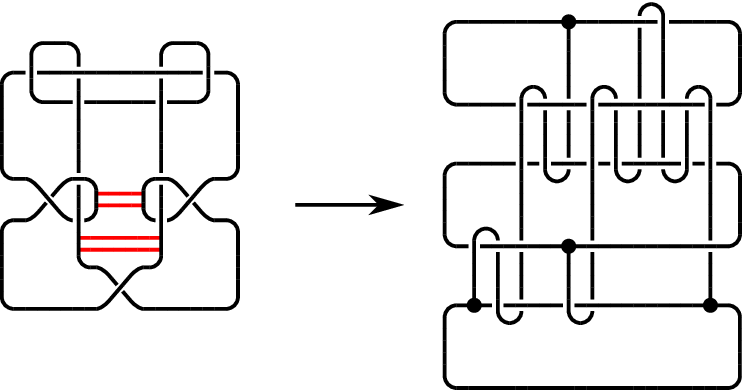}
\put(61,42){$x_1$}
\put(61,24){$x_2$}
\put(61,4){$x_3$}
\end{overpic}
\caption{Left: The knot diagram $D(12n_{556})$ of $12n_{556}$ and ribbon bands (red). Right: A $\tau$-handle diagram of $\tau_{R(D(12n_{556}))}$.}
\label{fig:12n556}
\end{figure}
From the diagram in Figure \ref{fig:12n556} and some $\tau$-handle calculus, we have 

\begin{eqnarray*}
&&\pi_1(\tau_{R(D(12n_{556}))}) \\
&\underset{\tau\text{-d.}}{=}& 
\left\langle
  \begin{lgathered} 
    x_1, x_2, \\
    x_3
  \end{lgathered} 
\, \middle| \,
  \begin{lgathered} 
    x_1^2=1, x_1(x_2^{-1}x_1^{-1}x_2x_3^{-1}x_2^{-1})x_3(x_2^{-1}x_1^{-1}x_2x_3^{-1}x_2^{-1})^{-1}=1, \\
    x_2(x_3x_2^{-1}x_1x_2x_1x_2^{-1}x_1^{-1}x_2)x_3(x_3x_2^{-1}x_1x_2x_1x_2^{-1}x_1^{-1}x_2)^{-1}=1
  \end{lgathered} 
  \right\rangle\\
&\underset{\alpha, \beta}{=}& 
\left\langle
  \begin{lgathered} 
    x_1, x_2, \\
    x_3
  \end{lgathered} 
\, \middle| \,
  \begin{lgathered} 
    x_1^2=1, x_1(x_2x_1^{-1}x_2x_3^{-1}x_2^{-1})x_3(x_2x_1^{-1}x_2x_3^{-1}x_2^{-1})^{-1}=1, \\
    x_2(x_3x_2^{-1}x_1x_2^{-1}x_1^{-1}x_2x_1^{-1}x_2)x_3(x_3x_2^{-1}x_1x_2^{-1}x_1^{-1}x_2x_1^{-1}x_2)^{-1}=1
  \end{lgathered} 
  \right\rangle\\
&\underset{\tau\text{-d.}}{=}& 
\pi_1(\tau_{R(D(12n_{553}))}). 
\end{eqnarray*}

A knot diagram $D(3_1\#6_1\#3_1^*)$ of $3_1\#6_1\#3_1^*$ and a $\tau$-handle diagram of $\tau_{R(D(3_1\#6_1\#3_1^*))}$ is depicted in Figure \ref{fig:316131}. 
From the $\tau$-handle diagram in Figure \ref{fig:316131} and some $\tau$-handle calculus, we have 

\begin{eqnarray*}
&&\pi_1(\tau_{R(D(3_1\#6_1\#3_1^*))}) \\
&\underset{\tau\text{-d.}}{=}& 
\left\langle x_1, x_2, x_3 \, \middle| \,
  \begin{lgathered} 
    x_1^2=1, x_1(x_2^{-1}x_1^{-1}x_3x_2)x_3(x_2^{-1}x_1^{-1}x_3x_2)^{-1}=1, \\
    x_2(x_3^{-1}x_1x_2x_1x_2^{-1}x_1^{-1})x_3(x_3^{-1}x_1x_2x_1x_2^{-1}x_1^{-1})^{-1}=1. 
  \end{lgathered} 
  \right\rangle\\
&\underset{\text{s.}, \alpha, \beta, \text{i.}}{=}& 
\left\langle x_1, x_2, x_3 \, \middle| \,
  \begin{lgathered} 
    x_1^2=1, x_1(x_2x_1x_3x_2)x_3(x_2x_1x_3x_2)^{-1}=1, \\
    x_2(x_3x_2)x_3(x_3x_2)^{-1}=1. 
  \end{lgathered} 
  \right\rangle\\
&\underset{\text{s.}, \alpha, \beta, \text{i.}}{=}& 
\left\langle x_1, x_2, x_3 \, \middle| \,
  \begin{lgathered} 
    x_1^2=1, x_1(x_2x_1)x_2(x_2x_1)^{-1}=1, \\
    x_2(x_3x_2)x_3(x_3x_2)^{-1}=1. 
  \end{lgathered} 
  \right\rangle. 
\end{eqnarray*}

A knot diagram $D(3_1\#3_1\#3_1^*\#3_1^*)$ of $3_1\#3_1\#3_1^*\#3_1^*$ and a $\tau$-handle diagram of $\tau_{R(D(3_1\#3_1\#3_1^*\# 3_1^*))}$ is depicted in Figure \ref{fig:31313131andtau}. 

From the $\tau$-handle diagram in Figure \ref{fig:31313131andtau} and some $\tau$-handle calculus, we have 

\begin{eqnarray*}
&&\pi_1(\tau_{R(D(3_1\#3_1\#3_1^*\#3_1^*))}) \\
&\underset{\tau\text{-d.}}{=}& 
\left\langle x_1,x_2,x_3 \, \middle| \,
  \begin{lgathered} 
    x_1^2=1, x_2(x_3^{-1}x_2)x_3(x_3^{-1}x_2)^{-1}=1, \\
    x_1(x_3x_2^{-1}x_3^{-1}x_2x_3^{-1}x_1)x_2(x_3x_2^{-1}x_3^{-1}x_2x_3^{-1}x_1)^{-1}=1
  \end{lgathered} 
  \right\rangle\\
&\underset{\text{s.}, \alpha, \beta, \text{i.}}{=}& 
\left\langle x_1,x_2,x_3 \, \middle| \,
  \begin{lgathered} 
    x_1^2=1, x_2(x_3x_2)x_3(x_3x_2)^{-1}=1, \\
    x_1(x_2x_1)x_2(x_2x_1)^{-1}=1
  \end{lgathered} 
  \right\rangle\\
&=& 
\left\langle x_1,x_2,x_3 \, \middle| \,
  \begin{lgathered} 
    x_1^2=1, x_1(x_2x_1)x_2(x_2x_1)^{-1}=1, \\
    x_2(x_3x_2)x_3(x_3x_2)^{-1}=1
  \end{lgathered} 
  \right\rangle. 
\end{eqnarray*}

These four $\tau$-presentations are the same. 
This completes the proof.
\end{proof}

\begin{rem}\label{rem:rf=2 example}
Let $m_1$, $m_2$ and $m_3$ be integers greater than or equal to $2$ or $\infty$ and $W(m_1,m_2,m_3)$ the Coxeter group 
$$\langle x_1, x_2, x_3 \mid x_1^2=x_2^2=x_3^2=1, (x_1x_2)^{m_1}= (x_2x_3)^{m_1}=(x_3x_1)^{m_3}=1\rangle,$$ 
where the relation $(x_ix_j)^{\infty}=1$ means that no relation of the form $(x_ix_j)^{m}=1$ for any integer $m \ge 2$ is imposed. 

From the proof of Proposition \ref{prop:irregularcase}, we can see that 
there exist knot diagrams $D(12n_{553})$, $D(12n_{556})$, $D(3_1\#6_1\#3_1^*)$ and $D(3_1 \# 3_1 \# 3_1^* \# 3_1^*)$ such that the Price twists $\tau_{R(D(12n_{553}))}$, $\tau_{R(D(12n_{556}))}$, $\tau_{R(D(3_1\#6_1\#3_1^*))}$ and $\tau_{R(D(3_1 \# 3_1 \# 3_1^* \# 3_1^*))}$ 
have the same $\tau$-handle diagram depicted in Figure \ref{fig:asimple2fusion}. 
One can check that the fundamental groups of these four Price twists are isomorphic to the Coxeter group $W(3,3,\infty)$ since from Proposition \ref{prop:irregularcase}, we obtain 
\begin{eqnarray*}
&&\pi_1(\tau_{R(D(12n_{553}))}) \cong \pi_1(\tau_{R(D(12n_{556}))})\\
&\cong& \pi_1(\tau_{R(D(3_1\#6_1\#3_1^*))})\cong\pi_1(\tau_{R(D(3_1\#3_1\#3_1^*\#3_1^*))})\\
&\cong& 
\left\langle x_1,x_2,x_3 \, \middle| \,
  \begin{lgathered} 
    x_1^2=1, x_1(x_2x_1)x_2(x_2x_1)^{-1}=1, \\
    x_2(x_3x_2)x_3(x_3x_2)^{-1}=1
  \end{lgathered} 
  \right\rangle\\
&\cong& 
\left\langle x_1,x_2,x_3 \, \middle| \,
  \begin{lgathered} 
    x_1^2=1, x_1^2=1, x_1^2=1, x_1(x_2x_1)x_2(x_2x_1)^{-1}=1, \\
    x_2(x_3x_2)x_3(x_3x_2)^{-1}=1
  \end{lgathered} 
  \right\rangle\\
&=& 
\left\langle x_1,x_2,x_3 \, \middle| \,
  \begin{lgathered} 
    x_1^2=x_2^2=x_3^2=1, x_1(x_2x_1)x_2(x_2x_1)^{-1}=1, \\
    x_2(x_3x_2)x_3(x_3x_2)^{-1}=1
  \end{lgathered} 
  \right\rangle\\
&=&
\langle x_1, x_2, x_3 \mid x_1^2=x_2^2=x_3^2=1, (x_1x_2)^3=(x_2x_3)^3=1\rangle=W(3,3,\infty). 
\end{eqnarray*}

\begin{figure}[htbp]
    \centering
\begin{overpic}[scale=0.6
]
{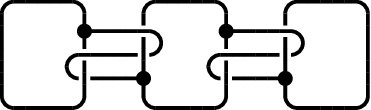}
\put(5,5){$x_1$}
\put(45,5){$x_2$}
\put(84,5){$x_3$}
\end{overpic}
\caption{A simple $\tau$-handle diagram.}
\label{fig:asimple2fusion}
\end{figure}

Since $W(3,3,\infty)$ is an infinite group, $W(3,3,\infty)$ is not isomorphic to the dihedral group $D_{n}$ for any positive integer $n$. 
Thus, we obtain 
\begin{eqnarray*}
&&rf(R(D(12n_{553})))=rf(R(D(12n_{556})))\\
&=&rf(R(D(3_1\#6_1\#3_1^*)))=rf(R(D(3_1\#3_1\#3_1^*\#3_1^*)))=2
\end{eqnarray*}
from Proposition \ref{prop:pao} and Theorems \ref{thm:pi1oftaut22n+1} and \ref{thm:1-fusion} (see also Table \ref{tab:12 crossing ribbon 1-knot}). 

This implies that Proposition \ref{prop:pao} and Theorems \ref{thm:pi1oftaut22n+1} and \ref{thm:1-fusion} provide one approach to proving that the fusion number of a ribbon $2$-knot is $2$. 
\end{rem}

\begin{rem}\label{rem:12a990}
Here we present an example other than Remark \ref{rem:rf=2 example}, in which we can determine that the fusion number of a $2$-knot is $2$. 

A knot diagram $D(12a_{990})$ of $12a_{990}$ and a $\tau$-handle diagram of $\tau_{R(D(12a_{990}))}$ is depicted in Figure \ref{fig:12a990}, which is obtained from Lamm's ribbon representation in \cite{lamm2021search}. 
From the knot diagram in Figure \ref{fig:12a990}, we have $rf(R(D(12a_{990}))) \le 2$. 
From the $\tau$-handle diagram in Figure \ref{fig:12a990} and some $\tau$-handle calculus, we have 
\begin{eqnarray*}
&&\pi_1(\tau_{R(D(12a_{990}))}) \\
&\underset{\tau\text{-d.}}{=}& 
\left\langle x_1,x_2,x_3 \, \middle| \,
  \begin{lgathered} 
    x_1^2=1, x_1(x_3^{-1}x_1x_3^{-1})x_3(x_3^{-1}x_1x_3^{-1})^{-1}=1, \\
    x_2(x_1^{-1}x_2^{-1}x_3^{-1}x_2)x_3(x_1^{-1}x_2^{-1}x_3^{-1}x_2)^{-1}=1
  \end{lgathered} 
  \right\rangle\\
&\underset{i.,\alpha, \beta}{=}& 
\left\langle x_1,x_2,x_3 \, \middle| \,
  \begin{lgathered} 
    x_1^2=1, x_1(x_3x_1)x_3(x_3x_1)^{-1}=1, \\
    x_2(x_1x_2x_3x_2)x_3(x_1x_2x_3x_2)^{-1}=1
  \end{lgathered} 
  \right\rangle\\
&\cong&
\left\langle x_1,x_2,x_3 \, \middle| \, 
x_1^2=x_2^2=x_3^2=1, (x_1x_2)^2=(x_2x_3)^3, (x_1x_3)^3=1
\right\rangle. 
\end{eqnarray*}

One can check that $\pi_1(\tau_{R(D(12a_{990}))})$ is not isomorphic to $D_{\abs{2n+1}}$ for any integer $n$. 
Thus, we obtain $rf(R(D(12a_{990})))=2$ from Proposition \ref{prop:pao} and Theorems \ref{thm:pi1oftaut22n+1} and \ref{thm:1-fusion} (see also Table \ref{tab:12 crossing ribbon 1-knot}). 

\begin{figure}[htbp]
    \centering
\begin{overpic}[scale=0.6
]
{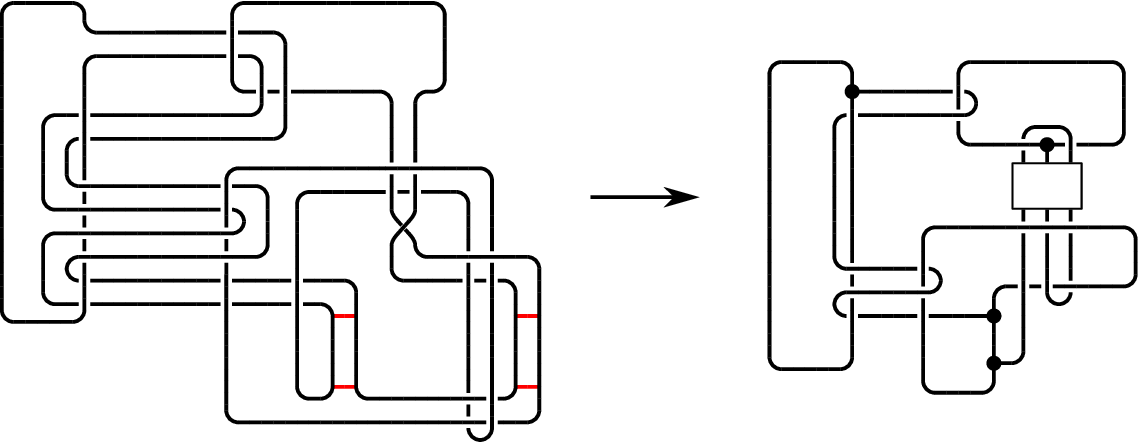}
\put(91.3,21.5){$1$}

\put(90,30){$x_1$}
\put(69,30){$x_2$}
\put(84,16){$x_3$}
\end{overpic}
\caption{Left: The knot diagram $D(12a_{990})$ of $12a_{990}$ and ribbon bands (red). Right: A $\tau$-handle diagram of $\tau_{R(D(12a_{990}))}$.}
\label{fig:12a990}
\end{figure}
\end{rem}

\begin{rem}\label{rem:fusion 2 case} 
We claim that there exist knot diagrams $D_1(k)$ and $D_2(k)$ of the same ribbon $1$-knot $k$ such that $R(D_1(k))$ and $R(D_2(k))$ do not have the same fusion number. \\
\noindent
(1) \ Let $D_1(10_{99})$ be a knot diagram of $10_{99}$, which is obtained from Kawauchi's ribbon representation in \cite{zbMATH00795683}. 
From Example \ref{exm:12 cross}, we obtain $rf(R(D_1(10_{99})))=1$ and $\tau_{R(D_1(10_{99}))}$ is diffeomorphic to $\tau_{S(T_{2,9})}$.  
Therefore, $\pi_1(\tau_{R(D_1(10_{99}))})$ is isomorphic to $D_{9}$ from Theorem \ref{thm:pi1oftaut22n+1}. 

Let $D_2(10_{99})$ be the knot diagram of $10_{99}$ depicted in the left side of Figure \ref{fig:1099}, which is obtained from Kishimoto-Shibuya-Tsukamoto-Ishikawa's ribbon representation in \cite{kishimoto2021alexander}.  
A $\tau$-handle diagram of $\tau_{R(D_2(10_{99}))}$ is depicted in the right side of Figure \ref{fig:1099}. 

\begin{figure}[htbp]
    \centering
\begin{overpic}[scale=0.6
]
{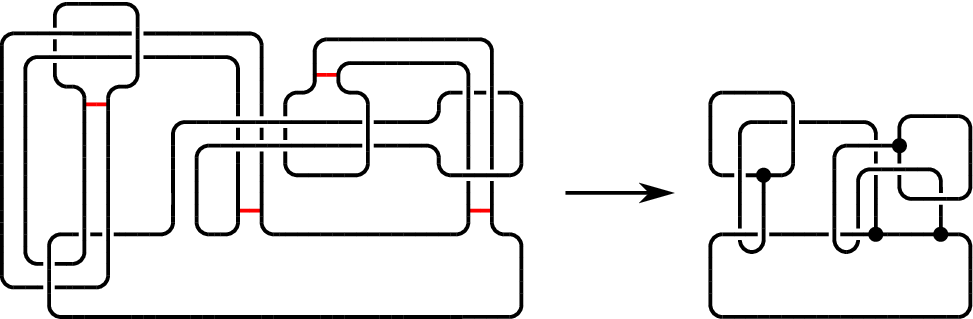}
\put(77,17){$x_1$}
\put(75,3){$x_2$}
\put(95,17){$x_3$}
\end{overpic}
\caption{Left: The knot diagram $D_2(10_{99})$ of $10_{99}$ and ribbon bands (red). Right: A $\tau$-handle diagram of $\tau_{R(D_2(10_{99}))}$.}
\label{fig:1099}
\end{figure}

From the $\tau$-handle diagram in Figure \ref{fig:1099} and some $\tau$-handle calculus, we have 
\begin{eqnarray*}
&&\pi_1(\tau_{R(D_2(10_{99}))}) \\
&\underset{\tau\text{-d.}}{=}& 
\left\langle x_1,x_2,x_3 \, \middle| \,
  \begin{lgathered} 
    x_1^2=1, x_1(x_2x_1)x_2(x_2x_1)^{-1}=1, \\
    x_2(x_3^{-1}x_2^{-1})x_3(x_3^{-1}x_2^{-1})^{-1}=1
  \end{lgathered} 
  \right\rangle\\
&\underset{\alpha, \beta}{=}& 
\left\langle x_1,x_2,x_3 \, \middle| \,
  \begin{lgathered} 
    x_1^2=1, x_1(x_2x_1)x_2(x_2x_1)^{-1}=1, \\
    x_2(x_3x_2)x_3(x_3x_2)^{-1}=1
  \end{lgathered} 
  \right\rangle\\
&\cong&
\left\langle x_1,x_2,x_3 \, \middle| \, 
x_1^2=x_2^2=x_3^2=1, (x_1x_2)^3=(x_2x_3)^3=1
\right\rangle
=
W(3,3,\infty)
. 
\end{eqnarray*}
One can check that $\pi_1(\tau_{R(D_2(10_{99}))})$ is not isomorphic to $D_{\abs{2n+1}}$ for any integer $n$. 
Thus, we obtain $rf(R(D_2(10_{99})))=2$ from Proposition \ref{prop:pao} and Theorems \ref{thm:pi1oftaut22n+1} and \ref{thm:1-fusion}. 
Then, we see that $\tau_{R(D_1(10_{99}))}$ is not homotopy equivalent to $\tau_{R(D_2(10_{99}))}$ and the $2$-knots $R(D_1(10_{99}))$ and $R(D_2(10_{99}))$ are not isotopic. 
Note that $\tau_{R(D_2(10_{99}))}$ have the $\tau$-handle diagram depicted in Figure \ref{fig:asimple2fusion}. 

\noindent
(2) \ Let $D_1(12a_{427})$ be a knot diagram of $12a_{427}$, which is obtained from the ribbon representation in \cite{arXiv:2409.12910}. 
From Example \ref{exm:12 cross}, we obtain $rf(R(D_1(12a_{427})))=1$ and $\tau_{R(D_1(12a_{427}))}$ is diffeomorphic to $\tau_{S(T_{2,15})}$.  
Therefore, $\pi_1(\tau_{R(D_1(12a_{427}))})$ is isomorphic to $D_{15}$ from Theorem \ref{thm:pi1oftaut22n+1}. 

Let $D_2(12a_{427})$ be the knot diagram of $12a_{427}$ depicted in the left side of Figure \ref{fig:12a427}, which is obtained from Lamm's ribbon representation in \cite{lamm2021search}.  
A $\tau$-handle diagram of $\tau_{R(D_2(12a_{427}))}$ is depicted in the right side of Figure \ref{fig:12a427}. 

\begin{figure}[htbp]
    \centering
\begin{overpic}[scale=0.6
]
{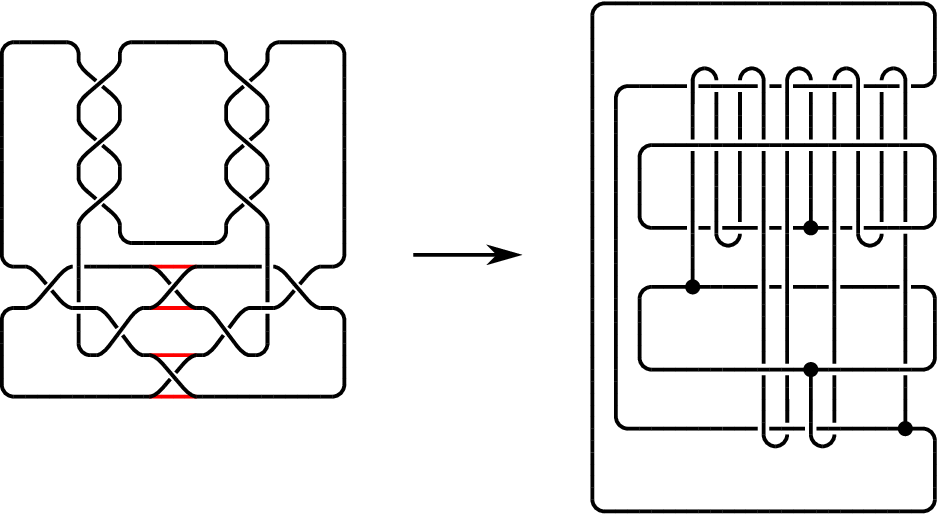}
\put(69,49){$x_1$}
\put(69,33){$x_2$}
\put(69,18){$x_3$}
\end{overpic}
\caption{Left: The knot diagram $D_2(12a_{427})$ of $12a_{427}$ and ribbon bands (red). Right: A $\tau$-handle diagram of $\tau_{R(D_2(12a_{427}))}$.}
\label{fig:12a427}
\end{figure}

From the $\tau$-handle diagram in Figure \ref{fig:12a427} and some $\tau$-handle calculus, we have 
\begin{eqnarray*}
&&\pi_1(\tau_{R(D_2(12a_{427}))}) \\
&\underset{\tau\text{-d.}}{=}& 
\left\langle 
  \begin{lgathered} 
    x_1, \\
    x_2, \\
    x_3
  \end{lgathered} 
\, \middle| \,
  \begin{lgathered} 
    x_1^2=1, x_1(x_3^{-1}x_2x_1x_2x_1^{-1}x_2^{-1}x_3x_1^{-1}) 
    x_3(x_3^{-1}x_2x_1x_2x_1^{-1}x_2^{-1}x_3x_1^{-1})^{-1}=1, \\
    x_2(x_1^{-1}x_2^{-1}x_3x_1^{-1}x_3^{-1}x_2x_1x_2x_1^{-1}x_2^{-1})\\
    \cdot x_3(x_1^{-1}x_2^{-1}x_3x_1^{-1}x_3^{-1}x_2x_1x_2x_1^{-1}x_2^{-1})^{-1}=1
  \end{lgathered} 
  \right\rangle\\
&\cong&
\left\langle 
    x_1, 
    x_2, 
    x_3
\, \middle| \,
  \begin{lgathered} 
    x_1^2=x_2^2=x_3^2=1, \\
    (x_1x_2)^3=(x_1x_3)^5=1
  \end{lgathered} 
  \right\rangle
= W(3,5,\infty). 
\end{eqnarray*}
One can check that $\pi_1(\tau_{R(D_2(12a_{427}))})$ is not isomorphic to $D_{\abs{2n+1}}$ for any integer $n$. 
Thus, we obtain $rf(R(D_2(12a_{427})))=2$ from Proposition \ref{prop:pao} and Theorems \ref{thm:pi1oftaut22n+1} and \ref{thm:1-fusion}. 
Then, we see that $\tau_{R(D_1(12a_{427}))}$ is not homotopy equivalent to $\tau_{R(D_2(12a_{427}))}$ and the $2$-knots $R(D_1(12a_{427}))$ and $R(D_2(12a_{427}))$ are not isotopic.

\noindent
(3) \ Let $D_1(12a_{1225})$ be a knot diagram of $12a_{1225}$, which is obtained from Miller's ribbon representation in \cite{zbMATH07379313}. 
From Example \ref{exm:12 cross}, we obtain $rf(R(D_1(12a_{1225})))=1$ and $\tau_{R(D_1(12a_{1225}))}$ is diffeomorphic to $\tau_{S(T_{2,15})}$.  
Therefore, $\pi_1(\tau_{R(D_1(12a_{1225}))})$ is isomorphic to $D_{15}$ from Theorem \ref{thm:pi1oftaut22n+1}. 

Let $D_2(12a_{1225})$ be the knot diagram of $12a_{1225}$ depicted in the left side of Figure \ref{fig:12a1225}, which is obtained from Lamm's ribbon representation in \cite{lamm2021search}.  
A $\tau$-handle diagram of $\tau_{R(D_2(12a_{1225}))}$ is depicted in the right side of Figure \ref{fig:12a1225}. 

\begin{figure}[htbp]
    \centering
\begin{overpic}[scale=0.6
]
{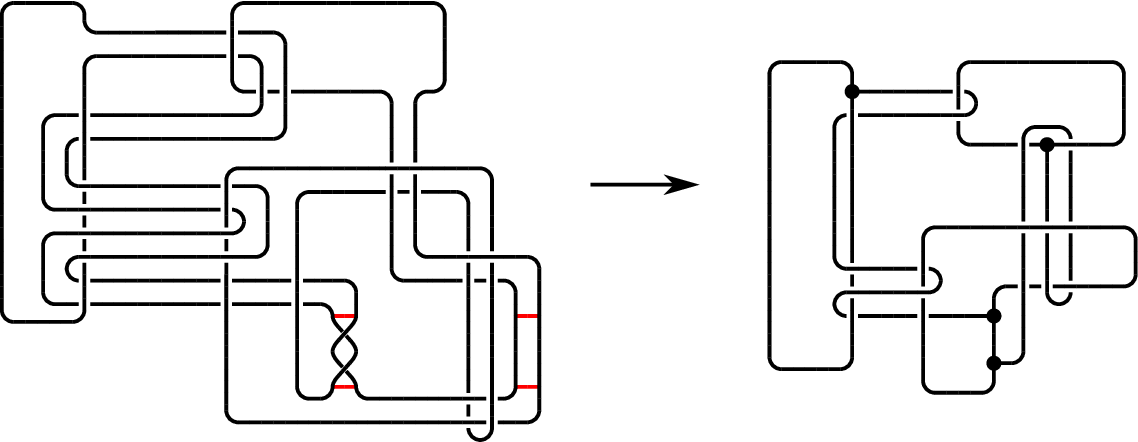}
\put(90,30){$x_1$}
\put(69,30){$x_2$}
\put(84,16){$x_3$}
\end{overpic}
\caption{Left: The knot diagram $D_2(12a_{1225})$ of $12a_{1225}$ and ribbon bands (red). Right: A $\tau$-handle diagram of $\tau_{R(D_2(12a_{1225}))}$.}
\label{fig:12a1225}
\end{figure}

From the $\tau$-handle diagram in Figure \ref{fig:12a1225} and some $\tau$-handle calculus, we have 
\begin{eqnarray*}
&&\pi_1(\tau_{R(D_2(12a_{1225}))}) \\
&\underset{\tau\text{-d.}}{=}& 
\left\langle x_1,x_2,x_3 \, \middle| \,
  \begin{lgathered} 
    x_1^2=1, x_1(x_3^{-1}x_1^{-1}x_3^{-1})x_3(x_3^{-1}x_1^{-1}x_3^{-1})^{-1}=1, \\
    x_2(x_1^{-1}x_2^{-1}x_3^{-1}x_2)x_3(x_1^{-1}x_2^{-1}x_3^{-1}x_2)^{-1}=1
  \end{lgathered} 
  \right\rangle\\
&\underset{i.,\alpha, \beta}{=}& 
\left\langle x_1,x_2,x_3 \, \middle| \,
  \begin{lgathered} 
    x_1^2=1, x_1(x_3x_1)x_3(x_3x_1)^{-1}=1, \\
    x_2(x_1x_2x_3x_2)x_3(x_1x_2x_3x_2)^{-1}=1
  \end{lgathered} 
  \right\rangle\\
&\cong&
\left\langle x_1,x_2,x_3 \, \middle| \, 
x_1^2=x_2^2=x_3^2=1, (x_1x_2)^2=(x_2x_3)^3, (x_1x_3)^3=1
\right\rangle. 
\end{eqnarray*}
One can check that $\pi_1(\tau_{R(D_2(12a_{1225}))})$ is not isomorphic to $D_{\abs{2n+1}}$ for any integer $n$. 
Thus, we obtain $rf(R(D_2(12a_{1225})))=2$ from Proposition \ref{prop:pao} and Theorems \ref{thm:pi1oftaut22n+1} and \ref{thm:1-fusion}. 
Then, we see that $\tau_{R(D_1(12a_{1225}))}$ is not homotopy equivalent to $\tau_{R(D_2(12a_{1225}))}$ and the $2$-knots $R(D_1(12a_{1225}))$ and $R(D_2(12a_{1225}))$ are not isotopic. 
Note that $\tau_{R(D_2(12a_{1225}))}$ is diffeomorphic to $\tau_{R(D(12a_{990}))}$ in Remark \ref{rem:12a990}. 
\end{rem}

Let $p$ and $q$ be integers with $\mathrm{gcd}(p,q)=1$ and $1<p<q$. 
It is known \cite[Theorem 1]{zbMATH01117469} that $rf(S(T_{p,q}))=\min\{p,q\}-1=p-1$.

\begin{prop}\label{prop:S(Tp,q)}
Figure \ref{fig:S(Tp,q)} is a $\tau$-handle diagram of $\tau_{S(T_{p,q})}$, where $\alpha_{p,q}$ in Figure \ref{fig:S(Tp,q)} is the remainder when we divide $q$ by $p$.
\end{prop}

\begin{figure}[htbp]
    \centering
\begin{overpic}[scale=0.6]
{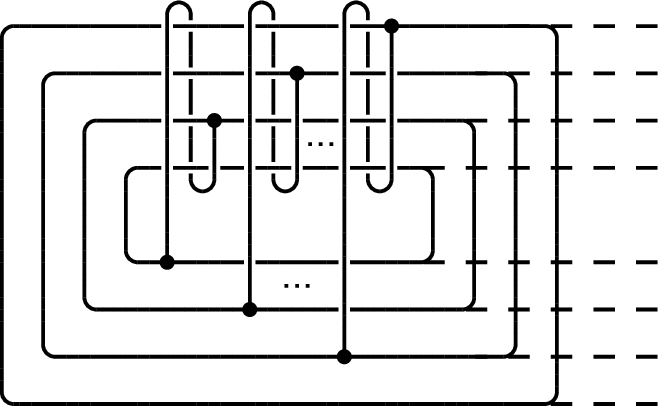}
\put(103,57){$a_{\alpha_{p,q}+p-1}$}
\put(103,50){$a_{\alpha_{p,q}+2}$}
\put(103,43){$a_{\alpha_{p,q}+1}$}
\put(103,36){$a_1$}

\put(103,21){$a_1$}
\put(103,14){$a_2$}
\put(103,7){$a_{p-1}$}
\put(103,0){$a_p$}
\end{overpic}
\caption{A $\tau$-handle diagram of $\tau_{S(T_{p,q})}$. 
The words in the fundamental group $\pi_1(\tau_{S(T_{p,q})})$ are read under the assumption that the circle corresponding to $a_i$ lies below that of $a_{i+1}$.
}
\label{fig:S(Tp,q)}
\end{figure}

\begin{proof}
A knot diagram $D(T_{p,q} \# T_{p,q}^*)$ in Figure \ref{fig:Tp,qTp,q} has a ribbon presentation depicted in Figure \ref{fig:Tp,qTp,q}, where the tangle $T$ in Figure \ref{fig:Tp,qTp,q} is defined by $(\prod_{i=1}^{p-1}\sigma_{i})^q$ and $T^*$ is the mirror image of $T$. 
Then, the $2$-knot $R(D(T_{p,q} \# T_{p,q}^*))$ is isotopic to $S(T_{p,q})$ and a $\tau$-handle diagram of $\tau_{R(D(T_{p,q} \# T_{p,q}^*))}$ depicted in Figure \ref{fig:S(Tp,q)} is obtained from Figure \ref{fig:Tp,qTp,q}.  
\end{proof}

\begin{figure}[htbp]
    \centering
\begin{overpic}[scale=0.6]
{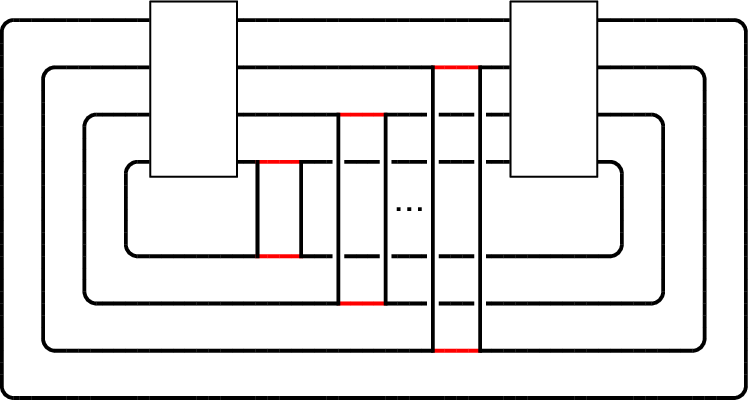}
\put(24,40){$T$}
\put(72,40){$T^*$}
\end{overpic}
\caption{A knot diagram of $T_{p,q}\#T_{p,q}^{*}$ and ribbon bands (red).}
\label{fig:Tp,qTp,q}
\end{figure}

Let $l_{p,q}$ be the quotient when we divide $q$ by $p$. 
From Proposition \ref{prop:S(Tp,q)}, a $\tau$-presentation of $\tau_{S(T_{p,q})}$ is obtained  from Figure \ref{fig:S(Tp,q)} as follows:
$$
\pi_1(\tau_{S(T_{p,q})})
\underset{\text{$\tau$-d.}}{=}
\left\langle  
    a_1, \ldots, a_p
\, \middle| \, 
  \begin{lgathered} 
    a_1^2=1, a_kw_{p,q}a_{k+\alpha_{p,q}}w_{p,q}^{-1}=1 \ (k=1,\ldots,p-1)\\
    \text{The index of each $a_i$ is taken modulo $p$.}
 \end{lgathered} 
\right\rangle,
$$
where 
$$w_{p,q}=\left(\prod_{i=1}^pa_{i+k}\right)^{l_{p,q}}\prod_{j=1}^{\alpha_{p,q}-1}a_{k+j}.$$

Thus, the following question naturally arises: 

\begin{que}\label{que:coxeter}
Is the fundamental group of $\tau_{S(T_{p,q})}$, a Coxeter group?
\end{que}

Note that the dihedral group $D_{\abs{2n+1}}$ that is the fundamental group of $\tau_{S(T_{2,2n+1})}$ is also a Coxeter group. By considering Theorem \ref{thm:1-fusion}, we ask the following question furthermore:

\begin{que}\label{que:n-fusion}
Let $K$ be a ribbon 2-knot of $n$-fusion for $n \ge 2$. Is $\tau_K$ diffeomorphic to $\tau_{S(T_{n+1,m})}$ for some integer $m \ge n+1$?
\end{que}

\subsection{Double coverings of some non-simply connected Price twists}
In this subsection, we study a double covering of the Price twist $\tau_{S(T_{2,2n+1})}$.
Since the dihedral group $D_{\abs{2n+1}}$ has only one subgroup $\mathbb{Z}_{\abs{2n+1}}$ of index 2, there exists only one double (cyclic) covering of $\tau_{S(T_{2,2n+1})}$ up to homeomorphism from Theorem \ref{thm:pi1oftaut22n+1}. 
Let $h: \Sigma_2(\tau_{S(T_{2,2n+1})}) \to \tau_{S(T_{2,2n+1})}$ be a double covering of $\tau_{S(T_{2,2n+1})}$. 
Then, the group $h_\#(\pi_1(\Sigma_2(\tau_{S(T_{2,2n+1})}))$ is the subgroup of index 2 in $\pi_1(\tau_{S(T_{2,2n+1})}) \cong D_{\abs{2n+1}}$, where $h_\# : \pi_1(\Sigma_2(\tau_{S(T_{2,2n+1})}) \to \pi_1(\tau_{S(T_{2,2n+1})})$ is the induced homomorphism of the covering $h$. 
Thus, $h_\#(\pi_1(\Sigma_2(\tau_{S(T_{2,2n+1})}))$ is isomorphic to $\mathbb{Z}_{\abs{2n+1}}$. 
Note that $\mathbb{Z}_{\abs{2n+1}}$ is the fundamental group of the Pao manifold $L_{2n+1}$.

\begin{prop}\label{prop:covering}
There exists a double cover $\Sigma_2(\tau_{S(T_{2,2n+1})})$ of $\tau_{S(T_{2,2n+1})}$ such that $\Sigma_2(\tau_{S(T_{2,2n+1})})$ is diffeomorphic to $L_{2n+1} \# S^2 \times S^2$. 
\end{prop}

\begin{proof}
It suffices to show the statement in the case where $n \ge 0$.

First, we prove the case where $n=0$. 
Since the 2-knot $S(T_{2,1})$ is isotopic to the unknotted 2-knot $O$, we see that the Price twist $\tau_{S(T_{2,1})}$ is diffeomorphic to $\tau_{O}$. 
Thus, a handle diagram of a double covering $\Sigma_2(\tau_{S(T_{2,1})})$ shown in the top left of Figure \ref{fig:double covering0} is obtained from the handle diagram of $\tau_{O}$ depicted in Figure \ref{fig:anotherspin21} by \cite[Subsection 6.3]{gompf20234}.
Then, by performing handle calculus described in
Figure \ref{fig:double covering0}, we see that $\Sigma_2(\tau_{S(T_{2,1})})$ is diffeomorphic to $S^2 \times S^2$. Note that $L_1$ is diffeomorphic to $S^4$.
\begin{figure}[htbp]
    \centering
\begin{overpic}[scale=0.6
]
{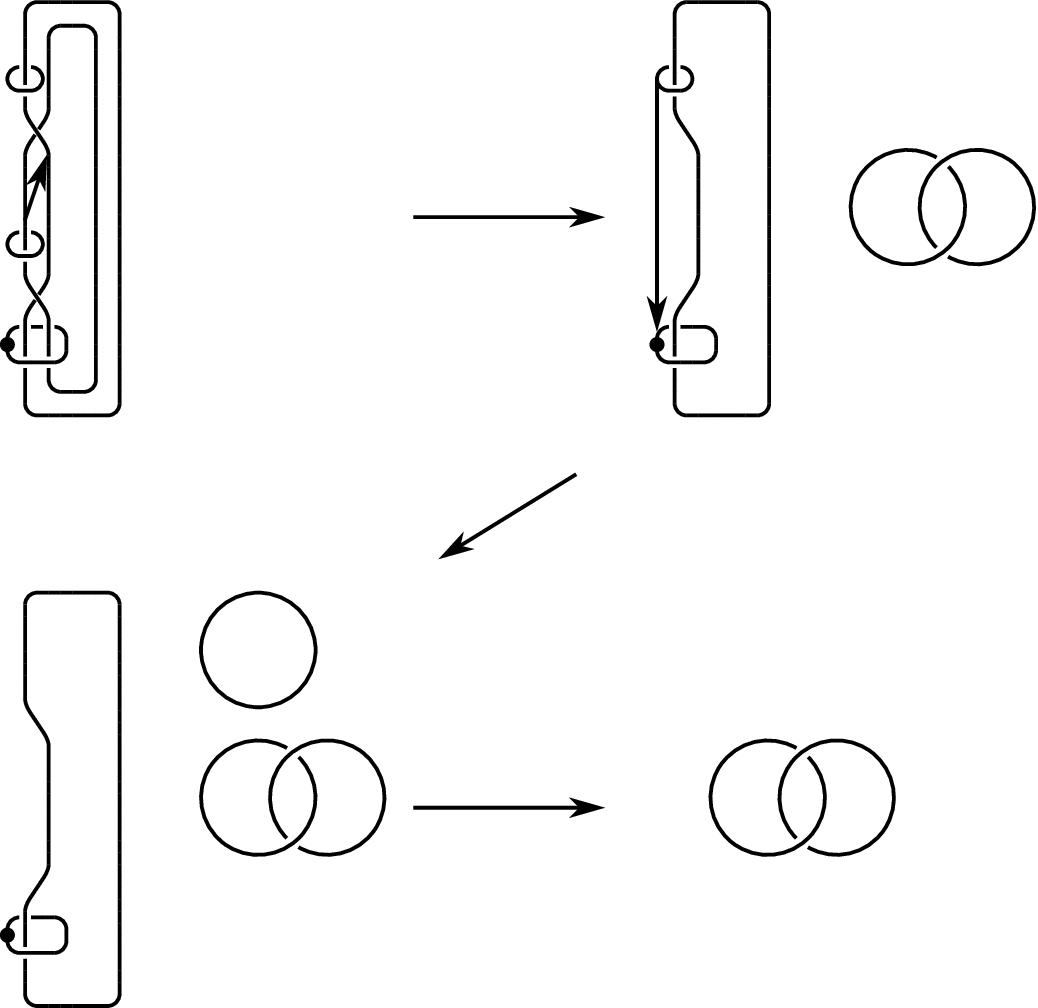}
    \put(-2,88){$0$}
    \put(0,92){$0$}
    \put(-2,72){$0$}
    \put(0,76){$0$}

    \put(20,64){$\cup$ 3-handle}
    \put(23.5,60){4-handle}

    \put(45,78){slides}

    \put(60.5,88){$0$}
    \put(62.5,92){$0$}

    \put(80,80){$0$}
    \put(100,80){$0$}

    \put(80,64){$\cup$ 3-handle}
    \put(83.5,60){4-handle}

    \put(43,50){slide}

    \put(0,34){$0$}

    \put(30,39){$0$}

    \put(17,23){$0$}
    \put(37,23){$0$}

    \put(20,4){$\cup$ 3-handle}
    \put(23.5,0){4-handle}

    \put(43,21){cancels}

    \put(66,23){$0$}
    \put(86,23){$0$}
    
    \put(80,4){$\cup$ 4-handle}
    \end{overpic}
\caption{Handle calculus in the proof for the case where $n=0$. 
In the first calculus (i.e. the first slide), we use several handle slides on a 0-framed meridian in the top left diagram. 
Each handle diagram is a handle diagram of $\Sigma_2(\tau_{S(T_{2,1})})$.}
\label{fig:double covering0}
\end{figure}

Next, we prove the case where $n>0$. 
By Proposition \ref{prop:alpha}, the handle diagram of $\tau_{S(T_{2,2n+1})}$ in the left side of Figure \ref{fig:S(T_{2,2n+1})second} can be changed to that depicted in Figure \ref{fig:S(T_{2,2n+1})third}. 

\begin{figure}[htbp]
    \centering
\begin{overpic}[scale=0.6
]
{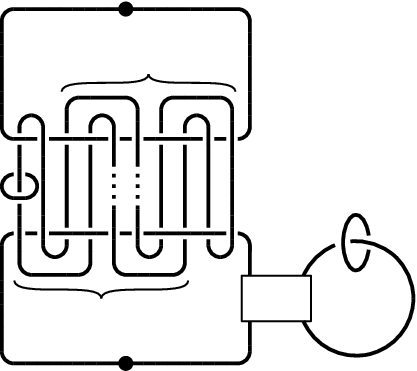}
    \put(-5,42){$0$}
    \put(59,45){$0$}
    \put(86,39){$0$}
    \put(97,26){$0$}
    \put(64,15){$2$}
    \put(33,73){$n$}
    \put(22,11){$n$}

    \put(113,49){$\cup$ $2$ $3$-handles}
    \put(128,40){$4$-handle}
    \end{overpic}
\caption{A handle diagram of the Price twist $\tau_{S(T_{2,2n+1})}$. }
\label{fig:S(T_{2,2n+1})third}
\end{figure}

By several handle slides on a 0-framed meridian, the handle diagram of $\tau_{S(T_{2,2n+1})}$ depicted in Figure \ref{fig:S(T_{2,2n+1})third} can be changed to that depicted in Figure \ref{fig:double coveringpre}. 
Then, a handle diagram of $\Sigma_2(\tau_{S(T_{2,2n+1})})$ shown in Figure \ref{fig:double covering1} is obtained from the handle diagram in Figure \ref{fig:double coveringpre} by \cite[Subsection 6.3]{gompf20234}. 
By the handle slide indicated in Figure \ref{fig:double covering1} and several handle slides on a 0-framed meridian in Figure \ref{fig:double covering1}, we obtain the handle diagram depicted in Figure \ref{fig:double covering2}. 
By canceling the pair of the leftmost string and the dotted circle, and the pair of the leftmost 0-framed meridian and a 3-handle in Figure \ref{fig:double covering2}, we obtain the handle diagram depicted in Figure \ref{fig:double covering3}. 
By several handle slides on 0-framed meridians and canceling the pair of the leftmost 0-framed knot and a 3-handle in Figure \ref{fig:double covering3}, we obtain the handle diagram depicted in Figure \ref{fig:double covering4}. 
By the handle slide indicated in Figure \ref{fig:double covering4}, several handle slides on 0-framed meridians and canceling the pair of the rightmost dotted circle and a framed knot in Figure \ref{fig:double covering4}, we obtain the handle diagram depicted in Figure \ref{fig:double covering5}. 
This handle diagram describes $L_{2n+1} \# S^2 \times S^2$ (see Figure \ref{fig:Pao}).
\end{proof}

\begin{figure}[htbp]
    \centering
\begin{overpic}[scale=0.6
]
{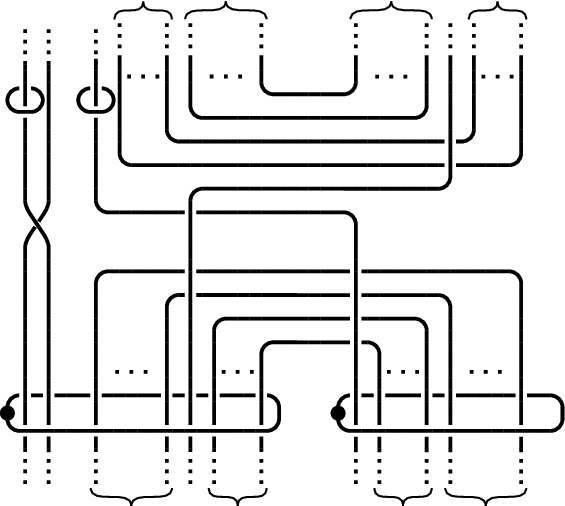}
    \put(18,91){$n-1$}
    \put(33,91){$n+1$}
    \put(62,91){$n+1$}
    \put(81,91){$n-1$}

    \put(-3,70){$0$}
    \put(10,70){$0$}

    \put(0,60){$0$}
    \put(13,60){$0$}

    \put(21,-5){$n$}
    \put(40,-5){$n$}
    \put(69,-5){$n$}
    \put(84,-5){$n$}

    \put(105,46){$\cup$ 2 3-handles}
    \put(116.5,39){4-handle}
    \end{overpic}
\caption{Another handle diagram of $\tau_{S(T_{2,2n+1})}$. The strings at the top and bottom are identified starting from the left end.}
\label{fig:double coveringpre}
\end{figure}

\begin{figure}[htbp]
    \centering
\begin{overpic}[scale=0.6
]
{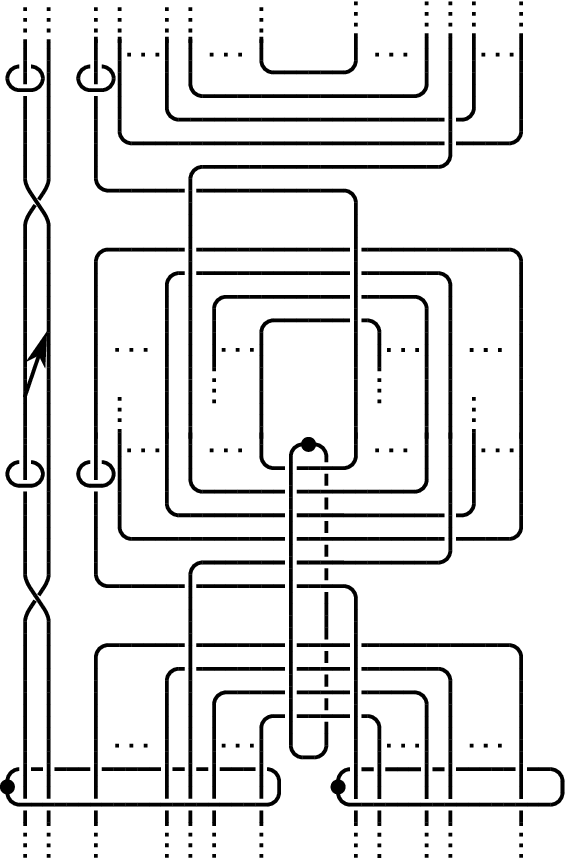}
    \put(-2,90){$0$}
    \put(6.5,90){$0$}

    \put(0,80){$0$}
    \put(8.5,80){$0$}

    \put(-2,43){$0$}
    \put(6.5,43){$0$}

    \put(0,33){$0$}
    \put(8.5,33){$0$}

    \put(72,19){$\cup$ 3 3-handles}
    \put(79.5,14){4-handle}
    \end{overpic}
\caption{A handle diagram of $\Sigma_2(\tau_{S(T_{2,2n+1})})$.}
\label{fig:double covering1}
\end{figure}

\begin{figure}[htbp]
    \centering
\begin{overpic}[scale=0.6
]
{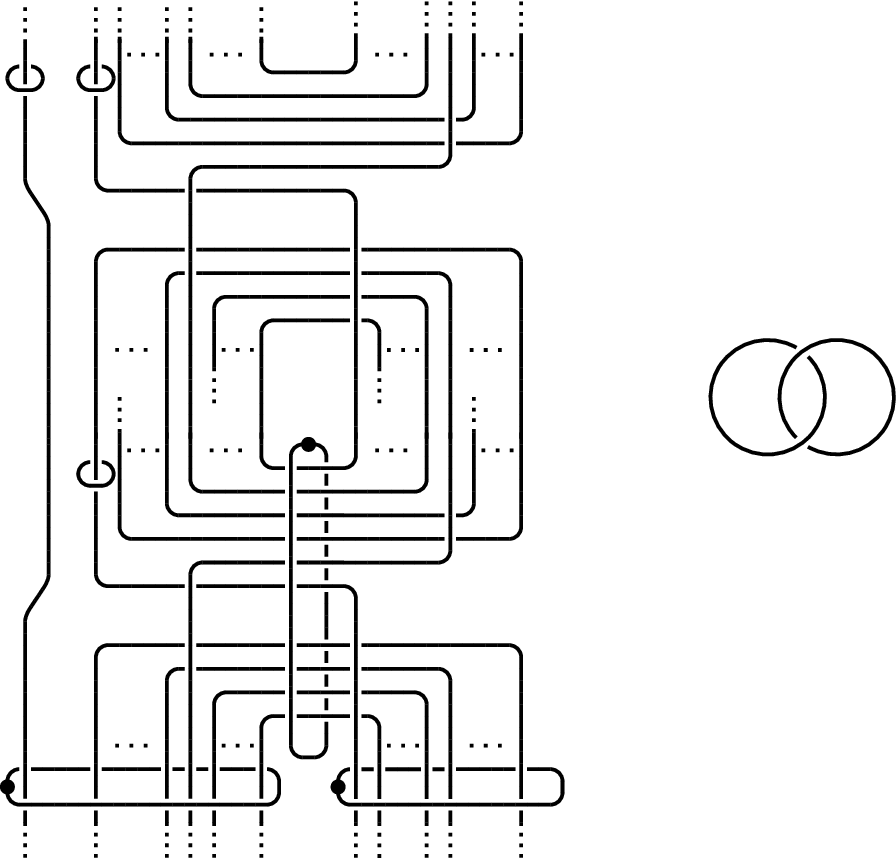}
    \put(-2,87){$0$}
    \put(6.5,87){$0$}

    \put(0,77){$0$}
    \put(8,77){$0$}

    \put(6.5,43){$0$}

    \put(8,33){$0$}

    \put(77,56){$0$}
    \put(99,56){$0$}

    \put(72,19){$\cup$ 3 3-handles}
    \put(79.5,14){4-handle}
    \end{overpic}
\caption{A handle diagram of $\Sigma_2(\tau_{S(T_{2,2n+1})})$ obtained from the handle slide indicated in Figure \ref{fig:double covering1} and several handle slides on a 0-framed meridian in Figure \ref{fig:double covering1}.}
\label{fig:double covering2}
\end{figure}

\begin{figure}[htbp]
    \centering
\begin{overpic}[scale=0.6
]
{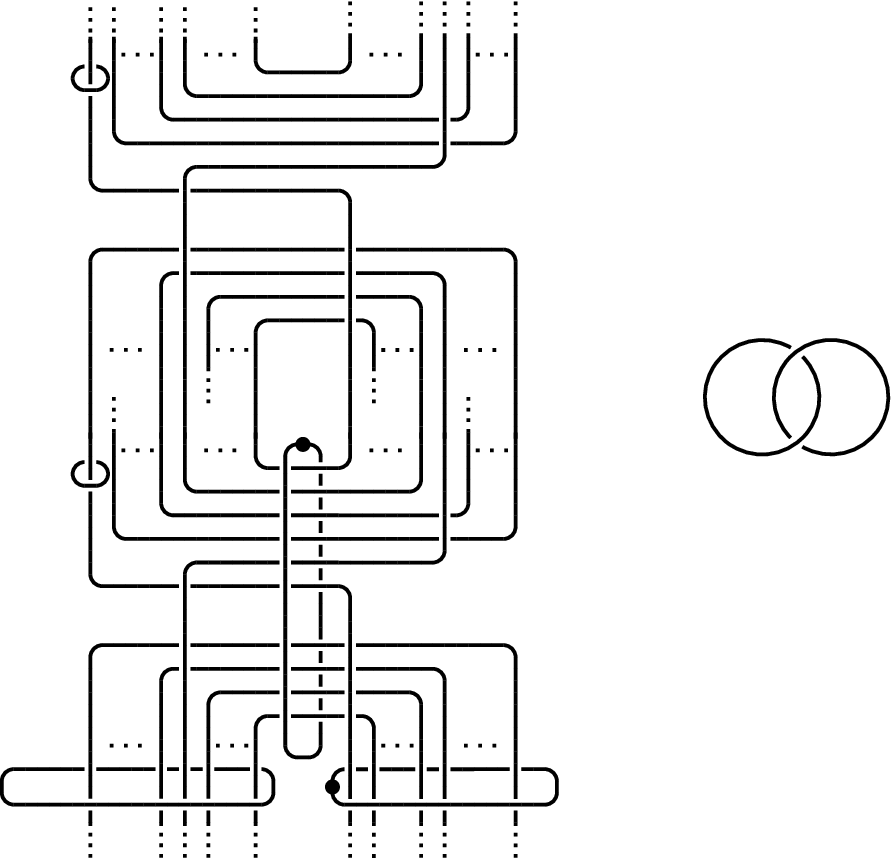}
    \put(5,87){$0$}

    \put(7,77){$0$}

    \put(5,43){$0$}

    \put(7,33){$0$}

    \put(77,56){$0$}
    \put(99,56){$0$}

    \put(5,11){$0$}

    \put(72,19){$\cup$ 2 3-handles}
    \put(79.5,14){4-handle}
    \end{overpic}
\caption{A handle diagram of $\Sigma_2(\tau_{S(T_{2,2n+1})})$ obtained from Figure \ref{fig:double covering2} by canceling the pair of the leftmost string and the dotted circle, and the pair of the leftmost $0$-framed meridian and a $3$-handle.}
\label{fig:double covering3}
\end{figure}

\begin{figure}[htbp]
    \centering
\begin{overpic}[scale=0.6
]
{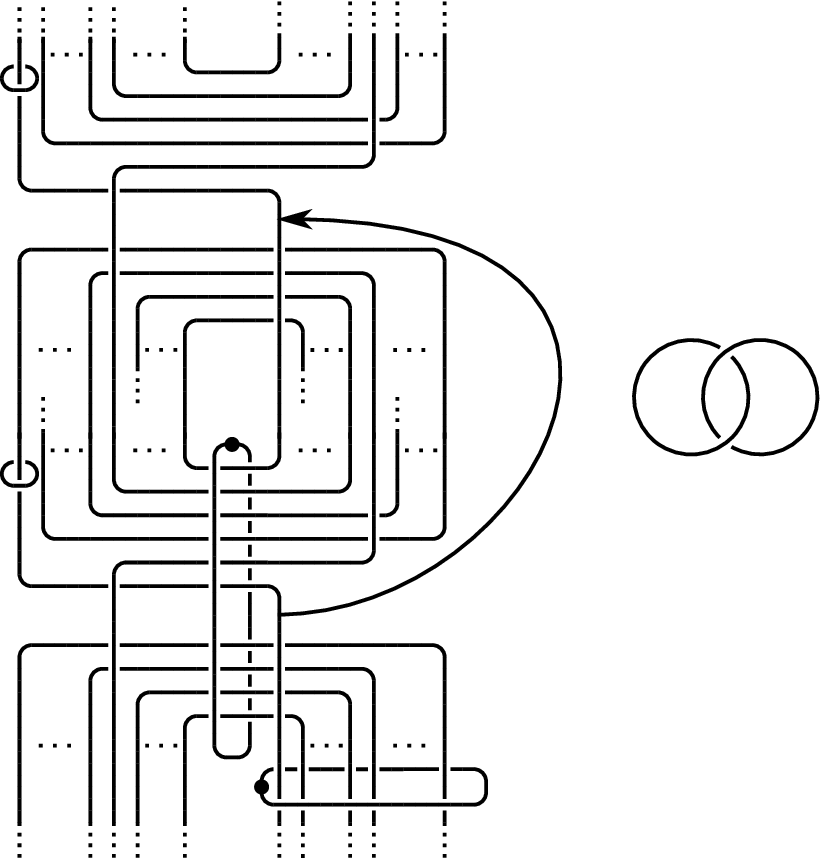}
    \put(-3,87){$0$}

    \put(-1,77){$0$}

    \put(-3,43){$0$}

    \put(-1,33){$0$}

    \put(72,58){$0$}
    \put(95,58){$0$}

    \put(75.5,19){$\cup$ 3-handle}
    \put(79.5,14){4-handle}
    \end{overpic}
\caption{A handle diagram of $\Sigma_2(\tau_{S(T_{2,2n+1})})$ obtained from Figure \ref{fig:double covering3} by several handle slides on 0-framed meridians and canceling the pair of the leftmost 0-framed knot and a 3-handle.}
\label{fig:double covering4}
\end{figure}

\begin{figure}[htbp]
    \centering
\begin{overpic}[scale=0.6
]
{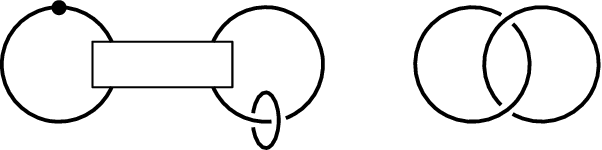}
    \put(19,12){$2n+1$}
    \put(47,0){$0$}
    \put(53,20){$0$}

    \put(66,20){$0$}
    \put(100,20){$0$}

    \put(118,16){$\cup$ 3-handle}
    \put(123.5,10){$4$-handle}
    \end{overpic}
\caption{A handle diagram of $\Sigma_2(\tau_{S(T_{2,2n+1})})$ obtained from Figure \ref{fig:double covering4} by the handle slide indicated in Figure \ref{fig:double covering4}, several handle slides on 0-framed meridians and canceling the pair of the rightmost dotted circle and a framed knot in Figure \ref{fig:double covering4}. This is a handle diagram of $L_{2n+1} \# S^2 \times S^2$.}
\label{fig:double covering5}
\end{figure}

\section{Theorems in terms of pochette surgery}\label{sec:main theorem for pochette}

In this section, we rephrase the results in Sections \ref{sec:property} and \ref{sec:diffeo type} in terms of pochette surgery by using Proposition \ref{prop:pricepochette}.

Let $F(K,p,\varepsilon)$ be a $2$-handlebody described by the handle diagram in Figure \ref{fig:exteriorP}. 
Note that $F(K,2,0)$ is nothing but $F(K\#P_0)$ (see Figure \ref{fig:QuasiExterior}). 
We recall that $e_K: P_{1,1} \to X$ is the embedding that the cord is trivial and the $2$-knot $(S_{1,1})_{e_K}$ in $(P_{1,1})_{e_K}$ is equal to $K$. 

\begin{figure}[htbp]
    \centering
\begin{overpic}[scale=0.6]
{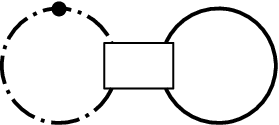}
\put(46,18.5){$p$}
\put(-10,32){$K$}
\put(97,32.5){$\varepsilon$}
\end{overpic}
\caption{A handle diagram of a $2$-handlebody $F(K,p,\varepsilon)$.}
\label{fig:exteriorP}
\end{figure}

\begin{thm}\label{thm:pochettever}
Let $K$ be a ribbon $2$-knot in the $4$-sphere $S^4$. Then, the pochette surgery $S^4(e_K,p/q,\varepsilon)$ is diffeomorphic to the double $DF(K,p,\varepsilon)$ of the $2$-handlebody $F(K,p,\varepsilon)$. 
\end{thm}

\begin{proof}
Using the same arguments as Lemma \ref{lem:predouble} and Theorem \ref{thm:double}, the pochette surgery $S^4(e_K,p,\varepsilon)$ is diffeomorphic to the double $DF(K,p,\varepsilon)$ of the $2$-handlebody $F(K,p,\varepsilon)$. 
Furthermore, by combining the argument of \cite[Subsection 2F]{MR4619857} and the proof of \cite[Proposition 1]{zbMATH07751599} (this argument originates from \cite{murase}), the pochette surgery $S^4(e_K,p/q,\varepsilon)$ is diffeomorphic to $S^4(e_K,p,\varepsilon)$.
Therefore, the pochette surgery $S^4(e_K,p/q,\varepsilon)$ is diffeomorphic to $DF(K,p,\varepsilon)$. 
\end{proof}

\begin{rem}
Let $D(K,p,\varepsilon)$ be a closed $4$-manifold described in Figure \ref{fig:exterior(P)}, where the integer $k$ in Figure \ref{fig:exterior(P)} is the number of the $1$-handles of the handle diagram in Figure \ref{fig:exterior(P)}. 
From the argument in Section \ref{sec:diffeo type} and Theorem \ref{thm:pochettever}, the pochette surgery $S^4(e_K,p/q,\varepsilon)$ is diffeomorphic to $D(K,p,\varepsilon)$. 
\end{rem}

\begin{figure}[htbp]
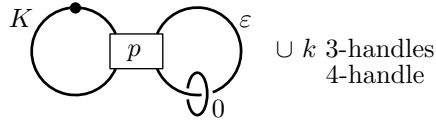

    \centering
\begin{overpic}[scale=0.6]
{pao.eps}
\put(45.5,28.5){$p$}
\put(-10,42){$K$}
\put(98,42){$\varepsilon$}
\put(85,0){$0$}
\put(115,28){$\cup$ $k$ $3$-handles}
\put(138.5,16){$4$-handle}
\end{overpic}
\caption{A handle diagram of $D(K,p,\varepsilon)$.}
\label{fig:exterior(P)}
\end{figure}

\begin{cor}\label{cor:(p/q)-pochette and tau}
Let $K$ be a ribbon $2$-knot in the $4$-sphere $S^4$. Then, the pochette surgery $S^4(e_K,2/(2m+1),0)$ is diffeomorphic to $\tau_K$ for any integer $m$. 
\end{cor}
\begin{proof}
This follows directly from Proposition \ref{prop:pricepochette} and Theorem \ref{thm:pochettever}. 
\end{proof}
Here, we perform a complete classification of the diffeomorphism types of the pochette surgeries for $S^4$ that satisfy $(S_{1,1})_e=O$. 
\begin{prop}\label{prop:the case where core sphere is trivial}
If the $2$-knot $(S_{1,1})_e$ is trivial, then the pochette surgery $S^4(e_O, p/q, \varepsilon)$ is diffeomorphic to the Pao manifold $L(p; \varepsilon, 1; 1)$. 
\end{prop}

\begin{proof}
If the $2$-knot $(S_{1,1})_e$ is the unknotted $2$-knot $O$, then each cord in $E((P_{1,1})_{e_O})$ is isotopic to the trivial cord by the proof of \cite[Theorem 1.5]{MR4619857}. 
Thus, a handle diagram of $S^4$ can be taken as in Figure \ref{fig:2-knot trivial case s4} from Figure \ref{fig: trivial case S4} and the $4$-manifold $(P_{1,1})_{e_O}$ consists of the $2$-handle presented by the leftmost $0$-framed unknot, the $3$-handle, and the $4$-handle in Figure \ref{fig:2-knot trivial case s4}. 
Therefore, a handle diagram of the exterior $E((P_{1,1})_{e_O})$ and the positions of $m_{e_O}$ and $l_{e_O}$ are shown as in Figure \ref{fig:Eml} by \cite[Figure 4]{MR4619857}. 
From Figure \ref{fig:Eml} and the proof of \cite[Proposition 1]{zbMATH07751599}, a handle diagram of the pochette surgery $S^4(e_O,0/1,0)$ is depicted in Figure \ref{fig:DP}. 
Therefore, the pochette surgery $S^4(e_O,0/1,0)$ is diffeomorphic to the double $DP_{1,1}$ of $P_{1,1}$ by Figure \ref{fig:DP}. 
Let $i_{P_{1,1}}: P_{1,1} \hookrightarrow DP_{1,1}$ be the inclusion. 
We note that the pochette $i_{P_{1,1}}(P_{1,1})=P_{1,1}$ consists of the $0$-handle, the $1$-handle presented by the leftmost dotted circle, and the $2$-handle presented by the rightmost $0$-framed unknot in Figure \ref{fig:DP}. 
Note that $m_{e_O}=l=l_{i_{P_{1,1}}}$ and $l_{e_O}=m=m_{i_{P_{1,1}}}$.
We define 
$$g:=g_{i_{P_{1,1}},p/q,\varepsilon} \circ g_{e_O,0/1,0}. $$
Then, we have 
\begin{eqnarray*}
g([m])&=&(g_{i_{P_{1,1}},q/p,\varepsilon})_{*}((g_{e_O,0/1,0})_*([m]))\\
&=&(g_{i_{P_{1,1}},q/p,\varepsilon})_{*}([l_{e_O}]))=(g_{i_{P_{1,1}},q/p,\varepsilon})_{*}([m])\\
&=&(g_{i_{P_{1,1}},q/p,\varepsilon})_{*}([m])=q[m]+p[l]=p[l]+q[m]\\
&=&p[m_{e_O}]+q[l_{e_O}]. 
\end{eqnarray*}
Then, the slope of the homology class $g_*([m])$ in $H_1(\partial P_{1, 1})$ is $p/q$. 
Furthermore, the mod $2$ framing around the knot $g(m)$ is $\varepsilon$. 
Therefore, the pochette surgery $S^4(e_O,p/q,\varepsilon)$ is diffeomorphic to $S^4(e_O,g)$ from Theorem \ref{thm:three conditions}. 
Note that the pochette surgery $S^4(e_O,g)$ is diffeomorphic to $S^4(e_O,0/1,0)(i_{P_{1,1}},q/p,\varepsilon)$. 
From Figure \ref{fig:DP}, a handle diagram of the pochette surgery $DP_{1,1}(i_{P_{1,1}},q/p,\varepsilon)$ is shown in Figure \ref{fig:trivial case p.s.} by \cite{murase} and \cite[Proposition 1]{zbMATH07751599}. 
By comparing Figure \ref{fig:trivial case p.s.} with Figure \ref{fig:Pao}, we see that the pochette surgery $DP_{1,1}(i_{P_{1,1}},q/p,\varepsilon)$ is diffeomorphic to the Pao manifold $L(p; \varepsilon, 1; 1)$. 
This completes the proof.

\begin{figure}
    \centering
\begin{overpic}[scale=0.6]
{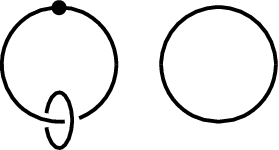}
\put(26.5,0){$0$}

\put(97,41){$0$}

\put(115,32){$\cup$ $3$-handle}
\put(127,19){$4$-handle}
\end{overpic}
\caption{A handle diagram of the $4$-sphere $S^4$.}
\label{fig:2-knot trivial case s4}
\end{figure}

\begin{figure}
    \centering
\begin{overpic}[scale=0.6]
{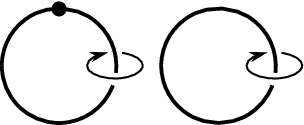}
\put(35.5,2){$m_{e_O}$}

\put(88,35){$0$}
\put(94,2){$l_{e_O}$}

\end{overpic}
\caption{A handle diagram of the exterior $E((P_{1,1})_{e_O})$ and the positions of the meridian $m_{e_O}$ and the longitude $l_{e_O}$.}
\label{fig:Eml}
\end{figure}

\begin{figure}
    \centering
\begin{overpic}[scale=0.6]
{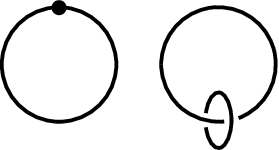}
\put(85,0){$0$}
\put(99,41){$0$}
\put(115,32){$\cup$ $3$-handle}
\put(127,19){$4$-handle}
\end{overpic}
\caption{A handle diagram of the $4$-manifold $S^4(e,0/1,0)=DP_{1,1}$.}
\label{fig:DP}
\end{figure}

\begin{figure}
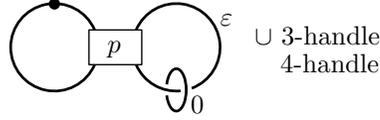

    \centering
\begin{overpic}[scale=0.6]
{pao.eps}
\put(46,28){$p$}
\put(85,0){$0$}
\put(99,41){$\varepsilon$}
\put(115,32){$\cup$ $3$-handle}
\put(127,19){$4$-handle}
\end{overpic}
\caption{A handle diagram of the $4$-manifold $DP_{1,1}(i_{P_{1,1}},q/p,\varepsilon)$.}
\label{fig:trivial case p.s.}
\end{figure}
\end{proof}
Note that Proposition \ref{prop:the case where core sphere is trivial} can be shown directly from the handle diagram of $D(K,p,\varepsilon)$ in Figure \ref{fig:exterior(P)} with $K=O$ and that of the Pao manifold in Figure \ref{fig:Pao}. 

\begin{rem}
From Proposition \ref{prop:the case where core sphere is trivial}, if $(S_{1,1})_e$ is the unknotted 2-knot, $S^4(e,p/q,\varepsilon)$ is diffeomorphic to $S(L(p,q))$ if $p$ is odd or $\varepsilon$ is zero, 
and not homotopy equivalent to $S(L(p,q))$ in the other cases. 
\end{rem}

Corollary \ref{cor:(p/q)-pochette and tau} implies that there exist an infinite homotopy types of pochette surgeries for $S^4$ with slope $2/(2m+1)$ for any integer $m$.
From Corollaries \ref{cor:not spin case}, \ref{cor:tauandPao}, \ref{cor:tauandIwasemfd} and \ref{cor:(p/q)-pochette and tau}, for any integer $m$, the pochette surgery $S^4(e_{S(T_{2,2n+1})},2/(2m+1),0)$ is not homotopy equivalent to the spun 4-manifold $S(M)$ and the twist spun 4-manifold $\widetilde{S}(M)$ for any closed $3$-manifold $M$, any Pao manifold or any Iwase manifold for each $n \neq -1,0$. 

\begin{rem}
From Remark \ref{rem:deformation}, if the slope is $p/q$ ($|p| \ge 3$) or $\varepsilon=1$, the deformation $\alpha$ cannot be applied, so a similar argument cannot be made.
In particular, we highlight the difference in the difficulty of classifying diffeomorphism types for $|p| = 1$, $|p| = 2$ and $|p| \ge 3$.
\end{rem}

Finally, we add a comment on the relationship between the Price twist $\tau_{S(T_{2, 2n+1})}$ and the Iwase manifolds. 
Any Iwase manifold corresponds to a torus surgery on $S^4$ along a torus $T^2$-knot.
In other words, as mentioned in \cite[Section 1]{MR941924}, any Iwase manifold can be interpreted as a $4$-dimensional version of a Dehn surgery on the $3$-sphere $S^3$ along a torus knot. 
In Subsection \ref{subsec:iwasemfd}, any pochette surgery on $S^4$ with mod $2$ framing $0$, which corresponds to a Iwase manifold, is diffeomorphic to the spin or twist-spin of Dehn surgery on $S^3$ along a torus knot. 
On the other hand, the Price twist $\tau_{S(T_{2, 2n+1})}$ is not diffeomorphic to any Iwase manifold from Corollary \ref{cor:tauandIwasemfd}. 
For any $2$-knot $K$, the Price twist $\tau_K$ is a pochette surgery, i.e., a torus surgery from Proposition \ref{prop:pricepochette}. 
So the torus surgery $\tau_{S(T_{2, 2n+1})}$ can be interpreted as a $4$-dimensional version of Dehn surgery on $S^3$ along a non-torus knot (i.e., a hyperbolic knot or a satellite knot) for each $n \neq -1,0$. 

\section*{Acknowledgement}
The authors would like to thank Hisaaki Endo and Yuichi Yamada for their comments on an earlier version of this paper and for inspiring them to work on this updated version. The authors also would like to thank Nobuo Iida, Brendan Owens, Masaaki Suzuki and Jumpei Yasuda for helpful comments. The first author was partially supported by JSPS KAKENHI Grant Number JP23KJ0888 and JP25KJ0301. The second author was partially supported by JST SPRING, Grant Number JPMJSP2106. 

\bibliographystyle{amsalpha}
\bibliography{pricetwist,Gluck}

\newcommand{\etalchar}[1]{$^{#1}$}
\providecommand{\bysame}{\leavevmode\hbox to3em{\hrulefill}\thinspace}
\providecommand{\MR}{\relax\ifhmode\unskip\space\fi MR }
\providecommand{\MRhref}[2]{%
  \href{http://www.ams.org/mathscinet-getitem?mr=#1}{#2}
}
\providecommand{\href}[2]{#2}
\begin{thebibliography}{AAC{\etalchar{+}}24}

\bibitem[AAC{\etalchar{+}}24]{arXiv:2409.12910}
X.~An, M.~Aronin, D.~Cates, A.~Goh, B.~Kirn, J.~Krienke, M.~Liang, S.~Lowery, E.~Malkoc, J.~Meier, M.~Natonson, V.~Radi{\'c}, Y.~Rodoplu, B.~Saha, E.~Scott, R.~Simkins, and A.~Zupan, \emph{Ribbon numbers of {$12$}-crossing knots}, Preprint, {arXiv}:2409.12910 [math.{GT}] (2024), 2024.

\bibitem[AT12]{abetange}
T.~Abe and M.~Tange, \emph{Omae's knot and {$12_{a990}$} are ribbon (intelligence of low-dimensional topology)}, RIMS Kôkyûroku \textbf{1812} (2012), 34--42.

\bibitem[Bel25]{arXiv:2502.10370}
Y.~Belousov, \emph{Explicit formulas for the alexander polynomial of pretzel knots}, Preprint, {arXiv}:2502.10370 [math.{GT}] (2025), 2025.

\bibitem[Cer70]{PMIHES_1970__39__5_0}
J.~Cerf, \emph{La stratification naturelle des espaces de fonctions diff\'erentiables r\'eelles et le th\'eor\`eme de la pseudo-isotopie}, Publications Math\'ematiques de l'IH\'ES \textbf{39} (1970), 5--173 (fr). \MR{292089}

\bibitem[CG86]{MR900252}
A.~J. Casson and C.~McA. Gordon, \emph{Cobordism of classical knots}, \`A{} la recherche de la topologie perdue, Progr. Math., vol.~62, Birkh\"auser Boston, Boston, MA, 1986, With an appendix by P. M. Gilmer, pp.~181--199. \MR{900252}

\bibitem[CM80]{MR562913}
H.~S.~M. Coxeter and W.~O.~J. Moser, \emph{Generators and relations for discrete groups}, fourth ed., Ergebnisse der Mathematik und ihrer Grenzgebiete [Results in Mathematics and Related Areas], vol.~14, Springer-Verlag, Berlin-New York, 1980. \MR{562913}

\bibitem[Fre82]{zbMATH03838948}
M.~H. Freedman, \emph{The topology of four-dimensional manifolds}, J. Differ. Geom. \textbf{17} (1982), 357--453 (English).

\bibitem[Glu62]{zbMATH03180037}
H.~Gluck, \emph{The embedding of two-spheres in the four-sphere}, Trans. Am. Math. Soc. \textbf{104} (1962), 308--333 (English).

\bibitem[Gor76]{key440561m}
C.~McA. Gordon, \emph{Knots in the 4-sphere}, Comment. Math. Helv. \textbf{51} (1976), 585--596, MR:440561. Zbl:0346.55004.

\bibitem[GS23]{gompf20234}
R.~E. Gompf and A.~I. Stipsicz, \emph{{$4$}-manifolds and Kirby calculus}, vol.~20, American Mathematical Society, 2023.

\bibitem[Hay11]{hayano2011genus}
K.~Hayano, \emph{On genus--{$1$} simplified broken lefschetz fibrations}, Algebraic \& Geometric Topology \textbf{11} (2011), no.~3, 1267--1322.

\bibitem[HI24]{zbMATH07899097}
S.~Horigome and K.~Ichihara, \emph{On two-bridge ribbon knots}, J. Knot Theory Ramifications \textbf{33} (2024), no.~5, 12 (English), Id/No 2450012.

\bibitem[HKS99]{MR1701683}
K.~Habiro, T.~Kanenobu, and A.~Shima, \emph{Finite type invariants of ribbon {$2$}-knots}, Low-dimensional topology ({F}unchal, 1998), Contemp. Math., vol. 233, Amer. Math. Soc., Providence, RI, 1999, pp.~187--196. \MR{1701683}

\bibitem[IM04]{iwase20044}
Z.~Iwase and Y.~Matsumoto, \emph{{$4$}-dimensional surgery on a “pochette”}, Proceedings of the East Asian School of Knots, Links and Related Topics, February (2004), 16--20.

\bibitem[Iwa88]{MR941924}
Z.~Iwase, \emph{Dehn-surgery along a torus {$T^2$}-knot}, Pacific J. Math. \textbf{133} (1988), no.~2, 289--299. \MR{941924}

\bibitem[Iwa90]{MR1091159}
\bysame, \emph{Dehn surgery along a torus {$T^2$}-knot. {II}}, Japan. J. Math. (N.S.) \textbf{16} (1990), no.~2, 171--196. \MR{1091159}

\bibitem[JT09]{zbMATH05522298}
W.~B. Jones and W.~J. Thron, \emph{Continued fractions. analytic theory and applications}, reprint of the 1984 hardback ed. ed., Encycl. Math. Appl., vol.~11, Cambridge: Cambridge University Press, 2009 (English).

\bibitem[Kam17]{kamada2017surface}
S.~Kamada, \emph{Surface-knots in {$4$}-space}, Springer, 2017.

\bibitem[Kaw96]{zbMATH00795683}
A.~Kawauchi, \emph{A survey of knot theory. transl. from the japan. by the author}, Basel: Birkh{\"a}user, 1996 (English).

\bibitem[Kin61]{MR133126}
S.~Kinoshita, \emph{On the alexander polynomials of {$2$}-spheres in a {$4$}-sphere}, Ann. of Math. (2) \textbf{74} (1961), 518--531. \MR{133126}

\bibitem[KM97]{zbMATH01117469}
T.~Kanenobu and Y.~Marumoto, \emph{Unknotting and fusion numbers of ribbon {$2$}-knots}, Osaka J. Math. \textbf{34} (1997), no.~3, 525--540 (English).

\bibitem[KM20]{MR4071374}
S.~Kim and M.~Miller, \emph{Trisections of surface complements and the price twist}, Algebr. Geom. Topol. \textbf{20} (2020), no.~1, 343--373. \MR{4071374}

\bibitem[KS20]{MR4160334}
T.~Kanenobu and T.~Sumi, \emph{Twisted alexander polynomial of a ribbon {$2$}-knot of {$1$}-fusion}, Osaka J. Math. \textbf{57} (2020), no.~4, 789--803. \MR{4160334}

\bibitem[KSTI21]{kishimoto2021alexander}
K.~Kishimoto, T.~Shibuya, T.~Tsukamoto, and T.~Ishikawa, \emph{Alexander polynomials of simple-ribbon knots}, Osaka Journal of Mathematics \textbf{58} (2021), no.~1, 41--57.

\bibitem[KSTY99]{MR1721575}
A.~Katanaga, O.~Saeki, M.~Teragaito, and Y.~Yamada, \emph{Gluck surgery along a {$2$}-sphere in a {$4$}-manifold is realized by surgery along a projective plane}, Michigan Math. J. \textbf{46} (1999), no.~3, 555--571. \MR{1721575}

\bibitem[Lam21a]{lamm2021search}
C.~Lamm, \emph{The search for nonsymmetric ribbon knots}, Experimental Mathematics \textbf{30} (2021), no.~3, 349--363.

\bibitem[Lam21b]{MR4394065}
\bysame, \emph{Symmetric union presentations for {$2$}-bridge ribbon knots}, J. Knot Theory Ramifications \textbf{30} (2021), no.~12, Paper No. 2141009, 11. \MR{4394065}

\bibitem[Lis07]{MR2302495}
P.~Lisca, \emph{Lens spaces, rational balls and the ribbon conjecture}, Geom. Topol. \textbf{11} (2007), 429--472. \MR{2302495}

\bibitem[Mar77]{MR467763}
Y.~Marumoto, \emph{On ribbon {$2$}-knots of {$1$}-fusion}, Math. Sem. Notes Kobe Univ. \textbf{5} (1977), no.~1, 59--68. \MR{467763}

\bibitem[Mil21]{zbMATH07379313}
M.~Miller, \emph{Extending fibrations of knot complements to ribbon disk complements}, Geom. Topol. \textbf{25} (2021), no.~3, 1479--1550 (English).

\bibitem[Miz05]{mizuma2005ribbon}
Y.~Mizuma, \emph{Ribbon knots of 1-fusion, the jones polynomial, and the casson-walker invariant.}, Revista Matem{\'a}tica Complutense \textbf{18} (2005), no.~2, 387--425.

\bibitem[Mur15]{murase}
Y.~Murase, \emph{Pochette surgery and kirby diagrams}, Tokyo Institute of Technology (Master thesis, in Japanese), 2015.

\bibitem[Nak81]{MR634000}
Y.~Nakanishi, \emph{A note on unknotting number}, Math. Sem. Notes Kobe Univ. \textbf{9} (1981), no.~1, 99--108. \MR{634000}

\bibitem[NN82]{MR704925}
Y.~Nakanishi and Y.~Nakagawa, \emph{On ribbon knots}, Math. Sem. Notes Kobe Univ. \textbf{10} (1982), no.~2, 423--430. \MR{704925}

\bibitem[NS12]{zbMATH06126098}
D.~Nash and A.~Stipsicz, \emph{Gluck twist on a certain family of {$2$}-knots}, Mich. Math. J. \textbf{61} (2012), no.~4, 703--713 (English).

\bibitem[NS22]{zbMATH07570610}
P.~Naylor and H.~R. Schwartz, \emph{Gluck twisting roll spun knots}, Algebr. Geom. Topol. \textbf{22} (2022), no.~2, 973--990 (English).

\bibitem[Oka20]{okawa}
T.~Okawa, \emph{On pochette surgery on the {$4$}-sphere}, Tokyo Institute of Technology (Master thesis, in Japanese), 2020.

\bibitem[Orl06]{orlik2006seifert}
P.~Orlik, \emph{Seifert manifolds}, vol. 291, Springer, 2006.

\bibitem[OS24]{zbMATH07921118}
B.~Owens and F.~Swenton, \emph{An algorithm to find ribbon disks for alternating knots}, Exp. Math. \textbf{33} (2024), no.~3, 437--455 (English).

\bibitem[Pao77]{MR431231}
P.~S. Pao, \emph{The topological structure of {$4$}-manifolds with effective torus actions. {I}}, Trans. Amer. Math. Soc. \textbf{227} (1977), 279--317. \MR{431231}

\bibitem[Per03]{arXiv:math/0307245}
G.~Perelman, \emph{Finite extinction time for the solutions to the ricci flow on certain three-manifolds}, Preprint, {arXiv}:math/0307245 [math.{DG}] (2003), 2003.

\bibitem[Plo82]{zbMATH03796838}
S.~Plotnick, \emph{Circle actions and fundamental groups for homology {$4$}-spheres}, Trans. Am. Math. Soc. \textbf{273} (1982), 393--404 (English).

\bibitem[Plo86]{plotnick1986equivariant}
S.~P. Plotnick, \emph{Equivariant intersection forms, knots in {$S^4$}, and rotations in {$2$}-spheres}, Transactions of the American Mathematical Society \textbf{296} (1986), no.~2, 543--575.

\bibitem[Pri77]{MR436151}
T.~M. Price, \emph{Homeomorphisms of quaternion space and projective planes in four space}, J. Austral. Math. Soc. Ser. A \textbf{23} (1977), no.~1, 112--128. \MR{436151}

\bibitem[{\c{S}}av24]{csavk2024survey}
O.~{\c{S}}avk, \emph{A survey of the homology cobordism group}, Bulletin of the American Mathematical Society \textbf{61} (2024), no.~1, 119--157.

\bibitem[ST23]{MR4619857}
T.~Suzuki and M.~Tange, \emph{Pochette surgery of {$4$}-sphere}, Pacific J. Math. \textbf{324} (2023), no.~2, 371--398. \MR{4619857}

\bibitem[Suc88]{MR922225}
A.~I. Suciu, \emph{The oriented homotopy type of spun {$3$}-manifolds}, Pacific J. Math. \textbf{131} (1988), no.~2, 393--399. \MR{922225}

\bibitem[Suz23]{zbMATH07751599}
T.~Suzuki, \emph{Constructions of homotopy {$4$}-spheres by pochette surgery}, Geom. Dedicata \textbf{217} (2023), no.~6, 22 (English), Id/No 106.

\bibitem[Vir73]{viro1973local}
O.~J. Viro, \emph{Local knotting of submanifolds}, Mathematics of the USSR-Sbornik \textbf{19} (1973), no.~2, 166.

\bibitem[Yas21]{arXiv:2105.08634}
J.~Yasuda, \emph{A plat form presentation for surface-links}, Preprint, {arXiv}:2105.08634 [math.{GT}] (2021), 2021.

\bibitem[Yas25]{arXiv:2506.15401}
\bysame, \emph{Normal forms of {$2$}-plat {$2$}-knots and their alexander polynomials}, Preprint, {arXiv}:2506.15401 [math.{GT}] (2025), 2025.

\end{thebibliography}

\begin{longtable}{c c c c c c c}
\hline $1$-knot $k$ & $\det(k)$ & $rf(k)$ & $\tau_{R(D(k))}$ & knot diagram $D(k)$ & \\ \hline \hline
\endfirsthead

\hline $1$-knot $k$ & $\det(k)$ & $rf(k)$ & $\tau_{R(D(k))}$ & knot diagram $D(k)$ & \\ \hline \hline 
\endhead

\endlastfoot
   $0_1$ & 1 & 0 & 1 & the circle $S^1$ \\ \hline 
   $6_1$ & 9 & 1 & 3 & \cite[Appendix F]{zbMATH00795683} \\ \hline 
   $3_1 \# 3_1^*$ & 9 & 1 & 3 & Figure \ref{fig:composite cases (ribbon)} \\ \hline 
   $8_8$ & 25 & 1 & 5 & \cite[Appendix F]{zbMATH00795683}\\ \hline 
   $8_9$ & 25 & 1 & 5 & \cite[Appendix F]{zbMATH00795683}\\ \hline 
   $8_{20}$ & 9 & 1 & 3 & \cite[Appendix F]{zbMATH00795683}\\ \hline 
   $4_1 \# 4_1^*$ & 25 & 1 & 5 & Figure \ref{fig:composite cases (ribbon)} \\ \hline 
   $9_{27}$ & 49 & 1 & 7 & \cite[Appendix F]{zbMATH00795683}\\ \hline 
   $9_{41}$ & 49 & 1 & 7 & \cite[Appendix F]{zbMATH00795683}\\ \hline 
   $9_{46}$ & 9 & 1 & 3 & \cite[Appendix F]{zbMATH00795683}\\ \hline 
   $10_3$ & 25 & 1 & 5 & \cite[Appendix F]{zbMATH00795683}\\ \hline 
   $10_{22}$ & 49 & 1 & 7 & \cite[Appendix F]{zbMATH00795683}\\ \hline 
   $10_{35}$ & 49 & 1 & 7 & \cite[Appendix F]{zbMATH00795683}\\ \hline 
   $10_{42}$ & 81 & 1 & 9 & \cite[Appendix F]{zbMATH00795683}\\ \hline 
   $10_{48}$ & 49 & 1 & 7 & \cite[Appendix F]{zbMATH00795683}\\ \hline 
   $10_{75}$ & 81 & 1 & 9 & \cite[Appendix F]{zbMATH00795683}\\ \hline 
   $10_{87}$ & 81 & 1 & 9 & \cite[Appendix F]{zbMATH00795683}\\ \hline 
   $10_{99}$ & 81 & 1 & 9 & \cite[Appendix F]{zbMATH00795683}\\ \hline 
   $10_{123}$ & 121 & 1 & 11 & \cite[Appendix F]{zbMATH00795683}\\ \hline 
   $10_{129}$ & 25 & 1 & 5 & \cite[Appendix F]{zbMATH00795683}\\ \hline 
   $10_{137}$ & 25 & 1 & 5 & \cite[Appendix F]{zbMATH00795683}\\ \hline 
   $10_{140}$ & 9 & 1 & 3 & \cite[Appendix F]{zbMATH00795683}\\ \hline 
   $10_{153}$ & 1 & 1 & 1 & \cite[Appendix F]{zbMATH00795683}\\ \hline 
   $10_{155}$ & 25 & 1 & 5 & \cite[Appendix F]{zbMATH00795683}\\ \hline 
   $5_1 \# 5_1^*$ & 25 & 1 & 5 & Figure \ref{fig:composite cases (ribbon)} \\ \hline 
   $5_2 \# 5_2^*$ & 49 & 1 & 7 & Figure \ref{fig:composite cases (ribbon)} \\ \hline 
   $11a_{28}$ & 121 & 1 & 11 & \cite[Appendix]{lamm2021search}\\ \hline 
   $11a_{35}$ & 121 & 1 & 11 & \cite[Appendix]{lamm2021search}\\ \hline 
   $11a_{36}$ & 121 & 1 & 11 & \cite[Appendix]{lamm2021search}\\ \hline 
   $11a_{58}$ & 81 & 1 & 9 & \cite[Appendix]{lamm2021search}\\ \hline 
   $11a_{87}$ & 121 & 1 & 11 & \cite[Appendix]{lamm2021search}\\ \hline 
   $11a_{96}$ & 121 & 1 & 11 & \cite[Appendix]{lamm2021search}\\ \hline 
   $11a_{103}$ & 81 & 1 & 9 & \cite[Figure 7]{lamm2021search}\\ \hline 
   $11a_{115}$ & 121 & 1 & 11 & \cite[Appendix]{lamm2021search}\\ \hline 
   $11a_{164}$ & 169 & 1 & 13 & \cite[Appendix]{lamm2021search}\\ \hline 
   $11a_{165}$ & 81 & 1 & 9 & \cite[Figure 7]{lamm2021search}\\ \hline 
   $11a_{169}$ & 121 & 1 & 11 & \cite[Appendix]{lamm2021search}\\ \hline 
   $11a_{201}$ & 81 & 1 & 9 & \cite[Figure 7]{lamm2021search}\\ \hline 
   $11a_{316}$ & 121 & 1 & 11 & \cite[Appendix]{lamm2021search}\\ \hline 
   $11a_{326}$ & 169 & 1 & 13 & \cite[Appendix]{lamm2021search}\\ \hline 
   $11n_{4}$ & 49 & 1 & 7 & \cite[Appendix]{lamm2021search}\\ \hline 
   $11n_{21}$ & 49 & 1 & 7 & \cite[Appendix]{lamm2021search}\\ \hline 
   $11n_{37}$ & 25 & 1 & 5 & \cite[Appendix]{lamm2021search}\\ \hline 
   $11n_{39}$ & 25 & 1 & 5 & \cite[Appendix]{lamm2021search}\\ \hline 
   $11n_{42}$ & 1 & 1 & 1 & \cite[Appendix]{lamm2021search}\\ \hline 
   $11n_{49}$ & 1 & 1 & 1 & \cite[Appendix]{lamm2021search}\\ \hline 
   $11n_{50}$ & 25 & 1 & 5 & \cite[Appendix]{lamm2021search}\\ \hline 
   $11n_{67}$ & 9 & 1 & 3 & \cite[Figure 5]{lamm2021search}\\ \hline 
   $11n_{73}$ & 9 & 1 & 3 & \cite[Figure 5]{lamm2021search}\\ \hline 
   $11n_{74}$ & 9 & 1 & 3 & \cite[Figure 5]{lamm2021search}\\ \hline 
   $11n_{83}$ & 49 & 1 & 7 & \cite[Appendix]{lamm2021search}\\ \hline 
   $11n_{97}$ & 9 & 1 & 3 & \cite[Figure 5]{lamm2021search}\\ \hline 
   $11n_{116}$ & 1 & 1 & 1 & \cite[Appendix]{lamm2021search}\\ \hline 
   $11n_{132}$ & 25 & 1 & 5 & \cite[Appendix]{lamm2021search}\\ \hline 
   $11n_{139}$ & 9 & 1 & 3 & \cite[Appendix]{lamm2021search}\\ \hline 
   $11n_{172}$ & 49 & 1 & 7 & \cite[Appendix]{lamm2021search}\\ \hline 
   $3_1 \# 8_{10}$ & 81 & 1 & 9 & \cite[Figure 7]{lamm2021search}\\ \hline 
   $3_1 \# 8_{11}$ & 81 & 1 & 9 & \cite[Figure 7]{lamm2021search}\\ \hline 
   $12a_3$ & 169 & 1 & 13 & \cite[Appendix]{lamm2021search}\\ \hline 
   $12a_{54}$ & 169 & 1 & 13 & \cite[Appendix]{lamm2021search}\\ \hline 
   $12a_{77}$ & 225 & 1 & 15 & \cite[Appendix]{lamm2021search}\\ \hline 
   $12a_{100}$ & 225 & 1 & 15 & \cite[Appendix]{lamm2021search}\\ \hline 
   $12a_{173}$ & 169 & 1 & 13 & \cite[Appendix]{lamm2021search}\\ \hline 
   $12a_{183}$ & 121 & 1 & 11 & \cite[Appendix]{lamm2021search}\\ \hline 
   $12a_{189}$ & 225 & 1 & 15 & \cite[Appendix]{lamm2021search}\\ \hline 
   $12a_{211}$ & 169 & 1 & 13 & \cite[Appendix]{lamm2021search}\\ \hline 
   $12a_{221}$ & 169 & 1 & 13 & \cite[Appendix]{lamm2021search}\\ \hline 
   $12a_{245}$ & 225 & 1 & 15 & \cite[Appendix]{lamm2021search}\\ \hline 
   $12a_{258}$ & 169 & 1 & 13 & \cite[Appendix]{lamm2021search}\\ \hline 
   $12a_{279}$ & 169 & 1 & 13 & \cite[Appendix]{lamm2021search}\\ \hline 
   $12a_{348}$ & 225 & 1 & 15 & ? ($rf(12a_{348})=1$ by \cite{zbMATH07921118})\\ \hline 
   $12a_{377}$ & 225 & 1 & 15 & \cite[Appendix]{lamm2021search}\\ \hline 
   $12a_{425}$ & 81 & 1 & 9 & \cite[Appendix]{lamm2021search}\\ \hline 
   $12a_{427}$ & 225 & 1 & 15 & \cite[Figure 11]{arXiv:2409.12910} \\ \hline 
   $12a_{435}$ & 225 & 1 & 15 & \cite[Appendix]{lamm2021search}\\ \hline 
   $12a_{447}$ & 121 & 1 & 11 & \cite[Appendix]{lamm2021search}\\ \hline 
   $12a_{456}$ & 225 & 1 & 15 & \cite[Appendix]{lamm2021search}\\ \hline 
   $12a_{458}$ & 289 & 1 & 17 & \cite[Appendix]{lamm2021search}\\ \hline 
   $12a_{464}$ & 225 & 1 & 15 & \cite[Appendix]{lamm2021search}\\ \hline 
   $12a_{473}$ & 289 & 1& 17 & \cite[Appendix]{lamm2021search}\\ \hline 
   $12a_{477}$ & 169 & 1 & 13 & \cite[Appendix]{lamm2021search}\\ \hline 
   $12a_{484}$ & 289 & 1 & 17 & \cite[Appendix]{lamm2021search}\\ \hline 
   $12a_{606}$ & 169 & 1 & 13 & \cite[Appendix]{lamm2021search}\\ \hline 
   $12a_{631}$ & 225 & 1, 2 & ? & \cite[Appendix]{lamm2021search}\\ \hline 
   $12a_{646}$ & 169 & 1 & 13 & \cite[Appendix]{lamm2021search}\\ \hline 
   $12a_{667}$ & 121 & 1 & 11 & \cite[Appendix]{lamm2021search}\\ \hline 
   $12a_{715}$ & 169 & 1 & 13 & \cite[Appendix]{lamm2021search}\\ \hline 
   $12a_{786}$ & 169 & 1 & 13 & \cite[Appendix]{lamm2021search}\\ \hline 
   $12a_{819}$ & 169 & 1 & 13 & \cite[Appendix]{lamm2021search}\\ \hline 
   $12a_{879}$ & 121 & 1 & 11 & \cite[Appendix]{lamm2021search}\\ \hline 
   $12a_{887}$ & 289 & 1 & 17 & \cite[Appendix]{lamm2021search}\\ \hline 
   $12a_{975}$ & 225 & 1 & 15 & \cite[Appendix]{lamm2021search}\\ \hline 
   $12a_{979}$ & 225 & 1 & 15 & \cite[Appendix]{lamm2021search}\\ \hline 
   $12a_{990}$ & 225 & 1, 2 & F & \cite[Figure 8]{lamm2021search}\\ \hline 
   $12a_{1011}$ & 121 & 1 & 11 & \cite[Appendix]{lamm2021search}\\ \hline 
   $12a_{1019}$ & 361 & 1 & 19 & \cite[Appendix]{lamm2021search}\\ \hline 
   $12a_{1029}$ & 81 & 1 & 9 & \cite[Appendix]{lamm2021search}\\ \hline 
   $12a_{1034}$ & 121 & 1 & 11 & \cite[Appendix]{lamm2021search}\\ \hline 
   $12a_{1083}$ & 169 & 1 & 13 & \cite[Appendix]{lamm2021search}\\ \hline 
   $12a_{1087}$ & 225 & 1 & 15 & \cite[Appendix]{lamm2021search}\\ \hline 
   $12a_{1105}$ & 289 & 1 & 17 & \cite[Appendix]{lamm2021search}\\ \hline 
   $12a_{1119}$ & 169 & 1 & 13 & \cite[Appendix]{lamm2021search}\\ \hline 
   $12a_{1202}$ & 169 & 1 & 13 & \cite[Appendix]{lamm2021search}\\ \hline 
   $12a_{1225}$ & 225 & 1 & 15 & \cite[Figure 49]{zbMATH07379313}\\ \hline 
   $12a_{1269}$ & 169 & 1 & 13 & \cite[Appendix]{lamm2021search}\\ \hline 
   $12a_{1277}$ & 121 & 1 & 11 & \cite[Appendix]{lamm2021search}\\ \hline 
   $12a_{1283}$ & 81 & 1 & 9 & \cite[Appendix]{lamm2021search}\\ \hline 
   $12n_4$ & 81 & 1 & 9 & \cite[Appendix]{lamm2021search}\\ \hline 
   $12n_{19}$ & 1 & 1 & 1 & \cite[Appendix]{lamm2021search}\\ \hline 
   $12n_{23}$ & 9 & 1 & 3 & \cite[Appendix]{lamm2021search}\\ \hline 
   $12n_{24}$ & 49 & 1 & 7 & \cite[Appendix]{lamm2021search}\\ \hline 
   $12n_{43}$ & 81 & 1 & 9 & \cite[Appendix]{lamm2021search}\\ \hline 
   $12n_{48}$ & 49 & 1 & 7 & \cite[Appendix]{lamm2021search}\\ \hline 
   $12n_{49}$ & 81 & 1 & 9 & \cite[Appendix]{lamm2021search}\\ \hline 
   $12n_{51}$ & 9 & 1 & 3 & \cite[Figure 5]{lamm2021search}\\ \hline 
   $12n_{56}$ & 9 & 1 & 3 & \cite[Figure 5]{lamm2021search}\\ \hline 
   $12n_{57}$ & 9 & 1 & 3 & \cite[Figure 5]{lamm2021search} \\ \hline 
   $12n_{62}$ & 81 & 1 & 9 & \cite[Figure 7]{lamm2021search} \\ \hline 
   $12n_{66}$ & 81 & 1 & 9 & \cite[Figure 7]{lamm2021search} \\ \hline 
   $12n_{87}$ & 49 & 1 & 7 & \cite[Appendix]{lamm2021search}\\ \hline 
   $12n_{106}$ & 81 & 1 & 9 & \cite[Appendix]{lamm2021search}\\ \hline 
   $12n_{145}$ & 25 & 1 & 5 & \cite[Appendix]{lamm2021search}\\ \hline 
   $12n_{170}$ & 81 & 1 & 9 & \cite[Appendix]{lamm2021search}\\ \hline 
   $12n_{214}$ & 1 & 1 & 1 & \cite[Appendix]{lamm2021search}\\ \hline 
   $12n_{256}$ & 25 & 1 & 5 & \cite[Appendix]{lamm2021search}\\ \hline 
   $12n_{257}$ & 25 & 1 & 5 & \cite[Appendix]{lamm2021search}\\ \hline 
   $12n_{268}$ & 9 & 1 & 3 & \cite[Appendix]{lamm2021search}\\ \hline 
   $12n_{279}$ & 25 & 1 & 5 & \cite[Appendix]{lamm2021search}\\ \hline 
   $12n_{288}$ & 49 & 1 & 7 & \cite[Appendix]{lamm2021search}\\ \hline 
   $12n_{309}$ & 1 & 1 & 1 & \cite[Appendix]{lamm2021search}\\ \hline 
   $12n_{312}$ & 49 & 1 & 7 & \cite[Appendix]{lamm2021search}\\ \hline 
   $12n_{313}$ & 1 & 1 & 1 & \cite[Appendix]{lamm2021search}\\ \hline 
   $12n_{318}$ & 1 & 1 & 1 & \cite[Appendix]{lamm2021search}\\ \hline 
   $12n_{360}$ & 49 & 1 & 7 & \cite[Appendix]{lamm2021search}\\ \hline 
   $12n_{380}$ & 81 & 1 & 9 & \cite[Appendix]{lamm2021search}\\ \hline 
   $12n_{393}$ & 49 & 1 & 7 & \cite[Appendix]{lamm2021search}\\ \hline 
   $12n_{394}$ & 25 & 1 & 5 & \cite[Appendix]{lamm2021search}\\ \hline 
   $12n_{397}$ & 49 & 1 & 7 & \cite[Appendix]{lamm2021search}\\ \hline 
   $12n_{399}$ & 81 & 1 & 9 & \cite[Appendix]{lamm2021search}\\ \hline 
   $12n_{414}$ & 25 & 1 & 5 & \cite[Appendix]{lamm2021search}\\ \hline 
   $12n_{420}$ & 81 & 1 & 9 & \cite[Appendix]{lamm2021search}\\ \hline 
   $12n_{430}$ & 1 & 1 & 1 & \cite[Appendix]{lamm2021search}\\ \hline 
   $12n_{440}$ & 81 & 1 & 9 & \cite[Appendix]{lamm2021search}\\ \hline 
   $12n_{462}$ & 25 & 1 & 5 & \cite[Appendix]{lamm2021search}\\ \hline 
   $12n_{501}$ & 49 & 1 & 7 & \cite[Appendix]{lamm2021search}\\ \hline 
   $12n_{504}$ & 121 & 1 & 11 & \cite[Appendix]{lamm2021search}\\ \hline 
   $12n_{553}$ & 81 & 2 & F & \cite[Appendix]{lamm2021search}\\ \hline 
   $12n_{556}$ & 81 & 2 & F & \cite[Appendix]{lamm2021search}\\ \hline 
   $12n_{582}$ & 9 & 1 & 3 & \cite[Appendix]{lamm2021search}\\ \hline 
   $12n_{605}$ & 9 & 1 & 3 & \cite[Appendix]{lamm2021search}\\ \hline 
   $12n_{636}$ & 81 & 1 & 9 & \cite[Appendix]{lamm2021search}\\ \hline 
   $12n_{657}$ & 81 & 1 & 9 & \cite[Appendix]{lamm2021search}\\ \hline 
   $12n_{670}$ & 25 & 1 & 5 & \cite[Appendix]{lamm2021search}\\ \hline 
   $12n_{676}$ & 9 & 1 & 3 & \cite[Appendix]{lamm2021search}\\ \hline 
   $12n_{702}$ & 121 & 1 & 11 & \cite[Appendix]{lamm2021search}\\ \hline 
   $12n_{706}$ & 49 & 1 & 7 & \cite[Appendix]{lamm2021search}\\ \hline 
   $12n_{708}$ & 49 & 1 & 7 & \cite[Appendix]{lamm2021search}\\ \hline 
   $12n_{721}$ & 25 & 1 & 5 & \cite[Appendix]{lamm2021search}\\ \hline 
   $12n_{768}$ & 25 & 1 & 5 & \cite[Appendix]{lamm2021search}\\ \hline 
   $12n_{782}$ & 81 & 1 & 9 & \cite[Appendix]{lamm2021search}\\ \hline 
   $12n_{802}$ & 121 & 1 & 11 & \cite[Appendix]{lamm2021search}\\ \hline 
   $12n_{817}$ & 49 & 1 & 7 & \cite[Appendix]{lamm2021search}\\ \hline 
   $12n_{838}$ & 25 & 1 & 5 & \cite[Appendix]{lamm2021search}\\ \hline 
   $12n_{870}$ & 25 & 1 & 5 & \cite[Appendix]{lamm2021search}\\ \hline 
   $12n_{876}$ & 81 & 1 & 9 & \cite[Appendix]{lamm2021search}\\ \hline 
   $6_1 \# 6_1^*$ & 81 & 1 & 9 & Figure \ref{fig:composite cases (ribbon)} \\ \hline 
   $6_2 \# 6_2^*$ & 121 & 1 & 11 & Figure \ref{fig:composite cases (ribbon)} \\ \hline 
   $6_3 \# 6_3^*$ & 169 & 1 & 13 & Figure \ref{fig:composite cases (ribbon)} \\ \hline 
   $3_1 \# 6_1 \# 3_1^*$ & 81 & 1, 2 & F & Figure \ref{fig:316131} \\ \hline 
   $3_1 \# 3_1 \# 3_1^* \# 3_1^*$ & 81 & 2 & F
& Figure \ref{fig:31313131andtau} \\ \hline 
\caption{Ribbon $1$-knots $k$ up to $12$ crossings and corresponding $\tau_{R(D(k))}$ for knot diagrams $D(k)$. In column $rf(k)$, the fusion number of $k$ is written. In column $\tau_{R(D(k))}$, we write the number $n$ of $\tau_{S(T_{2,n})}$ that is diffeomorphic to $\tau_{R(D(k))}$. The notation F means that $\tau_{R(D(k))}$ with F is not homotopy equivalent to $\tau_{S(T_{2,n})}$ (see Proposition \ref{prop:irregularcase} and Remark \ref{rem:12a990}). 
In column knot diagram $D(k)$, we write a reference that a ribbon presentation used in Example \ref{exm:12 cross}, Remark \ref{rem:1-fusion}, Proposition \ref{prop:irregularcase} and Remark \ref{rem:12a990} is depicted explicitly. We can read the upper bound of the fusion number by using the ribbon presentation.} \label{tab:12 crossing ribbon 1-knot} 
\end{longtable}
\end{document}